\documentclass[a4paper,12pt,leqno]{amsart}
\usepackage{latexsym}
\usepackage[all]{xy}

\usepackage{amssymb} 
\usepackage{amsmath} 
\usepackage{amscd}
\usepackage{color}
\usepackage{comment}
\usepackage{mathtools}
\usepackage{fancybox}
\usepackage{graphicx}

\definecolor{gray}{gray}{0.5}

\numberwithin{equation}{section} 

\newtheorem{theorem}{Theorem}[section]
\newtheorem{lemma}[theorem]{Lemma} 
\newtheorem{corollary}[theorem]{Corollary}
\newtheorem{proposition}[theorem]{Proposition} 
 
\newtheorem{remark}[theorem]{Remark}
\newtheorem{example}[theorem]{Example}
\newtheorem{definition}[theorem]{Definition}

\newtheorem*{theoremA*}{Theorem~A}
\newtheorem*{theoremB*}{Theorem~B}

\allowdisplaybreaks[3]

\def\C{\mathbb C}

\def\Z{\mathbb Z}

\DeclareMathOperator{\Lie}{Lie}

\DeclareMathOperator{\SL}{\textit{SL}}

\newcommand{\Hess}[2]{{\rm Hess}(#1,#2)}
\newcommand{\w}[1]{w_{#1}}
\newcommand{\rb}[1]{k_{#1}}
\newcommand{\lb}[1]{k_{#1,-}}
\newcommand{\rp}[1]{p_{#1}}
\newcommand{\lp}[1]{p_{#1,-}}
\newcommand{\m}[1]{m_{#1}}
\newcommand{\q}[1]{q_{#1}}
\newcommand{\LL}[1]{L_{#1}}
\newcommand{\DD}[1]{D({#1})}
\newcommand{\LB}{L}
\newcommand{\twosteps}{T}
\newcommand{\stable}{S}
\newcommand{\tk}[1]{h(1)+t_{#1}}
\newcommand{\tj}[1]{h(1)+t_{#1}-1}
\newcommand{\jh}{h(h(1))}
\newcommand{\Mh}{M}
\begin{document}
  
\title[Fano and weak Fano Hessenberg varieties]{Fano and weak Fano Hessenberg varieties}
\author {Hiraku Abe}
\address{Faculty of Liberal Arts and Sciences, Osaka Prefecture University
1-1 Gakuen-cho, Naka-ku, Sakai, Osaka 599-8531, Japan}
\email{hirakuabe@globe.ocn.ne.jp}

\author {Naoki Fujita}
\address{Graduate School of Mathematical Sciences, The University of Tokyo, 3-8-1 Komaba, Meguro-ku, Tokyo 153-8914, Japan}
\email{nfujita@ms.u-tokyo.ac.jp}

\author {Haozhi Zeng}
\address{School of Mathematics and Statistics, Huazhong University of Science and Technology, Wuhan, 430074, P.R. China}
\email{zenghaozhi@icloud.com}



\begin{abstract}
Regular semisimple Hessenberg varieties are smooth subvarieties of the flag variety, and their examples contain the flag variety itself and the permutohedral variety which is a toric variety. We give a complete classification of Fano and weak Fano regular semisimple Hessenberg varieties in type A in terms of combinatorics of Hessenberg functions.  In particular, we show that if the anti-canonical bundle of a regular semisimple Hessenberg variety is nef, then it is in fact nef and big. 
\end{abstract}

\maketitle

\setcounter{tocdepth}{1}

%%%%%%%%%%%%%%%%%%%%%%%%%%%%%%%%%%
%%%%%%%%%%%%%%%%%%%%%%%%%%%%%%%%%%
%%%%%%%%%%%%%%%%%%%%%%%%%%%%%%%%%%
\section{Introduction}\label{sec: intro}
%%%%%%%%%%%%%%%%%%%%%%%%%%%%%%%%%%
%%%%%%%%%%%%%%%%%%%%%%%%%%%%%%%%%%
%%%%%%%%%%%%%%%%%%%%%%%%%%%%%%%%%%
Hessenberg varieties in Lie type A$_{n-1}$ are subvarieties of the flag variety of nested linear subspaces of $\C^n$. They were introduced by De Mari-Procesi-Shayman \cite{ma-sh,ma-pr-sh}, and they have been studied from the perspective of geometry, representation theory, and combinatorics. For an $n\times  n$ matrix $X$ and a Hessenberg function $h\colon\{1,2,\ldots,n\}\rightarrow\{1,2,\ldots,n\}$, the Hessenberg variety associated with $X$ and $h$ is given as
\begin{align*}
 \Hess{X}{h} \coloneqq \{ V_{\bullet} \in Fl(\C^n) \mid XV_i\subseteq V_{h(i)} \text{ for all $1\le i\le n$}\},
\end{align*}
where $Fl(\C^n)$ is the flag variety of $\C^n$ consisting of sequences $V_{\bullet}=(V_1\subset V_2\subset \cdots \subset V_n=\C^n)$ of linear subspaces of $\C^n$ such that $\dim_{\C}V_{i}=i$ for $1\le i\le n$.
If $S$ is an $n\times n$ regular semisimple matrix (i.e.\ an $n\times n$ matrix with $n$ distinct eigenvalues), then $\Hess{S}{h}$ is smooth, which is called a regular semisimple Hessenberg variety. There are two extremal examples of regular semisimple Hessenberg varieties: the flag variety itself and the permutohedral variety which is a toric variety associated with the fan consisting of the collection of Weyl chambers of type A$_{n-1}$. The flag variety is a Fano variety (see \cite[Propositions 1.4.1 and 2.2.8 (iv)]{Bri}), whereas the permutohedral variety is not except for very small ranks. However, the permutohedral variety is a weak Fano variety (\cite{Ba-Bl,Huh}). Here, a complex algebraic variety is said to be Fano (resp.\ weak Fano) if its anti-canonical bundle is ample (resp.\ nef and big). In this paper, we give a complete classification of Fano and weak Fano regular semisimple Hessenberg varieties in terms of the combinatorics of the Hessenberg functions.

For each $1\le k\le n-1$, let $h_k:\{1,2,\ldots,n\}\rightarrow\{1,2,\ldots,n\}$ be the Hessenberg function given by $h_k(i)=k+i$ for $1\le i\le n-k$. 
This Hessenberg function is called the ``$k$-banded form", and $\Hess{S}{h_k}$ are the Hessenberg varieties studied in \cite{ma-sh}.
For example, $h_1$ gives the permutohedral variety $\Hess{S}{h_1}$, and $h_{n-1}$ gives the flag variety $\Hess{S}{h_{n-1}}=Fl(\C^n)$. 
The following theorem characterizes when a regular semisimple Hessenberg variety is Fano.

\bigskip

\begin{theoremA*}\label{thm: the first main theorem}
Let $X=\Hess{S}{h}$ be a regular semisimple Hessenberg variety with $h(i)\ge i+1$ for all $1\le i<n$.
Then the following are equivalent:
\begin{itemize}
\item[(i)] the anti-canonical bundle of $X$ is ample $($that is, $X$ is Fano$)$; 
\item[(ii)] $h=h_k$ for some $k$ such that $\frac{n-1}{2}\le k\le n-1$.
\end{itemize}
\end{theoremA*}

The permutohedral variety $\Hess{S}{h_1}$ is not Fano unless $n \le 3$, but it is always weak Fano as we explained above.
The next theorem characterizes when a regular semisimple Hessenberg variety is weak Fano.

\begin{theoremB*}\label{thm: the second main theorem}
Let $X=\Hess{S}{h}$ be a regular semisimple Hessenberg variety with $h(i)\ge i+1$ for all $1\le i<n$.
Then the following are equivalent:
\begin{itemize}
\item[(i)] the anti-canonical bundle of $X$ is nef; 
\item[(ii)] the anti-canonical bundle of $X$ is nef and big $($that is, $X$ is weak Fano$)$; 
\item[(iii)] the inequality
\begin{align*}
h(i)-h(i+1)+2-h^*(n+1-i)+h^*(n-i) \ge 0
\end{align*} 
holds for all $1\le i\le n-1$, where $h^*$ denotes the transpose of $h$.
\end{itemize}
\end{theoremB*}

For example, $\Hess{S}{h}$ for $h=(3,3,4,4)$ is a weak Fano variety since $h=(3,3,4,4)$ satisfies condition (iii) of Theorem~B.
Similarly, we obtain the following.

\begin{corollary}
For $1\le k\le n-1$, the regular semisimple Hessenberg variety ${\rm Hess}(S, h_k)$ with the $k$-banded form $h_k$ is a weak Fano variety.
\end{corollary}

To give the above classifications, we first compute the anti-canonical bundles of regular semisimple Hessenberg varieties explicitly, and we study their volumes by using the theory of line bundles over Richardson varieties. We note that the method of using Richardson varieties for computations of volumes of line bundles over Hessenberg varieties is motivated by Anderson-Tymoczko~\cite{AT} and Harada-Horiguchi-Masuda-Park~\cite{ha-ho-ma-pa}.

Let us see some geometric application of Theorem~B. By the Kawamata-Viehweg vanishing \cite[Theorem~4.3.1]{Laz}, we see that $H^i (Z, L) = 0$, $i > 0$, for a smooth weak Fano variety $Z$ and a nef line bundle $L$ over $Z$. Hence we obtain the following.

\begin{corollary}
Let $X = {\rm Hess}(S, h)$ be a regular semisimple Hessenberg variety with $h(i) \ge i + 1$ for all $1 \le i < n$. If $h$ satisfies condition {\rm (iii)} in Theorem~{\rm B}, then the equality
\[H^i ({\rm Hess}(S, h), L_\mu) = 0\] 
holds for all $i > 0$ and dominant integral weights $\mu$; see Section {\rm 2.2} and Lemma {\rm 3.5} for more details on the line bundle $L_\mu$.
\end{corollary}

Moreover, if $h$ satisfies condition {\rm (iii)} in Theorem~{\rm B}, then ${\rm Hess}(S, h)$ is a smooth Mori dream space since Theorem~{\rm B} implies that it is a smooth log Fano variety (cf.\ \cite{McKernan}). Hence, according to Postinghel-Urbinati \cite[Theorem 4.9]{Postinghel-Urbinati}, such ${\rm Hess}(S, h)$ admits a Newton-Okounkov body with desirable properties. 
In particular, it follows by Anderson \cite{Anderson} that there exists a toric degeneration of ${\rm Hess}(S, h)$ for which we can apply Harada-Kaveh's result \cite[Theorems A and B]{Harada-Kaveh} to ensure the existence of a completely integrable system on ${\rm Hess}(S, h)$. For example, the flag variety $Fl(\C^n)={\rm Hess}(S, h_{n-1})$ and the permutohedral variety ${\rm Hess}(S, h_{1})$ have these properties. It would be interesting to find explicit completely integrable systems on ${\rm Hess}(S, h_{k})$ for $1\le k\le n-1$.

\bigskip
\noindent \textbf{Acknowledgments}.
A part of the research for this paper was carried out at the Fields Institute; the first and second authors would like to thank the institute for its hospitality.
This research is supported in part by 
Osaka City University Advanced Mathematical Institute (MEXT Joint
Usage/Research Center on Mathematics and Theoretical Physics): geometry and topology of torus actions.
The first author is supported in part by JSPS Grant-in-Aid for Early-Career Scientists: 18K13413. The second author is supported by Grant-in-Aid for JSPS Fellows: 19J00123. The third author is supported in part by NSFC: 11901218.

\bigskip
%%%%%%%%%%%%%%%%%%%%%%%%%%%%%%%%%%
\section{Basic definitions and notations}\label{sec: background and notation}
%%%%%%%%%%%%%%%%%%%%%%%%%%%%%%%%%%
In this section, we recall some basic definitions and notations on Hessenberg varieties, which we will use throughout this paper.

%%%%%%%%%%%%%%%%%%%%%%%%%%%%%%%%%%
\subsection{Regular semisimple Hessenberg varieties}
%%%%%%%%%%%%%%%%%%%%%%%%%%%%%%%%%%
Let $n$ be a positive integer, and we denote by $[n]$ the set $\{1,2,\ldots,n\}$.
A function  $h\colon[n]\rightarrow [n]$ is called a \textbf{Hessenberg function} if it satisfies the following conditions:
\begin{itemize}
\item[(i)] $h(1)\le h(2)\le \cdots \le h(n)$, 
\item[(ii)] $h(i)\ge i$ for all $1\le i\le n$.
\end{itemize}
We frequently express this function by listing its values as $h=(h(1),h(2),\ldots,h(n))$.
Also, we may think of it as the boundary path of the configuration of boxes on the square grid of size $n$ which consists of boxes in the $i$-th row and the $j$-th column satisfying $i \leq h(j)$ for $i,j\in[n]$.
For example, if $n=5$ and $h=(3,4,4,5,5)$, then the corresponding boundary path is drawn in Figure $\ref{pic:stair-shape}$. 
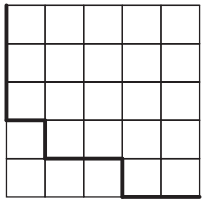
\begin{figure}[h]
\[
%WinTpicVersion4.32a
{\unitlength 0.1in%
\begin{picture}(10.0000,10.0000)(50.0000,-18.0000)%
% BOX 2 0 3 0 Black White  
% 2 5000 800 5200 1000
% 
\special{pn 8}%
\special{pa 5000 800}%
\special{pa 5200 800}%
\special{pa 5200 1000}%
\special{pa 5000 1000}%
\special{pa 5000 800}%
\special{pa 5200 800}%
\special{fp}%
% BOX 2 0 3 0 Black White  
% 2 5200 800 5400 1000
% 
\special{pn 8}%
\special{pa 5200 800}%
\special{pa 5400 800}%
\special{pa 5400 1000}%
\special{pa 5200 1000}%
\special{pa 5200 800}%
\special{pa 5400 800}%
\special{fp}%
% BOX 2 0 3 0 Black White  
% 2 5400 800 5600 1000
% 
\special{pn 8}%
\special{pa 5400 800}%
\special{pa 5600 800}%
\special{pa 5600 1000}%
\special{pa 5400 1000}%
\special{pa 5400 800}%
\special{pa 5600 800}%
\special{fp}%
% BOX 2 0 3 0 Black White  
% 2 5600 800 5800 1000
% 
\special{pn 8}%
\special{pa 5600 800}%
\special{pa 5800 800}%
\special{pa 5800 1000}%
\special{pa 5600 1000}%
\special{pa 5600 800}%
\special{pa 5800 800}%
\special{fp}%
% BOX 2 0 3 0 Black White  
% 2 5800 800 6000 1000
% 
\special{pn 8}%
\special{pa 5800 800}%
\special{pa 6000 800}%
\special{pa 6000 1000}%
\special{pa 5800 1000}%
\special{pa 5800 800}%
\special{pa 6000 800}%
\special{fp}%
% BOX 2 0 3 0 Black White  
% 2 5000 1000 5200 1200
% 
\special{pn 8}%
\special{pa 5000 1000}%
\special{pa 5200 1000}%
\special{pa 5200 1200}%
\special{pa 5000 1200}%
\special{pa 5000 1000}%
\special{pa 5200 1000}%
\special{fp}%
% BOX 2 0 3 0 Black White  
% 2 5200 1000 5400 1200
% 
\special{pn 8}%
\special{pa 5200 1000}%
\special{pa 5400 1000}%
\special{pa 5400 1200}%
\special{pa 5200 1200}%
\special{pa 5200 1000}%
\special{pa 5400 1000}%
\special{fp}%
% BOX 2 0 3 0 Black White  
% 2 5400 1000 5600 1200
% 
\special{pn 8}%
\special{pa 5400 1000}%
\special{pa 5600 1000}%
\special{pa 5600 1200}%
\special{pa 5400 1200}%
\special{pa 5400 1000}%
\special{pa 5600 1000}%
\special{fp}%
% BOX 2 0 3 0 Black White  
% 2 5600 1000 5800 1200
% 
\special{pn 8}%
\special{pa 5600 1000}%
\special{pa 5800 1000}%
\special{pa 5800 1200}%
\special{pa 5600 1200}%
\special{pa 5600 1000}%
\special{pa 5800 1000}%
\special{fp}%
% BOX 2 0 3 0 Black White  
% 2 5800 1000 6000 1200
% 
\special{pn 8}%
\special{pa 5800 1000}%
\special{pa 6000 1000}%
\special{pa 6000 1200}%
\special{pa 5800 1200}%
\special{pa 5800 1000}%
\special{pa 6000 1000}%
\special{fp}%
% BOX 2 0 3 0 Black White  
% 2 5000 1200 5200 1400
% 
\special{pn 8}%
\special{pa 5000 1200}%
\special{pa 5200 1200}%
\special{pa 5200 1400}%
\special{pa 5000 1400}%
\special{pa 5000 1200}%
\special{pa 5200 1200}%
\special{fp}%
% BOX 2 0 3 0 Black White  
% 2 5200 1200 5400 1400
% 
\special{pn 8}%
\special{pa 5200 1200}%
\special{pa 5400 1200}%
\special{pa 5400 1400}%
\special{pa 5200 1400}%
\special{pa 5200 1200}%
\special{pa 5400 1200}%
\special{fp}%
% BOX 2 0 3 0 Black White  
% 2 5400 1200 5600 1400
% 
\special{pn 8}%
\special{pa 5400 1200}%
\special{pa 5600 1200}%
\special{pa 5600 1400}%
\special{pa 5400 1400}%
\special{pa 5400 1200}%
\special{pa 5600 1200}%
\special{fp}%
% BOX 2 0 3 0 Black White  
% 2 5600 1200 5800 1400
% 
\special{pn 8}%
\special{pa 5600 1200}%
\special{pa 5800 1200}%
\special{pa 5800 1400}%
\special{pa 5600 1400}%
\special{pa 5600 1200}%
\special{pa 5800 1200}%
\special{fp}%
% BOX 2 0 3 0 Black White  
% 2 5800 1200 6000 1400
% 
\special{pn 8}%
\special{pa 5800 1200}%
\special{pa 6000 1200}%
\special{pa 6000 1400}%
\special{pa 5800 1400}%
\special{pa 5800 1200}%
\special{pa 6000 1200}%
\special{fp}%
% BOX 2 0 3 0 Black White  
% 2 5000 1400 5200 1600
% 
\special{pn 8}%
\special{pa 5000 1400}%
\special{pa 5200 1400}%
\special{pa 5200 1600}%
\special{pa 5000 1600}%
\special{pa 5000 1400}%
\special{pa 5200 1400}%
\special{fp}%
% BOX 2 0 3 0 Black White  
% 2 5200 1400 5400 1600
% 
\special{pn 8}%
\special{pa 5200 1400}%
\special{pa 5400 1400}%
\special{pa 5400 1600}%
\special{pa 5200 1600}%
\special{pa 5200 1400}%
\special{pa 5400 1400}%
\special{fp}%
% BOX 2 0 3 0 Black White  
% 2 5400 1400 5600 1600
% 
\special{pn 8}%
\special{pa 5400 1400}%
\special{pa 5600 1400}%
\special{pa 5600 1600}%
\special{pa 5400 1600}%
\special{pa 5400 1400}%
\special{pa 5600 1400}%
\special{fp}%
% BOX 2 0 3 0 Black White  
% 2 5600 1400 5800 1600
% 
\special{pn 8}%
\special{pa 5600 1400}%
\special{pa 5800 1400}%
\special{pa 5800 1600}%
\special{pa 5600 1600}%
\special{pa 5600 1400}%
\special{pa 5800 1400}%
\special{fp}%
% BOX 2 0 3 0 Black White  
% 2 5800 1400 6000 1600
% 
\special{pn 8}%
\special{pa 5800 1400}%
\special{pa 6000 1400}%
\special{pa 6000 1600}%
\special{pa 5800 1600}%
\special{pa 5800 1400}%
\special{pa 6000 1400}%
\special{fp}%
% BOX 2 0 3 0 Black White  
% 2 5000 1600 5200 1800
% 
\special{pn 8}%
\special{pa 5000 1600}%
\special{pa 5200 1600}%
\special{pa 5200 1800}%
\special{pa 5000 1800}%
\special{pa 5000 1600}%
\special{pa 5200 1600}%
\special{fp}%
% BOX 2 0 3 0 Black White  
% 2 5200 1600 5400 1800
% 
\special{pn 8}%
\special{pa 5200 1600}%
\special{pa 5400 1600}%
\special{pa 5400 1800}%
\special{pa 5200 1800}%
\special{pa 5200 1600}%
\special{pa 5400 1600}%
\special{fp}%
% BOX 2 0 3 0 Black White  
% 2 5400 1600 5600 1800
% 
\special{pn 8}%
\special{pa 5400 1600}%
\special{pa 5600 1600}%
\special{pa 5600 1800}%
\special{pa 5400 1800}%
\special{pa 5400 1600}%
\special{pa 5600 1600}%
\special{fp}%
% BOX 2 0 3 0 Black White  
% 2 5600 1600 5800 1800
% 
\special{pn 8}%
\special{pa 5600 1600}%
\special{pa 5800 1600}%
\special{pa 5800 1800}%
\special{pa 5600 1800}%
\special{pa 5600 1600}%
\special{pa 5800 1600}%
\special{fp}%
% BOX 2 0 3 0 Black White  
% 2 5800 1600 6000 1800
% 
\special{pn 8}%
\special{pa 5800 1600}%
\special{pa 6000 1600}%
\special{pa 6000 1800}%
\special{pa 5800 1800}%
\special{pa 5800 1600}%
\special{pa 6000 1600}%
\special{fp}%
% LINE 0 0 3 0 Black White  
% 12 5000 800 5000 1400 5000 1400 5200 1400 5200 1400 5200 1600 5200 1600 5600 1600 5600 1600 5600 1800 5600 1800 6000 1800
% 
\special{pn 20}%
\special{pa 5000 800}%
\special{pa 5000 1400}%
\special{fp}%
\special{pa 5000 1400}%
\special{pa 5200 1400}%
\special{fp}%
\special{pa 5200 1400}%
\special{pa 5200 1600}%
\special{fp}%
\special{pa 5200 1600}%
\special{pa 5600 1600}%
\special{fp}%
\special{pa 5600 1600}%
\special{pa 5600 1800}%
\special{fp}%
\special{pa 5600 1800}%
\special{pa 6000 1800}%
\special{fp}%
\end{picture}}%
\]
\caption{The boundary path corresponding to $h=(3,4,4,5,5)$.}
\label{pic:stair-shape}
\end{figure}

For an $n\times  n$ matrix $X$ and a Hessenberg function $h\colon[n]\rightarrow[n]$, the \textbf{Hessenberg variety} associated with $X$ and $h$ is defined to be
\begin{align*}
 \Hess{X}{h} \coloneqq \{ V_{\bullet} \in Fl(\C^n) \mid XV_i\subseteq V_{h(i)} \text{ for all $1\le i\le n$}\},
\end{align*}
where $Fl(\C^n)$ is the flag variety of $\C^n$ consisting of sequences $V_{\bullet}=(V_1\subset V_2\subset \cdots \subset V_n=\C^n)$ of linear subspaces of $\C^n$ such that $\dim_{\C}V_{i}=i$ for $1\le i\le n$.
Let $S$ be a complex $n\times n$ regular semisimple matrix (i.e.\ a complex $n\times n$ matrix with $n$ distinct eigenvalues). Then $\Hess{S}{h}$ is called a \textbf{regular semisimple Hessenberg variety}. 
It is known that $\Hess{S}{h}$ is a smooth projective variety (\cite{ma-sh,ma-pr-sh}). 

In this paper, we always assume that
\begin{align}\label{eq: assumption on h}
 h(i)\ge i+1 \qquad (1\le i< n)
\end{align}
so that the corresponding regular semisimple Hessenberg variety $\Hess{S}{h}$ is irreducible. In addition, since we have $\Hess{S}{h}\cong \Hess{g^{-1}Sg}{h}$ for $g\in \SL_n(\C)$, we assume that $S$ is a diagonal matrix.

\begin{remark}\label{rem: connected}
\normalfont{If we have $n\ge2$ and $h(j)=j$ for some $j<n$, then $\Hess{S}{h}$ is not connected, and each connected component is isomorphic to a product of regular semisimple Hessenberg varieties of smaller ranks. See \cite{Teff11} for details.}
\end{remark}

%%%%%%%%%%%%%%%%%%%%%%%%%%%%%%%%%%
\subsection{Line bundles over flag varieties}\label{subsec: line bundles on flag}
%%%%%%%%%%%%%%%%%%%%%%%%%%%%%%%%%%

Let $G=\SL_n (\C)$ be the complex special linear group of degree $n$. Let $B\subseteq G$ be the Borel subgroup consisting of the upper-triangular matrices, and $T\subseteq B$ the maximal torus of $B$ consisting of the diagonal matrices. 
We may identify the flag variety $Fl(\C^n)$ with $G/B$ by sending $gB\in G/B$ to the flag $V_{\bullet}\in Fl(\C^n)$ given by $V_i = \sum_{j=1}^i \C g_j$, where $g_j$ is the $j$-th column vector of $g$.
Let $\mu \colon T\rightarrow \C^{\times}$ be a weight of $T$. By composing this with the canonical projection $B \twoheadrightarrow T$, we obtain a homomorphism $\mu \colon B\rightarrow \C^{\times}$ which we also denote by $\mu$. Let $\C_{\mu}=\C$ be the $1$-dimensional representation of $B$ given by $b\cdot z=\mu(b)z$ for $b\in B$ and $z\in\C$. We denote by $\C_{\mu}^*$ its dual representation. Since the quotient map $p \colon G\rightarrow G/B$ is a principal $B$-bundle, we obtain the associated line bundle over $Fl(\C^n)=G/B$: 
\begin{align*}
 \LB_{\mu} \coloneqq G\times^B \C_{\mu}^*.
\end{align*}
Namely, it is the quotient of the product $G\times \C$ by the right $B$-action given by $(g,z)\cdot b=(gb,\mu^{-1}(b^{-1})z)=(gb,\mu(b)z)$ for $b\in B$ and $(g,z)\in G\times \C$. For a subvariety $Z \subseteq G/B$, we will also denote by $\LB_{\mu}$ the restriction of $\LB_{\mu}$ to $Z$ by abusing notation.

For $1\le i\le n$, let $x_i$ be the weight of $T$ which sends $t=\text{diag}(t_1,\ldots,t_n)\in T$ to $t_i \in \C^{\times}$. 
We identify weights of $T$ as induced homomorphisms $\Lie(T)\rightarrow \Lie(\C^{\times})=\C$  
to use the additive notation, e.g.\ $x_1+x_2+\cdots+x_n=0$.
The standard $i$-th fundamental weight of $T$ is given by $\varpi_i=x_1+x_2+\cdots+x_i$ for $1\le i\le n-1$. Also, for $1\le i<j\le n$, let $\alpha_{i,j}=x_i-x_j$ be the standard $(i,j)$-th positive root.
A weight $\mu$ of $T$ is called a dominant integral weight if we can write $\mu=\sum_{i=1}^{n-1} a_i\varpi_i$ with $a_i\geq0$ for all $1\le i\le n-1$. We denote by $P_+$ the semigroup of the dominant integral weights.

\bigskip
%%%%%%%%%%%%%%%%%%%%%%%%%%%%%%%%%%
%%%%%%%%%%%%%%%%%%%%%%%%%%%%%%%%%%
%%%%%%%%%%%%%%%%%%%%%%%%%%%%%%%%%%
\section{Fano Hessenberg varieties}\label{sec: Fano Hessenberg}
%%%%%%%%%%%%%%%%%%%%%%%%%%%%%%%%%%
%%%%%%%%%%%%%%%%%%%%%%%%%%%%%%%%%%
%%%%%%%%%%%%%%%%%%%%%%%%%%%%%%%%%%
In this section, we describe the anti-canonical bundle of $\Hess{S}{h}$ in  terms of a line bundle over the flag variety $Fl(\C^n)$, and give a proof of Theorem~A which is stated in Section~\ref{sec: intro}.

%%%%%%%%%%%%%%%%%%%%%%%%%%%%%%%%%%
\subsection{The anti-canonical bundles of Hessenberg varieties}\label{subsec: anti-canonical bundles}
%%%%%%%%%%%%%%%%%%%%%%%%%%%%%%%%%%
For a Hessenberg function $h\colon[n]\rightarrow[n]$, let $h^*\colon[n]\rightarrow[n]$ be the transpose of $h$, that is,
\begin{align*}
h^*(i)
&\coloneqq |\{ k\in[n] \mid n+1-i \le h(k) \}|. 
\end{align*}
For example, if $n=5$ and $h=(3,4,4,5,5)$, then we have $h^*=(2,4,5,5,5)$. See Figure \ref{pic:stair-transpose}.
\begin{figure}[h]
\[
%WinTpicVersion4.32a
{\unitlength 0.1in%
\begin{picture}(32.0000,11.0000)(18.0000,-19.0000)%
% BOX 2 0 3 0 Black White  
% 2 1800 800 2000 1000
% 
\special{pn 8}%
\special{pa 1800 800}%
\special{pa 2000 800}%
\special{pa 2000 1000}%
\special{pa 1800 1000}%
\special{pa 1800 800}%
\special{pa 2000 800}%
\special{fp}%
% BOX 2 0 3 0 Black White  
% 2 2000 800 2200 1000
% 
\special{pn 8}%
\special{pa 2000 800}%
\special{pa 2200 800}%
\special{pa 2200 1000}%
\special{pa 2000 1000}%
\special{pa 2000 800}%
\special{pa 2200 800}%
\special{fp}%
% BOX 2 0 3 0 Black White  
% 2 2200 800 2400 1000
% 
\special{pn 8}%
\special{pa 2200 800}%
\special{pa 2400 800}%
\special{pa 2400 1000}%
\special{pa 2200 1000}%
\special{pa 2200 800}%
\special{pa 2400 800}%
\special{fp}%
% BOX 2 0 3 0 Black White  
% 2 2400 800 2600 1000
% 
\special{pn 8}%
\special{pa 2400 800}%
\special{pa 2600 800}%
\special{pa 2600 1000}%
\special{pa 2400 1000}%
\special{pa 2400 800}%
\special{pa 2600 800}%
\special{fp}%
% BOX 2 0 3 0 Black White  
% 2 2600 800 2800 1000
% 
\special{pn 8}%
\special{pa 2600 800}%
\special{pa 2800 800}%
\special{pa 2800 1000}%
\special{pa 2600 1000}%
\special{pa 2600 800}%
\special{pa 2800 800}%
\special{fp}%
% BOX 2 0 3 0 Black White  
% 2 1800 1000 2000 1200
% 
\special{pn 8}%
\special{pa 1800 1000}%
\special{pa 2000 1000}%
\special{pa 2000 1200}%
\special{pa 1800 1200}%
\special{pa 1800 1000}%
\special{pa 2000 1000}%
\special{fp}%
% BOX 2 0 3 0 Black White  
% 2 2000 1000 2200 1200
% 
\special{pn 8}%
\special{pa 2000 1000}%
\special{pa 2200 1000}%
\special{pa 2200 1200}%
\special{pa 2000 1200}%
\special{pa 2000 1000}%
\special{pa 2200 1000}%
\special{fp}%
% BOX 2 0 3 0 Black White  
% 2 2200 1000 2400 1200
% 
\special{pn 8}%
\special{pa 2200 1000}%
\special{pa 2400 1000}%
\special{pa 2400 1200}%
\special{pa 2200 1200}%
\special{pa 2200 1000}%
\special{pa 2400 1000}%
\special{fp}%
% BOX 2 0 3 0 Black White  
% 2 2400 1000 2600 1200
% 
\special{pn 8}%
\special{pa 2400 1000}%
\special{pa 2600 1000}%
\special{pa 2600 1200}%
\special{pa 2400 1200}%
\special{pa 2400 1000}%
\special{pa 2600 1000}%
\special{fp}%
% BOX 2 0 3 0 Black White  
% 2 2600 1000 2800 1200
% 
\special{pn 8}%
\special{pa 2600 1000}%
\special{pa 2800 1000}%
\special{pa 2800 1200}%
\special{pa 2600 1200}%
\special{pa 2600 1000}%
\special{pa 2800 1000}%
\special{fp}%
% BOX 2 0 3 0 Black White  
% 2 1800 1200 2000 1400
% 
\special{pn 8}%
\special{pa 1800 1200}%
\special{pa 2000 1200}%
\special{pa 2000 1400}%
\special{pa 1800 1400}%
\special{pa 1800 1200}%
\special{pa 2000 1200}%
\special{fp}%
% BOX 2 0 3 0 Black White  
% 2 2000 1200 2200 1400
% 
\special{pn 8}%
\special{pa 2000 1200}%
\special{pa 2200 1200}%
\special{pa 2200 1400}%
\special{pa 2000 1400}%
\special{pa 2000 1200}%
\special{pa 2200 1200}%
\special{fp}%
% BOX 2 0 3 0 Black White  
% 2 2200 1200 2400 1400
% 
\special{pn 8}%
\special{pa 2200 1200}%
\special{pa 2400 1200}%
\special{pa 2400 1400}%
\special{pa 2200 1400}%
\special{pa 2200 1200}%
\special{pa 2400 1200}%
\special{fp}%
% BOX 2 0 3 0 Black White  
% 2 2400 1200 2600 1400
% 
\special{pn 8}%
\special{pa 2400 1200}%
\special{pa 2600 1200}%
\special{pa 2600 1400}%
\special{pa 2400 1400}%
\special{pa 2400 1200}%
\special{pa 2600 1200}%
\special{fp}%
% BOX 2 0 3 0 Black White  
% 2 2600 1200 2800 1400
% 
\special{pn 8}%
\special{pa 2600 1200}%
\special{pa 2800 1200}%
\special{pa 2800 1400}%
\special{pa 2600 1400}%
\special{pa 2600 1200}%
\special{pa 2800 1200}%
\special{fp}%
% BOX 2 0 3 0 Black White  
% 2 1800 1400 2000 1600
% 
\special{pn 8}%
\special{pa 1800 1400}%
\special{pa 2000 1400}%
\special{pa 2000 1600}%
\special{pa 1800 1600}%
\special{pa 1800 1400}%
\special{pa 2000 1400}%
\special{fp}%
% BOX 2 0 3 0 Black White  
% 2 2000 1400 2200 1600
% 
\special{pn 8}%
\special{pa 2000 1400}%
\special{pa 2200 1400}%
\special{pa 2200 1600}%
\special{pa 2000 1600}%
\special{pa 2000 1400}%
\special{pa 2200 1400}%
\special{fp}%
% BOX 2 0 3 0 Black White  
% 2 2200 1400 2400 1600
% 
\special{pn 8}%
\special{pa 2200 1400}%
\special{pa 2400 1400}%
\special{pa 2400 1600}%
\special{pa 2200 1600}%
\special{pa 2200 1400}%
\special{pa 2400 1400}%
\special{fp}%
% BOX 2 0 3 0 Black White  
% 2 2400 1400 2600 1600
% 
\special{pn 8}%
\special{pa 2400 1400}%
\special{pa 2600 1400}%
\special{pa 2600 1600}%
\special{pa 2400 1600}%
\special{pa 2400 1400}%
\special{pa 2600 1400}%
\special{fp}%
% BOX 2 0 3 0 Black White  
% 2 2600 1400 2800 1600
% 
\special{pn 8}%
\special{pa 2600 1400}%
\special{pa 2800 1400}%
\special{pa 2800 1600}%
\special{pa 2600 1600}%
\special{pa 2600 1400}%
\special{pa 2800 1400}%
\special{fp}%
% BOX 2 0 3 0 Black White  
% 2 1800 1600 2000 1800
% 
\special{pn 8}%
\special{pa 1800 1600}%
\special{pa 2000 1600}%
\special{pa 2000 1800}%
\special{pa 1800 1800}%
\special{pa 1800 1600}%
\special{pa 2000 1600}%
\special{fp}%
% BOX 2 0 3 0 Black White  
% 2 2000 1600 2200 1800
% 
\special{pn 8}%
\special{pa 2000 1600}%
\special{pa 2200 1600}%
\special{pa 2200 1800}%
\special{pa 2000 1800}%
\special{pa 2000 1600}%
\special{pa 2200 1600}%
\special{fp}%
% BOX 2 0 3 0 Black White  
% 2 2200 1600 2400 1800
% 
\special{pn 8}%
\special{pa 2200 1600}%
\special{pa 2400 1600}%
\special{pa 2400 1800}%
\special{pa 2200 1800}%
\special{pa 2200 1600}%
\special{pa 2400 1600}%
\special{fp}%
% BOX 2 0 3 0 Black White  
% 2 2400 1600 2600 1800
% 
\special{pn 8}%
\special{pa 2400 1600}%
\special{pa 2600 1600}%
\special{pa 2600 1800}%
\special{pa 2400 1800}%
\special{pa 2400 1600}%
\special{pa 2600 1600}%
\special{fp}%
% BOX 2 0 3 0 Black White  
% 2 2600 1600 2800 1800
% 
\special{pn 8}%
\special{pa 2600 1600}%
\special{pa 2800 1600}%
\special{pa 2800 1800}%
\special{pa 2600 1800}%
\special{pa 2600 1600}%
\special{pa 2800 1600}%
\special{fp}%
% LINE 0 0 3 0 Black White  
% 12 1800 800 1800 1400 1800 1400 2000 1400 2000 1400 2000 1600 2000 1600 2400 1600 2400 1600 2400 1800 2400 1800 2800 1800
% 
\special{pn 20}%
\special{pa 1800 800}%
\special{pa 1800 1400}%
\special{fp}%
\special{pa 1800 1400}%
\special{pa 2000 1400}%
\special{fp}%
\special{pa 2000 1400}%
\special{pa 2000 1600}%
\special{fp}%
\special{pa 2000 1600}%
\special{pa 2400 1600}%
\special{fp}%
\special{pa 2400 1600}%
\special{pa 2400 1800}%
\special{fp}%
\special{pa 2400 1800}%
\special{pa 2800 1800}%
\special{fp}%
% BOX 2 0 3 0 Black White  
% 2 4000 800 4200 1000
% 
\special{pn 8}%
\special{pa 4000 800}%
\special{pa 4200 800}%
\special{pa 4200 1000}%
\special{pa 4000 1000}%
\special{pa 4000 800}%
\special{pa 4200 800}%
\special{fp}%
% BOX 2 0 3 0 Black White  
% 2 4200 800 4400 1000
% 
\special{pn 8}%
\special{pa 4200 800}%
\special{pa 4400 800}%
\special{pa 4400 1000}%
\special{pa 4200 1000}%
\special{pa 4200 800}%
\special{pa 4400 800}%
\special{fp}%
% BOX 2 0 3 0 Black White  
% 2 4400 800 4600 1000
% 
\special{pn 8}%
\special{pa 4400 800}%
\special{pa 4600 800}%
\special{pa 4600 1000}%
\special{pa 4400 1000}%
\special{pa 4400 800}%
\special{pa 4600 800}%
\special{fp}%
% BOX 2 0 3 0 Black White  
% 2 4600 800 4800 1000
% 
\special{pn 8}%
\special{pa 4600 800}%
\special{pa 4800 800}%
\special{pa 4800 1000}%
\special{pa 4600 1000}%
\special{pa 4600 800}%
\special{pa 4800 800}%
\special{fp}%
% BOX 2 0 3 0 Black White  
% 2 4800 800 5000 1000
% 
\special{pn 8}%
\special{pa 4800 800}%
\special{pa 5000 800}%
\special{pa 5000 1000}%
\special{pa 4800 1000}%
\special{pa 4800 800}%
\special{pa 5000 800}%
\special{fp}%
% BOX 2 0 3 0 Black White  
% 2 4000 1000 4200 1200
% 
\special{pn 8}%
\special{pa 4000 1000}%
\special{pa 4200 1000}%
\special{pa 4200 1200}%
\special{pa 4000 1200}%
\special{pa 4000 1000}%
\special{pa 4200 1000}%
\special{fp}%
% BOX 2 0 3 0 Black White  
% 2 4200 1000 4400 1200
% 
\special{pn 8}%
\special{pa 4200 1000}%
\special{pa 4400 1000}%
\special{pa 4400 1200}%
\special{pa 4200 1200}%
\special{pa 4200 1000}%
\special{pa 4400 1000}%
\special{fp}%
% BOX 2 0 3 0 Black White  
% 2 4400 1000 4600 1200
% 
\special{pn 8}%
\special{pa 4400 1000}%
\special{pa 4600 1000}%
\special{pa 4600 1200}%
\special{pa 4400 1200}%
\special{pa 4400 1000}%
\special{pa 4600 1000}%
\special{fp}%
% BOX 2 0 3 0 Black White  
% 2 4600 1000 4800 1200
% 
\special{pn 8}%
\special{pa 4600 1000}%
\special{pa 4800 1000}%
\special{pa 4800 1200}%
\special{pa 4600 1200}%
\special{pa 4600 1000}%
\special{pa 4800 1000}%
\special{fp}%
% BOX 2 0 3 0 Black White  
% 2 4800 1000 5000 1200
% 
\special{pn 8}%
\special{pa 4800 1000}%
\special{pa 5000 1000}%
\special{pa 5000 1200}%
\special{pa 4800 1200}%
\special{pa 4800 1000}%
\special{pa 5000 1000}%
\special{fp}%
% BOX 2 0 3 0 Black White  
% 2 4000 1200 4200 1400
% 
\special{pn 8}%
\special{pa 4000 1200}%
\special{pa 4200 1200}%
\special{pa 4200 1400}%
\special{pa 4000 1400}%
\special{pa 4000 1200}%
\special{pa 4200 1200}%
\special{fp}%
% BOX 2 0 3 0 Black White  
% 2 4200 1200 4400 1400
% 
\special{pn 8}%
\special{pa 4200 1200}%
\special{pa 4400 1200}%
\special{pa 4400 1400}%
\special{pa 4200 1400}%
\special{pa 4200 1200}%
\special{pa 4400 1200}%
\special{fp}%
% BOX 2 0 3 0 Black White  
% 2 4400 1200 4600 1400
% 
\special{pn 8}%
\special{pa 4400 1200}%
\special{pa 4600 1200}%
\special{pa 4600 1400}%
\special{pa 4400 1400}%
\special{pa 4400 1200}%
\special{pa 4600 1200}%
\special{fp}%
% BOX 2 0 3 0 Black White  
% 2 4600 1200 4800 1400
% 
\special{pn 8}%
\special{pa 4600 1200}%
\special{pa 4800 1200}%
\special{pa 4800 1400}%
\special{pa 4600 1400}%
\special{pa 4600 1200}%
\special{pa 4800 1200}%
\special{fp}%
% BOX 2 0 3 0 Black White  
% 2 4800 1200 5000 1400
% 
\special{pn 8}%
\special{pa 4800 1200}%
\special{pa 5000 1200}%
\special{pa 5000 1400}%
\special{pa 4800 1400}%
\special{pa 4800 1200}%
\special{pa 5000 1200}%
\special{fp}%
% BOX 2 0 3 0 Black White  
% 2 4000 1400 4200 1600
% 
\special{pn 8}%
\special{pa 4000 1400}%
\special{pa 4200 1400}%
\special{pa 4200 1600}%
\special{pa 4000 1600}%
\special{pa 4000 1400}%
\special{pa 4200 1400}%
\special{fp}%
% BOX 2 0 3 0 Black White  
% 2 4200 1400 4400 1600
% 
\special{pn 8}%
\special{pa 4200 1400}%
\special{pa 4400 1400}%
\special{pa 4400 1600}%
\special{pa 4200 1600}%
\special{pa 4200 1400}%
\special{pa 4400 1400}%
\special{fp}%
% BOX 2 0 3 0 Black White  
% 2 4400 1400 4600 1600
% 
\special{pn 8}%
\special{pa 4400 1400}%
\special{pa 4600 1400}%
\special{pa 4600 1600}%
\special{pa 4400 1600}%
\special{pa 4400 1400}%
\special{pa 4600 1400}%
\special{fp}%
% BOX 2 0 3 0 Black White  
% 2 4600 1400 4800 1600
% 
\special{pn 8}%
\special{pa 4600 1400}%
\special{pa 4800 1400}%
\special{pa 4800 1600}%
\special{pa 4600 1600}%
\special{pa 4600 1400}%
\special{pa 4800 1400}%
\special{fp}%
% BOX 2 0 3 0 Black White  
% 2 4800 1400 5000 1600
% 
\special{pn 8}%
\special{pa 4800 1400}%
\special{pa 5000 1400}%
\special{pa 5000 1600}%
\special{pa 4800 1600}%
\special{pa 4800 1400}%
\special{pa 5000 1400}%
\special{fp}%
% BOX 2 0 3 0 Black White  
% 2 4000 1600 4200 1800
% 
\special{pn 8}%
\special{pa 4000 1600}%
\special{pa 4200 1600}%
\special{pa 4200 1800}%
\special{pa 4000 1800}%
\special{pa 4000 1600}%
\special{pa 4200 1600}%
\special{fp}%
% BOX 2 0 3 0 Black White  
% 2 4200 1600 4400 1800
% 
\special{pn 8}%
\special{pa 4200 1600}%
\special{pa 4400 1600}%
\special{pa 4400 1800}%
\special{pa 4200 1800}%
\special{pa 4200 1600}%
\special{pa 4400 1600}%
\special{fp}%
% BOX 2 0 3 0 Black White  
% 2 4400 1600 4600 1800
% 
\special{pn 8}%
\special{pa 4400 1600}%
\special{pa 4600 1600}%
\special{pa 4600 1800}%
\special{pa 4400 1800}%
\special{pa 4400 1600}%
\special{pa 4600 1600}%
\special{fp}%
% BOX 2 0 3 0 Black White  
% 2 4600 1600 4800 1800
% 
\special{pn 8}%
\special{pa 4600 1600}%
\special{pa 4800 1600}%
\special{pa 4800 1800}%
\special{pa 4600 1800}%
\special{pa 4600 1600}%
\special{pa 4800 1600}%
\special{fp}%
% BOX 2 0 3 0 Black White  
% 2 4800 1600 5000 1800
% 
\special{pn 8}%
\special{pa 4800 1600}%
\special{pa 5000 1600}%
\special{pa 5000 1800}%
\special{pa 4800 1800}%
\special{pa 4800 1600}%
\special{pa 5000 1600}%
\special{fp}%
% LINE 0 0 3 0 Black White  
% 12 5000 1800 4400 1800 4400 1800 4400 1600 4400 1600 4200 1600 4200 1600 4200 1200 4200 1200 4000 1200 4000 1200 4000 800
% 
\special{pn 20}%
\special{pa 5000 1800}%
\special{pa 4400 1800}%
\special{fp}%
\special{pa 4400 1800}%
\special{pa 4400 1600}%
\special{fp}%
\special{pa 4400 1600}%
\special{pa 4200 1600}%
\special{fp}%
\special{pa 4200 1600}%
\special{pa 4200 1200}%
\special{fp}%
\special{pa 4200 1200}%
\special{pa 4000 1200}%
\special{fp}%
\special{pa 4000 1200}%
\special{pa 4000 800}%
\special{fp}%
% STR 2 0 3 0 Black White  
% 4 2250 1930 2250 2030 2 0 0 0
% $h$
\put(22.5000,-20.3000){\makebox(0,0)[lb]{$h$}}%
% STR 2 0 3 0 Black White  
% 4 4450 1930 4450 2030 2 0 0 0
% $h^*$
\put(44.5000,-20.3000){\makebox(0,0)[lb]{$h^*$}}%
\end{picture}}%
\]
\caption{$h=(3,4,4,5,5)$ and its transpose $h^*=(2,4,5,5,5)$}
\label{pic:stair-transpose}
\end{figure}
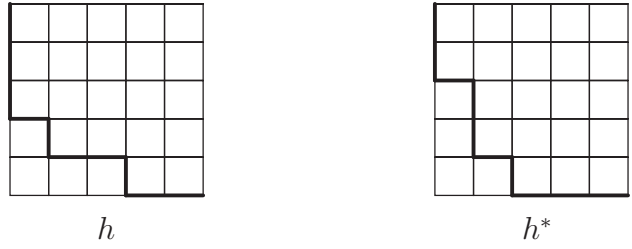
Since we are assuming that $h$ satisfies \eqref{eq: assumption on h}, so does $h^*$, that is, $h^*(i)\ge i+1$ $(1\le i<n)$.
Set
\begin{align*}
\xi_h \coloneqq \sum_{1\le i<j\le h(i)}\alpha_{i,j}
=\sum_{1\le i<j\le h(i)}(x_i-x_j).
\end{align*}
\begin{lemma}\label{lemm: new}
The  following equality holds:
\begin{align*}
\xi_h = \sum_{i=1}^{n-1} \big( h(i)-h(i+1)+2-h^*(n+1-i)+h^*(n-i) \big)\varpi_i .
\end{align*}
\end{lemma}
\begin{proof}
By definition, we have
\begin{align*}
\xi_h
=\sum_{1\le k<\ell\le h(k)}(x_{k}-x_{\ell})
=\sum_{1\le k<\ell\le h(k)} x_{k} - \sum_{1\le k<\ell\le h(k)} x_{\ell} .
\end{align*}
For each $1\le i\le n$, we count the number of $x_i$ appearing in the right-most expression. In the former summand, the number of $x_i$ is $|\{\ell\in[n]\mid i<\ell\le h(i)\}|$ which is equal to $h(i)-i$. In the latter summand, the number of $x_i$ is $|\{k\in[n]\mid k<i\le h(k)\}|$ which is equal to $h^*(n+1-i)-(n+1-i)$ by the definition of $h^*$. Thus we obtain
\begin{align*}
\xi_h
&=\sum_{i=1}^n \big( h(i)-h^*(n+1-i) +n+1-2i\big)x_i \\
&=\sum_{i=1}^{n-1} \big( h(i)-h^*(n+1-i) + n - 2i + h^*(1) \big)x_i, 
\end{align*}
where we used $x_n=-(x_1+x_2+\cdots+x_{n-1})$ for the second equality. Since we have $x_i=\varpi_{i}-\varpi_{i-1}$ with the convention $\varpi_0=0$, this means the desired equality.
\end{proof}

\begin{proposition}\label{prop: anti-can of Hess}
The anti-canonical bundle of $\Hess{S}{h}$ is isomorphic to $L_{\xi_h}$.
\end{proposition}
\begin{proof}
Since $\Hess{S}{h}$ is a smooth projective variety which admits a torus action with finite fixed points \cite{ma-pr-sh}, the higher cohomology groups of the structure sheaf vanish \cite{CL}. This means that there is a natural isomorphism $\text{Pic}(\Hess{S}{h})\cong H^2(\Hess{S}{h};\Z)$ so that 
algebraic line bundles $L$ and $L'$ over $\Hess{S}{h}$ are isomorphic if and only if their first Chern classes coincide (see for instance \cite[Corollary 5.3]{AFZ1}).

For Hessenberg functions $h\colon[n]\rightarrow[n]$ and $h'\colon[n]\rightarrow[n]$, we say $h'\subseteq h$ if and only if $h'(i)\le h(i)$ for all $1\le i\le n$. This gives a partial order on the set of Hessenberg functions on $[n]$.
We prove the claim by descending induction on $h$ with respect to the partial order $\subseteq$ given above.

When $h=(n,n,\ldots,n)$, we know that $\Hess{S}{h}=Fl(\C^n)=G/B$.
Writing $\mathfrak{g}=\Lie (G)$ and $\mathfrak{b}=\Lie (B)$,
the tangent bundle of $G/B$ is isomorphic to the vector bundle
$G\times^B(\mathfrak{g}/\mathfrak{b})$, which is the quotient of the product $G\times (\mathfrak{g}/\mathfrak{b})$ by the $B$-action given by $b\cdot(g,v)=(gb^{-1}, b\cdot v)$, where $b\cdot v$ is induced from the adjoint action of $B$ on $\mathfrak{g}$.
This means that the anti-canonical bundle of $Fl(\C^n)$ is given by the line bundle $G\times^B\wedge^{N}(\mathfrak{g}/\mathfrak{b})$, where $N=\dim_{\C} Fl(\C^n)=\dim_{\C}\mathfrak{g}/\mathfrak{b}$. 
Thus it is isomorphic to $L_{\xi}$, where $\xi$ is the sum of all positive roots (\cite[Proposition 2.2.7 (ii)]{Bri}):
\[\xi=\sum_{1\le i<j\le n}\alpha_{i,j}.\]
This verifies the case of $h=(n,n,\ldots,n)$.

Suppose by induction that the anti-canonical bundle of $\Hess{S}{h}$ is isomorphic to $L_{\xi_h}$ for a Hessenberg function $h\colon[n]\rightarrow[n]$ satisfying condition (\ref{eq: assumption on h}).
Let $p \coloneqq \min\{j \in [n] \mid h(j) \ge j + 2\}$ if exists, and $h' \subset h$ the Hessenberg function given by
\begin{align*}
h'(i) =
\begin{cases}
 h(i) \quad &\text{if $i\ne p$}, \\
 h(i)-1 &\text{if $i= p$}.
\end{cases}
\end{align*}
Then $h'$ also satisfies condition (\ref{eq: assumption on h}), and
$\Hess{S}{h'}$ is a nonsingular subvariety of $\Hess{S}{h}$ with codimension $1$.
By \cite[Lemma 5.2 and the proof of Lemma 8.11]{AHMMS}, 
the normal bundle of $\Hess{S}{h'}$ in $\Hess{S}{h}$ is isomorphic to $L_{\alpha_{p,h(p)}}$. 
Denote by $\omega_{h}$ and $\omega_{h'}$ the canonical bundles of $\Hess{S}{h}$ and $\Hess{S}{h'}$, respectively.
Then the adjunction formula tells us that
\begin{align*}
\omega_{h'}\cong \omega_{h}|_{\Hess{S}{h'}}\otimes L_{\alpha_{p,h(p)}}.
\end{align*}
By the induction hypothesis, the dual of the line bundle in the right-hand side is isomorphic to $L_{\xi_{h}}\otimes L_{-\alpha_{p,h(p)}}\cong L_{\xi_{h'}}$, as desired.
\end{proof}

For $w \in \mathfrak{S}_n$, we denote by \[X_w \subseteq G/B\quad ({\rm resp.},\ \Omega_w \subseteq G/B)\] the Schubert variety (resp., the dual Schubert variety) associated with $w$, that is, 
\begin{align*}
X_w = \overline{B w B/B}\quad({\rm resp.},\ \Omega_w = \overline{B^- w B/B}), 
\end{align*}
where $B^- \subseteq G$ is the Borel subgroup of lower-triangular matrices.
Then we have \[\dim_\C (X_w) = \dim_\C (G/B) - \dim_\C (\Omega_w) = \ell (w).\]
Let $s_1, s_2, \ldots, s_{n-1} \in \mathfrak{S}_n$ be the simple transpositions, and $e \in \mathfrak{S}_n$ the identity element.
We call each $X_{s_i}$ $(1\le i\le n-1)$ a Schubert curve.

\begin{lemma}\label{lem:plus}
$($\cite[Proposition~1.4.3]{Bri}$)$
Let $\mu=\sum_{i=1}^{n-1}a_i\varpi_i\ (a_i \in \mathbb{Z})$, and $L_{\mu}$ the corresponding line bundle over $Fl(\C^n).$ 
Then 
\[a_i = \int_{X_{s_i}}c_1(L_\mu) \qquad (1\le i\le n-1).
\]
\end{lemma}

For a complex algebraic variety $Y$, a line bundle $L$ over $Y$ is called \textbf{nef} (or \textbf{numerically effective}) if the intersection number with an arbitrary irreducible curve in $Y$ is non-negative.
$L$ is called \textbf{ample} if the global sections of $L^{\otimes m}$ give an embedding of $Y$ into a projective space for a large enough integer $m>0$. Note that if $L$ is ample, then $L$ is nef (\cite[Example 1.4.5]{Laz}). For the definitions of Fano varieties and weak Fano varieties, we refer Section \ref{sec: intro}.

\begin{lemma}\label{lem:3.3}
$\Hess{S}{h}$ contains all the Schubert curves. 
\end{lemma}

\begin{proof}
For $1\le k\le n-1$, let $h^{(k)}\colon[n]\rightarrow[n]$ be the Hessenberg function defined by 
\begin{align*}
 h^{(k)}(i) = 
 \begin{cases}
  i+1 \quad &(i=k),\\
  i &(i\neq k).
 \end{cases}
\end{align*}
By the assumption \eqref{eq: assumption on h}, we know that $\Hess{S}{h^{(k)}}\subset \Hess{S}{h}$ (cf.\ Remark~\ref{rem: connected}).
Recall that $S$ is a diagonal matrix with distinct eigenvalues.
One can verify from the definition that the $B$-orbit $Bs_kB/B$ is contained in $\Hess{S}{h^{(k)}}$, and hence that $X_{s_k}=\overline{Bs_kB/B}\subset \Hess{S}{h^{(k)}}$ by taking the closure\footnote{In fact, $X_{s_k}$ is the connected component of $\Hess{S}{h^{(k)}}$
containing the $T$-fixed point $s_k B/B$.}. 
\end{proof}

\begin{lemma}\label{lem:3.4}
Let $\mu$ be a weight of $T$.
Then the following hold.
\begin{itemize}
\item[(1)] $L_{\mu}$ is ample on $\Hess{S}{h}$ if and only if $\mu\in \sum_{i=1}^{n-1}\Z_{>0}\varpi_i$.
\item[(2)] $L_{\mu}$ is nef on $\Hess{S}{h}$ if and only if $\mu\in \sum_{i=1}^{n-1}\Z_{\ge0}\varpi_i$.
\end{itemize}
\end{lemma}
\begin{proof}
(1) If $\mu\in \sum_{i=1}^{n-1}\Z_{>0}\varpi_i$, then we see by \cite[the proof of Proposition 1.4.1]{Bri} that $L_{\mu}$ is ample on $Fl(\C^n)$. Hence $L_{\mu}$ is ample on $\Hess{S}{h}$ by definition.
Conversely, suppose that $L_{\mu}$ is ample on $\Hess{S}{h}$.
We know from Lemma~\ref{lem:3.3} that $\Hess{S}{h}$ contains all the Schubert curves $X_{s_i}$.
Hence, for all $1\le i\le n-1$, we have 
\begin{align*}
\int_{X_{s_i}}c_1(L_\mu)>0
\end{align*}
by the Nakai-Moishezon-Kleiman criterion (see \cite[Theorem 1.2.23]{Laz}). 
By Lemma \ref{lem:plus}, this means that $\mu\in \sum_{i=1}^{n-1}\Z_{>0}\varpi_i$.

(2) If $\mu\in \sum_{i=1}^{n-1}\Z_{\ge0}\varpi_i$, then we see that $L_{\mu}$ is a nef line bundle over $Fl(\C^n)$ (see \cite[Example 1.4.4 (i)]{Laz} and \cite[the proof of Proposition 1.4.1]{Bri}). Hence $L_{\mu}$ is nef on $\Hess{S}{h}$ by definition. Conversely, suppose that $L_{\mu}$ is nef on $\Hess{S}{h}$. Then we have
\begin{align*}
\int_{X_{s_i}}c_1(L_\mu)\ge 0
\end{align*}
for all $1\le i\le n-1$ by Lemma~\ref{lem:3.3}.
Thus Lemma \ref{lem:plus} shows that $\mu\in \sum_{i=1}^{n-1}\Z_{\ge0}\varpi_i$.
\end{proof}

Proposition~\ref{prop: anti-can of Hess} and Lemma~\ref{lem:3.4} now implies the following claim which establishes the equivalence of (i) and (iii) in Theorem~B stated in Section \ref{sec: intro}.

\begin{proposition}\label{prop: equiv condition for nef}
The anti-canonical bundle of $\Hess{S}{h}$ is nef if and only if the inequality
\begin{align}\label{eq: nef}
d_i \coloneqq h(i)-h(i+1)+2-h^*(n+1-i)+h^*(n-i) \ge 0
\end{align} 
holds for all $1\le i\le n-1$.
\end{proposition}

Motivated by Proposition~\ref{prop: equiv condition for nef}, we say a Hessenberg function $h\colon[n]\rightarrow[n]$ is \textbf{nef} if it satisfies the inequality \eqref{eq: nef} for all $1\le i\le n-1$.
This definition implies the following property which we will use in Section~\ref{sec: proof of Thm B}.

\begin{lemma}\label{cor: nef condition}
Suppose that $h$ is nef.
Then the following hold for all $1\le i\le n-2$.
\begin{itemize}
 \item[(1)] If $h(i)=h(i+1)<n$, then $h(i+1)<h(i+2)$.
 \item[(2)] If $h^*(i)=h^*(i+1)<n$, then $h^*(i+1)<h^*(i+2)$.
\end{itemize}
\end{lemma}

\begin{proof}
(1) Suppose for a contradiction that $h(i+1)=h(i+2)$. Then we have $h(i)=h(i+1)=h(i+2)$ by the assumption, which implies that 
$h^*(j+1)>h^*(j)+2$ for $j=n-h(i)$ since $h^*$ is the transpose of $h$. However, this is impossible since we have $d_{h(i)}\geq0$.
The same argument works for (2) by replacing $h$ with $h^*$.
\end{proof}

%%%%%%%%%%%%%%%%%%%%%%%%%%%%%%%%%%
\subsection{Proof of Theorem~A}
%%%%%%%%%%%%%%%%%%%%%%%%%%%%%%%%%%
In this subsection, we prove Theorem~A which is stated in Section \ref{sec: intro}. For this purpose, we prepare the following two lemmas.
\begin{lemma}\label{lemm:3-1}
Assume that $\Hess{S}{h}$ is Fano. Then the inequalities $h(i+1)\le h(i)+1$ and $h^*(i+1)\le h^*(i)+1$ hold for all $1\le i\le n-1$.
\end{lemma}

\begin{proof}
Since the anti-canonical bundle of $\Hess{S}{h}$ is ample by the assumption, Proposition \ref{prop: anti-can of Hess} and Lemma \ref{lem:3.4} (1) imply that the coefficients of $\xi_h$ with respect to the fundamental weights $\varpi_i$ must be positive, that is,
\begin{align*}
h(i)-h(i+1)+2+h^*(n-i)-h^*(n+1-i)>0
\end{align*} 
by Lemma~\ref{lemm: new}.
In the left-hand side, we know that $h(i)-h(i+1)$ and $h^*(n-i)-h^*(n+1-i)$ are both less than or equal to $0$. Thus the inequality means that we must have 
$-1\le h(i)-h(i+1)$  and $-1\le h^*(n-i)-h^*(n+1-i)$ for all $1\leq i\leq n-1$, which implies the desired inequalities.
\end{proof}
\begin{lemma}\label{lemm:3-2}
Assume that $\Hess{S}{h}$ is Fano. If $h(i+1)=h(i)$, then $h(i)=n$.
\end{lemma}
\begin{proof}
Suppose for a contradiction that we have $h(i)=h(i+1)<n$ for some $i\in[n-1]$. Writing $j=n-h(i)$, this means that $h^*(j+1)\ge h^*(j)+2$, which contradicts Lemma~\ref{lemm:3-1}.
\end{proof}

We now give a proof of Theorem~A.

\begin{proof}[Proof of Theorem~A]
Assume that $\Hess{S}{h}$ is Fano. Then Lemmas~\ref{lemm:3-1} and \ref{lemm:3-2} imply that $h$ must be of the form $(k+1,k+2,...,n-1,n,...,n)$ for some $1\le k\le n-1$.
This means that 
\begin{align}\label{eq: anti-can of banded}
\xi_h = 
\sum_{i=1}^{k} \varpi_i + \sum_{i=n-k}^{n-1} \varpi_i .
\end{align}
Since $L_{\xi_h}$ is ample by the assumption, Lemma~\ref{lem:3.4} (1) now implies that all coefficients of $\xi_h$ with respect to the fundamental weights must be positive. Hence it follows that
$k\geq n-k-1$ by \eqref{eq: anti-can of banded}, which is equivalent to $\frac{n-1}{2}\le k \ (\le n-1)$.

If $h=(k+1,k+2,...,n-1,n,...,n)$ for some $\frac{n-1}{2}\le k\le n-1$, then one can directly verify that \eqref{eq: anti-can of banded} holds. Thus, the coefficients of $\xi_h$ with respect to the fundamental weights are positive. This means by Lemma \ref{lem:3.4} (1) that $L_{\xi_h}$ is ample on $\Hess{S}{h}$, which implies that $\Hess{S}{h}$ is Fano.
\end{proof}

\bigskip
%%%%%%%%%%%%%%%%%%%%%%%%%%%%%%%%%%
%%%%%%%%%%%%%%%%%%%%%%%%%%%%%%%%%%
%%%%%%%%%%%%%%%%%%%%%%%%%%%%%%%%%%
\section{Relation with Richardson varieties}\label{sec: Using Richardson varieties}
%%%%%%%%%%%%%%%%%%%%%%%%%%%%%%%%%%
%%%%%%%%%%%%%%%%%%%%%%%%%%%%%%%%%%
%%%%%%%%%%%%%%%%%%%%%%%%%%%%%%%%%%
In this section, we study bigness of the anti-canonical bundles of regular semisimple Hessenberg varieties via positivity of line bundles over Richardson varieties.

%%%%%%%%%%%%%%%%%%%%%%%%%%%%%%%%%%
\subsection{Richardson varieties}
%%%%%%%%%%%%%%%%%%%%%%%%%%%%%%%%%%
Denote by $\le$ the Bruhat order on $\mathfrak{S}_n$, and by $\lessdot$ a cover in the Bruhat order, that is, $u \lessdot v$ if and only if $u < v$ and  $\ell (v) = \ell (u) +1$.
For $v, w \in \mathfrak{S}_n$ such that $v \le w$, the subvariety \[X_w ^v \coloneqq X_w \cap \Omega_v \subseteq G/B\] is called a \emph{Richardson variety}, where $X_w$ is a Schubert variety, and $\Omega_w$ is a dual Schubert variety (see Section~\ref{subsec: anti-canonical bundles}).
The Richardson variety $X_w ^v$ is irreducible, and we have \[\dim_\C (X_w ^v) = \ell(w) - \ell(v).\] 

%%%%%%%%%%%%%%%%%%%%%%%%%%%%%%%%%%
\subsection{Bigness of line bundles over Richardson varieties}\label{subsec: Richardson}
%%%%%%%%%%%%%%%%%%%%%%%%%%%%%%%%%%

A line bundle $L$ over a normal projective variety $Y$ is said to be \textbf{big} if 
the Iitaka dimension $\kappa(Y,L)$ takes the maximum possible value, i.e.\ the dimension of $Y$ (\cite[Definition 2.2.1]{Laz}). It is equivalent to the inequality
\begin{align*}
\limsup_{m\rightarrow \infty} \frac{h^0(X,L^{\otimes m})}{m^d} >0,
\end{align*}
where $d=\dim_{\C}Y$.
In this subsection, we study big line bundles over Richardson varieties, which come from line bundles over the flag variety. Our main reference is \cite{KLS}. 

Recall from Section \ref{subsec: line bundles on flag} that each weight $\mu$ of $T$ defines a line bundle $L_{\mu}$ over $X_w ^v\subseteq G/B$. In Corollary \ref{c:big line bundles on Richardson}, we give a necessary and sufficient condition for $L_\mu$ to be a big line bundle over $X_w ^v$ under the assumption that $\mu \in P_+$, which is a straightforward consequence of \cite{KLS}. Fix a parabolic subgroup $B \subseteq P \subseteq G$. Let $\mathfrak{S}_P \subset \mathfrak{S}_n$ denote the corresponding parabolic subgroup, and \[\pi_P \colon \mathfrak{S}_n \twoheadrightarrow \mathfrak{S}_n/\mathfrak{S}_P\] the canonical projection onto the set of left cosets. We use the $P$-Bruhat order on $\mathfrak{S}_n$, which is a specific lift of the Bruhat order on $\mathfrak{S}_n/\mathfrak{S}_P$ (see \cite[Ch.\ 2]{BB} and \cite[Ch.\ 4]{Ses} for references on the Bruhat order on $\mathfrak{S}_n/\mathfrak{S}_P$).

\begin{definition}[{see \cite[Sect.\ 2]{KLS}}]\normalfont
The \emph{$P$-Bruhat order} $\le_P$ on $\mathfrak{S}_n$ is defined by: $v \le_P w$ if and only if there is a chain \[v = u_0 \lessdot u_1 \lessdot u_2 \lessdot \cdots \lessdot u_k = w\] such that \[\pi_P (u_0) < \pi_P (u_1) < \pi_P (u_2) < \cdots < \pi_P (u_k)\] in the Bruhat order on $\mathfrak{S}_n/\mathfrak{S}_P$.
\end{definition}

\begin{example}\normalfont
Let $n = 3$, and take a parabolic subgroup $B \subset P \subset G$ such that $\mathfrak{S}_P$ is generated by $s_1$. Since 
\[\pi_P (e) = \pi_P (s_1) < \pi_P (s_2) = \pi_P (s_2 s_1) < \pi_P (s_1 s_2) = \pi_P (s_1 s_2 s_1),\]
the $P$-Bruhat order $\le_P$ on $\mathfrak{S}_3$ is given by the following:
\begin{align*}
&e <_P s_2 <_P s_1 s_2,\\
&s_1 <_P s_2 s_1 <_P s_1 s_2 s_1, \\
&s_1 <_P s_1s_2.
\end{align*}
\end{example}

By abuse of notation, let \[\pi_P \colon G/B \twoheadrightarrow G/P\] denote the canonical projection. For $v, w \in \mathfrak{S}_n$ such that $v \le w$, we set \[\Pi_w ^v \coloneqq \pi_P (X_w ^v) \subseteq G/P.\] This variety $\Pi_w ^v$ is called a \emph{projected Richardson variety}. The projected Richardson variety was studied by Lusztig \cite{Lus} and Rietsch \cite{Rie} in the context of total positivity, and by Goodearl-Yakimov \cite{GY} in the context of Poisson geometry. In order to study big line bundles over $X_w ^v$, we use the following relations (Propositions \ref{p:isomorphism}, \ref{p:birational}) between $X_w ^v$ and $\Pi_w ^v$, which are given in \cite{KLS}.

\begin{proposition}[{see the proof of \cite[Theorem~4.5]{KLS}}]\label{p:isomorphism}
For $v, w \in \mathfrak{S}_n$ such that $v \le w$ and an ample line bundle $L$ over $G/P$, the map \[\pi_P ^\ast \colon H^0 (\Pi_w ^v, L) \rightarrow H^0 (X_w ^v, \pi_P ^\ast L)\] is a $\C$-linear isomorphism.
\end{proposition}

\begin{proposition}[{see \cite[Sect.\ 3]{KLS}}]\label{p:birational}
For $v, w \in \mathfrak{S}_n$ such that $v \le w$, the morphism \[\pi_P \colon X_w ^v \twoheadrightarrow \Pi_w ^v\] is birational if and only if $v \le_{P} w$. In addition, this is equivalent to $\dim_\C (X_w ^v) = \dim_\C (\Pi_w ^v)$.
\end{proposition}

For $\mu \in P_+$, let $\mathfrak{S}_\mu \subseteq \mathfrak{S}_n$ be the parabolic subgroup generated by 
\[\{s_i \mid 1 \le i \le n-1,\ s_i (\mu) = \mu\}.\] 
The equality $s_i (\mu) = \mu$ is equivalent to the condition that $\mu_i = 0$ when we write $\mu = \sum_{j = 1} ^{n-1} \mu_j \varpi_j$. We denote by $B \subseteq P_\mu \subseteq G$ the unique parabolic subgroup such that $\mathfrak{S}_{P_\mu} = \mathfrak{S}_\mu$. 

\begin{corollary}\label{c:big line bundles on Richardson}
For $\mu \in P_+$ and $v, w \in \mathfrak{S}_n$ such that $v \le w$, the line bundle $L_\mu$ over $X_w ^v$ is big if and only if $v \le_{P_\mu} w$.
\end{corollary}

\begin{proof}
By the definition of $P_\mu$, the line bundle $L_\mu$ over $G/B$ is the pull-back of the ample line bundle $L_\mu$ over $G/P_\mu$ (see \cite[Sect.\ II.4.4]{Jan}). 
Thus $L_\mu$ on $\Pi_w ^v\subseteq G/P_\mu$ is big since it is ample. Since we have \[H^0 (\Pi_w ^v, L_\mu ^{\otimes k}) \simeq H^0 (X_w ^v, L_\mu ^{\otimes k})\] 
for all $k \in \Z_{> 0}$ by Proposition~\ref{p:isomorphism}, this and the definition of big line bundles imply that the line bundle $L_\mu$ over $X_w ^v$ is big if and only if $\dim_\C (X_w ^v) = \dim_\C (\Pi_w ^v)$. Since this is equivalent to $v \le_{P_\mu} w$ by Proposition~\ref{p:birational}, we obtain the assertion.
\end{proof}

%%%%%%%%%%%%%%%%%%%%%%%%%%%%%%%%%%
\subsection{Hessenberg varieties and Richardson varieties}\label{subsec: Hessenberg and Richardson}
%%%%%%%%%%%%%%%%%%%%%%%%%%%%%%%%%%

Anderson-Tymoczko \cite{AT} introduced a permutation associated with a Hessenberg function to express the cohomology classes of Hessenberg varieties in terms of Schubert classes. We use a similar but slightly different notation.

\begin{definition}
\normalfont{For a Hessenberg function $h\colon[n]\rightarrow [n]$, we define $\w{h}\in\mathfrak{S}_n$ as follows: let $\w{h}(1)=h(1)$, and take $\w{h}(i)$ to be the $(n+1-h(i))$-th largest element of $[n]\setminus\{\w{h}(1),\ldots,\w{h}(i-1)\}$.}
\end{definition}

For example, if $n=5$ and $h=(3,4,4,5,5)$ as in Figure \ref{pic:stair-shape}, then $\w{h}= 3\ 4\ 2\ 5\ 1$ in one-line notation. The positions of $1$'s of the permutation matrix associated with $\w{h}$ are depicted as the dots in Figure \ref{pic:wh}.
\begin{figure}[h]
\[
%WinTpicVersion4.32a
{\unitlength 0.1in%
\begin{picture}(50.0000,10.0000)(13.0000,-18.0000)%
% BOX 2 0 3 0 Black White  
% 2 2200 800 2400 1000
% 
\special{pn 8}%
\special{pa 2200 800}%
\special{pa 2400 800}%
\special{pa 2400 1000}%
\special{pa 2200 1000}%
\special{pa 2200 800}%
\special{pa 2400 800}%
\special{fp}%
% BOX 2 0 3 0 Black White  
% 2 2400 800 2600 1000
% 
\special{pn 8}%
\special{pa 2400 800}%
\special{pa 2600 800}%
\special{pa 2600 1000}%
\special{pa 2400 1000}%
\special{pa 2400 800}%
\special{pa 2600 800}%
\special{fp}%
% BOX 2 0 3 0 Black White  
% 2 2600 800 2800 1000
% 
\special{pn 8}%
\special{pa 2600 800}%
\special{pa 2800 800}%
\special{pa 2800 1000}%
\special{pa 2600 1000}%
\special{pa 2600 800}%
\special{pa 2800 800}%
\special{fp}%
% BOX 2 0 3 0 Black White  
% 2 2800 800 3000 1000
% 
\special{pn 8}%
\special{pa 2800 800}%
\special{pa 3000 800}%
\special{pa 3000 1000}%
\special{pa 2800 1000}%
\special{pa 2800 800}%
\special{pa 3000 800}%
\special{fp}%
% BOX 2 0 3 0 Black White  
% 2 3000 800 3200 1000
% 
\special{pn 8}%
\special{pa 3000 800}%
\special{pa 3200 800}%
\special{pa 3200 1000}%
\special{pa 3000 1000}%
\special{pa 3000 800}%
\special{pa 3200 800}%
\special{fp}%
% BOX 2 0 3 0 Black White  
% 2 2200 1000 2400 1200
% 
\special{pn 8}%
\special{pa 2200 1000}%
\special{pa 2400 1000}%
\special{pa 2400 1200}%
\special{pa 2200 1200}%
\special{pa 2200 1000}%
\special{pa 2400 1000}%
\special{fp}%
% BOX 2 0 3 0 Black White  
% 2 2400 1000 2600 1200
% 
\special{pn 8}%
\special{pa 2400 1000}%
\special{pa 2600 1000}%
\special{pa 2600 1200}%
\special{pa 2400 1200}%
\special{pa 2400 1000}%
\special{pa 2600 1000}%
\special{fp}%
% BOX 2 0 3 0 Black White  
% 2 2600 1000 2800 1200
% 
\special{pn 8}%
\special{pa 2600 1000}%
\special{pa 2800 1000}%
\special{pa 2800 1200}%
\special{pa 2600 1200}%
\special{pa 2600 1000}%
\special{pa 2800 1000}%
\special{fp}%
% BOX 2 0 3 0 Black White  
% 2 2800 1000 3000 1200
% 
\special{pn 8}%
\special{pa 2800 1000}%
\special{pa 3000 1000}%
\special{pa 3000 1200}%
\special{pa 2800 1200}%
\special{pa 2800 1000}%
\special{pa 3000 1000}%
\special{fp}%
% BOX 2 0 3 0 Black White  
% 2 3000 1000 3200 1200
% 
\special{pn 8}%
\special{pa 3000 1000}%
\special{pa 3200 1000}%
\special{pa 3200 1200}%
\special{pa 3000 1200}%
\special{pa 3000 1000}%
\special{pa 3200 1000}%
\special{fp}%
% BOX 2 0 3 0 Black White  
% 2 2200 1200 2400 1400
% 
\special{pn 8}%
\special{pa 2200 1200}%
\special{pa 2400 1200}%
\special{pa 2400 1400}%
\special{pa 2200 1400}%
\special{pa 2200 1200}%
\special{pa 2400 1200}%
\special{fp}%
% BOX 2 0 3 0 Black White  
% 2 2400 1200 2600 1400
% 
\special{pn 8}%
\special{pa 2400 1200}%
\special{pa 2600 1200}%
\special{pa 2600 1400}%
\special{pa 2400 1400}%
\special{pa 2400 1200}%
\special{pa 2600 1200}%
\special{fp}%
% BOX 2 0 3 0 Black White  
% 2 2600 1200 2800 1400
% 
\special{pn 8}%
\special{pa 2600 1200}%
\special{pa 2800 1200}%
\special{pa 2800 1400}%
\special{pa 2600 1400}%
\special{pa 2600 1200}%
\special{pa 2800 1200}%
\special{fp}%
% BOX 2 0 3 0 Black White  
% 2 2800 1200 3000 1400
% 
\special{pn 8}%
\special{pa 2800 1200}%
\special{pa 3000 1200}%
\special{pa 3000 1400}%
\special{pa 2800 1400}%
\special{pa 2800 1200}%
\special{pa 3000 1200}%
\special{fp}%
% BOX 2 0 3 0 Black White  
% 2 3000 1200 3200 1400
% 
\special{pn 8}%
\special{pa 3000 1200}%
\special{pa 3200 1200}%
\special{pa 3200 1400}%
\special{pa 3000 1400}%
\special{pa 3000 1200}%
\special{pa 3200 1200}%
\special{fp}%
% BOX 2 0 3 0 Black White  
% 2 2200 1400 2400 1600
% 
\special{pn 8}%
\special{pa 2200 1400}%
\special{pa 2400 1400}%
\special{pa 2400 1600}%
\special{pa 2200 1600}%
\special{pa 2200 1400}%
\special{pa 2400 1400}%
\special{fp}%
% BOX 2 0 3 0 Black White  
% 2 2400 1400 2600 1600
% 
\special{pn 8}%
\special{pa 2400 1400}%
\special{pa 2600 1400}%
\special{pa 2600 1600}%
\special{pa 2400 1600}%
\special{pa 2400 1400}%
\special{pa 2600 1400}%
\special{fp}%
% BOX 2 0 3 0 Black White  
% 2 2600 1400 2800 1600
% 
\special{pn 8}%
\special{pa 2600 1400}%
\special{pa 2800 1400}%
\special{pa 2800 1600}%
\special{pa 2600 1600}%
\special{pa 2600 1400}%
\special{pa 2800 1400}%
\special{fp}%
% BOX 2 0 3 0 Black White  
% 2 2800 1400 3000 1600
% 
\special{pn 8}%
\special{pa 2800 1400}%
\special{pa 3000 1400}%
\special{pa 3000 1600}%
\special{pa 2800 1600}%
\special{pa 2800 1400}%
\special{pa 3000 1400}%
\special{fp}%
% BOX 2 0 3 0 Black White  
% 2 3000 1400 3200 1600
% 
\special{pn 8}%
\special{pa 3000 1400}%
\special{pa 3200 1400}%
\special{pa 3200 1600}%
\special{pa 3000 1600}%
\special{pa 3000 1400}%
\special{pa 3200 1400}%
\special{fp}%
% BOX 2 0 3 0 Black White  
% 2 2200 1600 2400 1800
% 
\special{pn 8}%
\special{pa 2200 1600}%
\special{pa 2400 1600}%
\special{pa 2400 1800}%
\special{pa 2200 1800}%
\special{pa 2200 1600}%
\special{pa 2400 1600}%
\special{fp}%
% BOX 2 0 3 0 Black White  
% 2 2400 1600 2600 1800
% 
\special{pn 8}%
\special{pa 2400 1600}%
\special{pa 2600 1600}%
\special{pa 2600 1800}%
\special{pa 2400 1800}%
\special{pa 2400 1600}%
\special{pa 2600 1600}%
\special{fp}%
% BOX 2 0 3 0 Black White  
% 2 2600 1600 2800 1800
% 
\special{pn 8}%
\special{pa 2600 1600}%
\special{pa 2800 1600}%
\special{pa 2800 1800}%
\special{pa 2600 1800}%
\special{pa 2600 1600}%
\special{pa 2800 1600}%
\special{fp}%
% BOX 2 0 3 0 Black White  
% 2 2800 1600 3000 1800
% 
\special{pn 8}%
\special{pa 2800 1600}%
\special{pa 3000 1600}%
\special{pa 3000 1800}%
\special{pa 2800 1800}%
\special{pa 2800 1600}%
\special{pa 3000 1600}%
\special{fp}%
% BOX 2 0 3 0 Black White  
% 2 3000 1600 3200 1800
% 
\special{pn 8}%
\special{pa 3000 1600}%
\special{pa 3200 1600}%
\special{pa 3200 1800}%
\special{pa 3000 1800}%
\special{pa 3000 1600}%
\special{pa 3200 1600}%
\special{fp}%
% LINE 0 0 3 0 Black White  
% 12 2200 800 2200 1400 2200 1400 2400 1400 2400 1400 2400 1600 2400 1600 2800 1600 2800 1600 2800 1800 2800 1800 3200 1800
% 
\special{pn 20}%
\special{pa 2200 800}%
\special{pa 2200 1400}%
\special{fp}%
\special{pa 2200 1400}%
\special{pa 2400 1400}%
\special{fp}%
\special{pa 2400 1400}%
\special{pa 2400 1600}%
\special{fp}%
\special{pa 2400 1600}%
\special{pa 2800 1600}%
\special{fp}%
\special{pa 2800 1600}%
\special{pa 2800 1800}%
\special{fp}%
\special{pa 2800 1800}%
\special{pa 3200 1800}%
\special{fp}%
% CIRCLE 2 0 0 0 Black Black  
% 4 2300 1300 2320 1320 2320 1320 2320 1320
% 
\special{sh 1.000}%
\special{ia 2300 1300 28 28 0.0000000 6.2831853}%
\special{pn 8}%
\special{ar 2300 1300 28 28 0.0000000 6.2831853}%
% CIRCLE 2 0 0 0 Black Black  
% 4 2500 1500 2520 1520 2520 1520 2520 1520
% 
\special{sh 1.000}%
\special{ia 2500 1500 28 28 0.0000000 6.2831853}%
\special{pn 8}%
\special{ar 2500 1500 28 28 0.0000000 6.2831853}%
% CIRCLE 2 0 0 0 Black Black  
% 4 2700 1100 2720 1120 2720 1120 2720 1120
% 
\special{sh 1.000}%
\special{ia 2700 1100 28 28 0.0000000 6.2831853}%
\special{pn 8}%
\special{ar 2700 1100 28 28 0.0000000 6.2831853}%
% CIRCLE 2 0 0 0 Black Black  
% 4 2900 1700 2920 1720 2920 1720 2920 1720
% 
\special{sh 1.000}%
\special{ia 2900 1700 28 28 0.0000000 6.2831853}%
\special{pn 8}%
\special{ar 2900 1700 28 28 0.0000000 6.2831853}%
% CIRCLE 2 0 0 0 Black Black  
% 4 3100 900 3120 920 3120 920 3120 920
% 
\special{sh 1.000}%
\special{ia 3100 900 28 28 0.0000000 6.2831853}%
\special{pn 8}%
\special{ar 3100 900 28 28 0.0000000 6.2831853}%
% STR 2 0 3 0 Black White  
% 4 4000 1700 4000 1800 2 0 0 0
% $\begin{pmatrix} 0&0&0&0&1 \\ 0&0&1&0&0 \\ 1&0&0&0&0 \\ 0&1&0&0&0 \\ 0&0&0&1&0 \end{pmatrix}$
\put(40.0000,-18.0000){\makebox(0,0)[lb]{$\begin{pmatrix} 0&0&0&0&1 \\ 0&0&1&0&0 \\ 1&0&0&0&0 \\ 0&1&0&0&0 \\ 0&0&0&1&0 \end{pmatrix}$}}%
% BOX 2 5 3 0 Black White  
% 2 1300 800 6300 1800
% 
\special{pn 8}%
\special{pa 1300 800}%
\special{pa 6300 800}%
\special{pa 6300 1800}%
\special{pa 1300 1800}%
\special{pa 1300 800}%
\special{ip}%
\end{picture}}%
\]
\caption{The positions of $1$'s of the permutation matrix associated with $\w{h}$ for $h=(3,4,4,5,5)$.}
\label{pic:wh}
\end{figure}
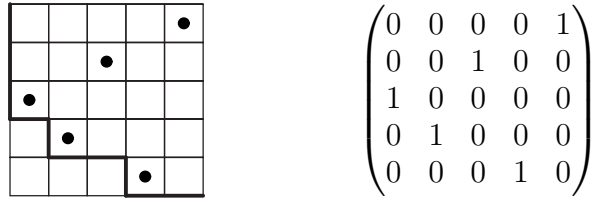
\begin{remark}\label{rem: w(h)}
\normalfont{The permutation $w(h)\coloneqq (w_0\w{h})^{-1}$ is precisely the one which was considered in \cite{AT}.}
\end{remark}

Let $[\Hess{S}{h}]\in H^*(Fl(\C^n))$ be the cohomology class of $\Hess{S}{h}$. We have the following formula\footnote{In \cite{AT}, it was described in terms of the permutation $w(h)$ which is explained in Remark \ref{rem: w(h)}.} for $[\Hess{S}{h}]$ in terms of products of Schubert classes due to \cite[Corollary 3.3 and equation (14)]{AT}: \begin{align}\label{eq: additional 450}
[\Hess{S}{h}] = \sum_{\substack{u\in \mathfrak{S}_n; \\ \ell (u)+\ell(\w{h})=\ell(u\w{h}) }} [\Omega_{u}][\Omega_{w_0u\w{h}}].
\end{align}
Using this formula, we deduce a sufficient condition for the anti-canonical bundle $L_{\xi_h}$ of $\Hess{S}{h}$ to be big when it is assumed to be nef.

\begin{proposition}\label{cor: using Richardson 20}
Assume that $\xi_h \in P_+$, that is, $L_{\xi_h}$ is a nef line bundle over $G/B$. If there exists $u \in \mathfrak{S}_n$ such that \begin{equation}\label{eq:key conditions ver2}
\begin{aligned}
\ell(u) + \ell (\w{h}) = \ell(u\w{h}), \quad
u \le_{P_{\xi_h}} u\w{h},
\end{aligned}
\end{equation}
then $L_{\xi_h}$ is a big line bundle over $\Hess{S}{h}$.
\end{proposition} 

\begin{proof}
According to \cite[Theorem~2.2.16]{Laz}, it suffices to prove that
\begin{align*}
\int_{\Hess{S}{h}} c_1(L_{\xi_h})^d >0, 
\end{align*}
where $d=\dim_{\C} \Hess{S}{h}$.
By multiplying the class $[\Hess{S}{h}]\in H^*(Fl(\C^n))$, we may express the integral on $\Hess{S}{h}$ as an integral on $Fl(\C^n)$:
\begin{align}\notag
\int_{\Hess{S}{h}} c_1(L_{\xi_h})^d 
= \int_{Fl(\C^n)} c_1(L_{\xi_h})^d [\Hess{S}{h}]. 
\end{align}
Combining this with \eqref{eq: additional 450}, we obtain
\begin{align}\label{eq: additional 460}
\int_{\Hess{S}{h}} c_1(L_{\xi_h})^d 
= \sum_{\substack{u\in \mathfrak{S}_n; \\ \ell (u)+\ell(\w{h})=\ell(u\w{h}) }} \int_{Fl(\C^n)} c_1(L_{\xi_h})^d [\Omega_{u}][\Omega_{w_0u\w{h}}].
\end{align}
We claim that each summand in the right-hand side is non-negative.
This is because we may expand the product $[\Omega_{u}][\Omega_{w_0u\w{h}}]$ as a non-negative sum of the (dual) Schubert classes by Kleiman's transversality theorem (\cite[Sect.\ 1.3]{Bri}): 
 \begin{align*}
[\Omega_{u}][\Omega_{w_0u\w{h}}] 
= \sum_{v\in\mathfrak{S}_n} c_v [\Omega_{v}] \qquad (c_v\ge0).
\end{align*}
Hence each integral in the right-hand side of \eqref{eq: additional 460} is expressed as
\begin{align*}
\int_{Fl(\C^n)} c_1(L_{\xi_h})^d [\Omega_{u}][\Omega_{w_0u\w{h}}] 
&= \sum_{v\in\mathfrak{S}_n} c_v \int_{Fl(\C^n)} c_1(L_{\xi_h})^d  [\Omega_{v}] \\ \label{eq: additional 470}
&= \sum_{v\in\mathfrak{S}_n} c_v \int_{\Omega_{v}} c_1(L_{\xi_h})^d.
\end{align*}
Since $L_{\xi_h}$ is nef on $\Omega_{v}$, this is a non-negative integer, as claimed above. Thus it suffices to find a permutation $u\in\mathfrak{S}_n$ in \eqref{eq: additional 460} such that $\int_{Fl(\C^n)} c_1(L_{\xi_h})^d [\Omega_{u}][\Omega_{w_0u\w{h}}]>0$.

Now, take $u\in \mathfrak{S}_n$ which satisfies the assumption \eqref{eq:key conditions ver2}. Then
the integral
\[
\int_{Fl(\C^n)} c_1(L_{\xi_h})^d [\Omega_{u}][\Omega_{w_0u\w{h}}]
\]
appears as a summand in \eqref{eq: additional 460}. 
Note that the second condition of \eqref{eq:key conditions ver2} implies that $u \le u\w{h}$ in the Bruhat order.
Since 
$
[\Omega_{w_0u\w{h}}] = [X_{u\w{h}}]
$ 
by \cite[Lemma~3 in Sect.\ 10.2]{Ful}, we have
\begin{align*}
[\Omega_{u}][\Omega_{w_0u\w{h}}] 
= [\Omega_{u}][X_{u\w{h}}] = [X_{u\w{h}} ^{u}], 
\end{align*}
where the second equality follows from $u \le u\w{h}$ and \cite[Sect.\ 1.3]{Bri}. Hence it follows that
 \begin{align*}
\int_{Fl(\C^n)} c_1(L_{\xi_h})^d [\Omega_{u}][\Omega_{w_0u\w{h}}] = \int_{X_{u\w{h}} ^{u}} c_1(L_{\xi_h}) ^{d},
\end{align*}
where we note that $d = \dim_\C (\Hess{S}{h}) = \dim_\C (X_{u\w{h}} ^{u})$. 
Since $u \le_{P_{\xi_h}} u\w{h}$, we see by Corollary \ref{c:big line bundles on Richardson} that $L_{\xi_h}$ on $X_{u\w{h}} ^{u}$ is nef and big, which implies that \[\int_{X_{u\w{h}} ^{u}} c_1(L_{\xi_h}) ^{d} > 0\]
by \cite[Theorem~2.2.16]{Laz}.
From this and the argument above, it follows that \[\int_{\Hess{S}{h}} c_1(L_{\xi_h}) ^{d} > 0.\] 
\end{proof}

Let $\mathfrak{S}_P \subset \mathfrak{S}_n$ be a parabolic subgroup as in Section \ref{subsec: Richardson}.
Note that for $w\in \mathfrak{S}_n$, there is a unique factorization 
\[
w = w^P w_P
\]
with $w^P\in \mathfrak{S}^P$ and $w_P\in \mathfrak{S}_P$, where $\mathfrak{S}^P$ is the set of minimal length representatives for $\mathfrak{S}_n/\mathfrak{S}_P$ (cf.\ \cite[Sect.\ 1.2]{Bri}). 
In the next section, we will use the following claim to find the desired $u\in \mathfrak{S}_n$ in the previous proposition.

\begin{lemma}\label{cor: using Richardson 30}
If $u\in\mathfrak{S}_P$ and $u_P = (u\w{h})_P$, then we have $u \le_P u \w{h}$.
\end{lemma} 

\begin{proof}
Since we have $e\le (u\w{h})^P$ and $(u\w{h})^P\in\mathfrak{S}^P$, we obtain
$e\le_P (u\w{h})^P$ by \cite[Proposition~2.5]{KLS}.
Hence there exists a chain 
\[e = u_0 \lessdot u_1 \lessdot u_2 \lessdot \cdots \lessdot u_k = (u\w{h})^P
\] 
of permutations $u_0,\ldots ,u_k\in\mathfrak{S}_n$ such that 
\[
\pi_P (u_0) < \pi_P (u_1) < \pi_P (u_2) < \cdots < \pi_P (u_k).
\] 
It follows that $u_i \in \mathfrak{S}^P$ for all $0 \le i \le k$ by induction on $i$.
We prove this as follows.
Since $u_0=e \in \mathfrak{S}^P$, we have $\ell(\pi_P (u_0)) = \ell(u_0)$, and hence we obtain that
\[\ell(\pi_P (u_1)) -\ell(\pi_P (u_0)) \le \ell (u_1) -\ell(u_0) = 1.\]
Since $\pi_P (u_0) < \pi_P (u_1)$, it also follows that $\ell(\pi_P (u_1)) -\ell(\pi_P (u_0)) \ge 1$, and hence that $\ell(\pi_P (u_1)) -\ell(\pi_P (u_0)) = 1$. Thus we obtain $\ell(\pi_P (u_1)) = \ell(u_1)$ by $\ell(\pi_P (u_0)) = \ell(u_0)$, and this means that $u_1 \in \mathfrak{S}^P$. Continuing this argument, we have $u_i \in \mathfrak{S}^P$ for all $0 \le i \le k$.
From these, it follows that 
\begin{align*}
&u_P= u_0 u_P \lessdot u_1 u_P \lessdot u_2 u_P \lessdot \cdots \lessdot u_k u_P = (u\w{h})^P u_P, \\ 
&\pi_P (u_0 u_P) < \pi_P (u_1 u_P) < \pi_P (u_2 u_P) < \cdots < \pi_P (u_k u_P).
\end{align*}
The left-most permutation is $u$ by the assumption $u\in\mathfrak{S}_P$, and the right-most permutation is $u\w{h}$ by the assumption $u_P = (u\w{h})_P$.
Thus we have proved $u \le_P u \w{h}$. 
\end{proof}

\bigskip
%%%%%%%%%%%%%%%%%%%%%%%%%%%%%%%%%%
%%%%%%%%%%%%%%%%%%%%%%%%%%%%%%%%%%
\section{Weak Fano Hessenberg  varieties}\label{sec: proof of Thm B}
%%%%%%%%%%%%%%%%%%%%%%%%%%%%%%%%%%
%%%%%%%%%%%%%%%%%%%%%%%%%%%%%%%%%%

In this section, we prove Theorem~B which is stated in Section \ref{sec: intro}. 
We first prepare some notations and lemmas in Sections \ref{subsec: preliminary notations} and \ref{subsec: similar shapes}. A proof of Theorem B is given in Section \ref{subsect: Proof of Thm B}. To exhibit our argument, we provide a pair of running examples for $n=20$ and $n=19$ in Section \ref{subsec: Example}, which we will refer repeatedly.
Throughout this section, we always assume that $h$ is nef, that is, we assume that
\begin{align*}
h(i)-h(i+1)+2-h^*(n+1-i)+h^*(n-i) \ge 0
\end{align*} 
for all $1\le i\le n-1$.

%%%%%%%%%%%%%%%%%%
\subsection{Preliminary notations}\label{subsec: preliminary notations}
%%%%%%%%%%%%%%%%%%
Let $h\colon[n]\rightarrow [n]$ be a nef Hessenberg function satisfying the assumption \eqref{eq: assumption on h}, that is, $h(i)\geq i+1$ for $1\le i<n$. The weight $\xi_h \in P_+$ of the anti-canonical bundle $L_{\xi_h}$ of $\Hess{S}{h}$ defines a parabolic subgroup $\mathfrak{S}_{P_{\xi_h}}\subseteq \mathfrak{S}_n$ as in Section \ref{subsec: Richardson}. This subgroup is generated by the simple transpositions $s_{i}$ satisfying $s_i(\xi_h)=\xi_h$, that is, $d_i=0$ when we write $\xi_h=\sum_{i=1}^{n-1}d_i\varpi_i$. Let us describe this more explicitly in what follows.
For $1 \le i \le n-1$ such that $s_i (\xi_h) = \xi_h$, we set
\begin{align*}
&\rb{i} \coloneqq \max\{k \ge 0 \mid s_i (\xi_h) = s_{i +1} (\xi_h) = \cdots = s_{i +k} (\xi_h) = \xi_h\},\\
&\lb{i} \coloneqq \max\{k \ge 0 \mid s_{i-k} (\xi_h) = \cdots = s_{i -1} (\xi_h) = s_i (\xi_h) = \xi_h\}, \\
&J_i \coloneqq \{i-\lb{i}, \ldots, i-1, i, i+1, \ldots, i +\rb{i} +1\}.
\end{align*} 
Noticing that $J_{i-\lb{i}}=\cdots=J_i=\cdots=J_{i +\rb{i}}$, let $J\coloneqq \{J_i \mid 1\le i\le n-1, \ s_i(\xi_h)=\xi_h\}$. For example, $J=\{J_9, J_{19}\}$ in the running example in Section \ref{subsec: Example}.
The parabolic subgroup $\mathfrak{S}_{P_{\xi_h}}$ is now given by $\prod_{J_i\in J} \mathfrak{S}_{J_i} \subseteq \mathfrak{S}_n$, where each $\mathfrak{S}_{J_i}$ is the permutation group on $J_i$ which is regarded as a subgroup of $\mathfrak{S}_n$ in the natural way. 
For simplicity, we denote this parabolic subgroup by $\mathfrak{S}_J$, that is,
\begin{align*}
\mathfrak{S}_J = \prod_{J_i\in J} \mathfrak{S}_{J_i} = \mathfrak{S}_{P_{\xi_h}} \subseteq \mathfrak{S}_n.
\end{align*} 
When $J=\emptyset$, we mean that $\mathfrak{S}_J$ is the trivial subgroup consisting of the identity element. For example, if $n=7$ and $h=(2,4,5,6,7,7,7)$, then $\xi_h=0\varpi_1+0\varpi_2+\varpi_3+0\varpi_4+\varpi_5+\varpi_6$ so that $\mathfrak{S}_J\cong \mathfrak{S}_3\times \mathfrak{S}_2$. When we indicate the dependence of $J_i$ and $J$ on the Hessenberg function $h$, we will also denote them by $J_i(h)$ and $J(h)$, respectively.

Recall from Section \ref{subsec: Hessenberg and Richardson} that for a permutation $w\in\mathfrak{S}_n$, there is a unique factorization 
\[
w = w^J w_J
\]
with $w^J\in\mathfrak{S}^J$ and $w_J\in\mathfrak{S}_J$, where $\mathfrak{S}^J$ is the set of minimal length representatives for $\mathfrak{S}_n/\mathfrak{S}_J$. 
Note that $w_J$ encodes the order of the numbers of $w$ on each $J_i$ in one-line notation. More specifically, for $v,w\in\mathfrak{S}_n$, the equality $v_J=w_J$ is equivalent to the condition that the following statement holds for each 
$J_i\in J$: 
\begin{align*}
\text{$v (j_1) < v (j_2)$ if and only if $w (j_1) < w (j_2)$ for $j_1, j_2 \in J_i$.}
\end{align*} 

For $2\le i\le n$, we say that $h$ is \textbf{stable at} $i$ if 
\begin{align*}
h(i)=h(i-1).
\end{align*} 
For example, $h$ is stable at $i=3$ in the running example in Section \ref{subsec: Example}.

For $1 \le i \le n-1$, we consider the following two conditions 
\begin{equation}\label{eq:def of L 10}
\begin{split}
&\{j \ge i \mid h(j +1) = h(j) +2\} = \emptyset \text{ or } \\
&\min\{j \ge i \mid h(j +1) = h(j)\} < \min\{j \ge i \mid h(j +1) = h(j) +2\},
\end{split}
\end{equation}
and 
\begin{equation}\label{eq:def of L 20}
\begin{split}
&\{j \ge i \mid h(j +1) = h(j) +2\} \neq \emptyset \text{ and} \\
&\min\{j \ge i \mid h(j +1) = h(j)\} > \min\{j \ge i \mid h(j +1) = h(j) +2\}.
\end{split}
\end{equation}
Noticing that either \eqref{eq:def of L 10} or \eqref{eq:def of L 20} holds for all $1\le i\le n-1$, we set
\begin{align*}
i^{(+)} \coloneqq 
\begin{cases}
i \qquad &(\text{if $i$ satisfies \eqref{eq:def of L 10}}), \\
\w{h}^{-1} (h(\hat{i})+1) &(\text{if $i$ satisfies \eqref{eq:def of L 20}}),
\end{cases}
\end{align*}
for $1\le i\le n-1$, where $\hat{i} \coloneqq \min\{j \ge i \mid h(j +1) = h(j) +2\}$. See Figure \ref{pic:def of Li}. Note that $i\le i^{(+)}\le n-1$ since $n=\w{h}^{-1}(1)$. It is obvious that the $(+)$-operation will be trivial after repeating it on $i$ sufficiently many times $i \mapsto i^{(+)} \mapsto (i^{(+)})^{(+)} \mapsto \cdots$, and we denote by $i^{(+\infty)}(\le n-1)$ the limit of this sequence.

For $1\le i\le n-1$, we define 
\begin{align}\label{eq: def of Li}
\LL{i} \coloneqq \text{min}\{j\ge i^{(+\infty)} \mid h(j)=h(j+1)\}+1,
\end{align}
where the set appearing in the right-hand side is non-empty since we have $h(n-1)=h(n)$. For example,  $\LL{4}=6$ and $\LL{6}=15$ in the running example of $n=20$ in Section \ref{subsec: Example}. When we indicate the dependence of these operations on the Hessenberg function $h$, we will also denote them by $i^{(+)_h}$ and $\LL{i}(h)$, respectively.
\begin{figure}[htbp]
%WinTpicVersion4.32a
{\unitlength 0.1in%
\begin{picture}(56.0000,26.9000)(19.0000,-34.5000)%
% CIRCLE 2 0 0 0 Black Black  
% 4 5300 1700 5320 1720 5320 1720 5320 1720
% 
\special{sh 1.000}%
\special{ia 5300 1700 28 28 0.0000000 6.2831853}%
\special{pn 8}%
\special{ar 5300 1700 28 28 0.0000000 6.2831853}%
% LINE 2 2 3 0 Black Black  
% 2 4700 800 4700 3200
% 
\special{pn 8}%
\special{pa 4700 800}%
\special{pa 4700 3200}%
\special{dt 0.045}%
% STR 2 0 3 0 Black White  
% 4 4680 3300 4680 3400 2 0 0 0
% $i$
\put(46.8000,-34.0000){\makebox(0,0)[lb]{$i$}}%
% CIRCLE 2 0 0 0 Black Black  
% 4 6100 2500 6120 2520 6120 2520 6120 2520
% 
\special{sh 1.000}%
\special{ia 6100 2500 28 28 0.0000000 6.2831853}%
\special{pn 8}%
\special{ar 6100 2500 28 28 0.0000000 6.2831853}%
% STR 2 0 3 0 Black White  
% 4 6240 3320 6240 3420 2 0 0 0
% $i^{(+)}$
\put(62.4000,-34.2000){\makebox(0,0)[lb]{$i^{(+)}$}}%
% LINE 2 0 3 0 Black Black  
% 4 6000 2392 6000 2592 6000 2592 6400 2592
% 
\special{pn 8}%
\special{pa 6000 2392}%
\special{pa 6000 2592}%
\special{fp}%
\special{pa 6000 2592}%
\special{pa 6400 2592}%
\special{fp}%
% STR 2 0 3 0 Black White  
% 4 5600 1940 5600 2040 2 0 0 0
% $\ddots$
\put(56.0000,-20.4000){\makebox(0,0)[lb]{$\ddots$}}%
% LINE 2 0 3 0 Black Black  
% 6 5000 1400 5200 1400 5200 1400 5200 1800 5200 1800 5400 1800
% 
\special{pn 8}%
\special{pa 5000 1400}%
\special{pa 5200 1400}%
\special{fp}%
\special{pa 5200 1400}%
\special{pa 5200 1800}%
\special{fp}%
\special{pa 5200 1800}%
\special{pa 5400 1800}%
\special{fp}%
% LINE 2 0 3 0 Black White  
% 2 6400 2592 6400 2792
% 
\special{pn 8}%
\special{pa 6400 2592}%
\special{pa 6400 2792}%
\special{fp}%
% LINE 2 2 3 0 Black Black  
% 2 6300 800 6300 3200
% 
\special{pn 8}%
\special{pa 6300 800}%
\special{pa 6300 3200}%
\special{dt 0.045}%
% BOX 2 5 3 0 Black White  
% 2 1900 760 7500 3450
% 
\special{pn 8}%
\special{pa 1900 760}%
\special{pa 7500 760}%
\special{pa 7500 3450}%
\special{pa 1900 3450}%
\special{pa 1900 760}%
\special{ip}%
% LINE 2 0 3 0 Black White  
% 4 5000 1400 5000 1200 5000 1200 4800 1200
% 
\special{pn 8}%
\special{pa 5000 1400}%
\special{pa 5000 1200}%
\special{fp}%
\special{pa 5000 1200}%
\special{pa 4800 1200}%
\special{fp}%
% LINE 2 2 3 0 Black White  
% 2 4600 1500 7200 1500
% 
\special{pn 8}%
\special{pa 4600 1500}%
\special{pa 7200 1500}%
\special{dt 0.045}%
% CIRCLE 2 0 0 0 Black Black  
% 4 6300 1500 6320 1520 6320 1520 6320 1520
% 
\special{sh 1.000}%
\special{ia 6300 1500 28 28 0.0000000 6.2831853}%
\special{pn 8}%
\special{ar 6300 1500 28 28 0.0000000 6.2831853}%
% LINE 2 2 3 0 Black Black  
% 2 2300 800 2300 3200
% 
\special{pn 8}%
\special{pa 2300 800}%
\special{pa 2300 3200}%
\special{dt 0.045}%
% STR 2 0 3 0 Black White  
% 4 2120 3300 2120 3400 2 0 0 0
% $i=i^{(+)}$
\put(21.2000,-34.0000){\makebox(0,0)[lb]{$i=i^{(+)}$}}%
% STR 2 0 3 0 Black White  
% 4 3040 3320 3040 3420 2 0 0 0
% $\LL{i}$
\put(30.4000,-34.2000){\makebox(0,0)[lb]{$\LL{i}$}}%
% LINE 2 0 3 0 Black White  
% 4 2400 1800 2400 1600 2400 1600 2200 1600
% 
\special{pn 8}%
\special{pa 2400 1800}%
\special{pa 2400 1600}%
\special{fp}%
\special{pa 2400 1600}%
\special{pa 2200 1600}%
\special{fp}%
% LINE 2 0 3 0 Black White  
% 10 2400 1800 2600 1800 2600 1800 2600 2000 2600 2000 2800 2000 2800 2000 2800 2200 2800 2200 3200 2200
% 
\special{pn 8}%
\special{pa 2400 1800}%
\special{pa 2600 1800}%
\special{fp}%
\special{pa 2600 1800}%
\special{pa 2600 2000}%
\special{fp}%
\special{pa 2600 2000}%
\special{pa 2800 2000}%
\special{fp}%
\special{pa 2800 2000}%
\special{pa 2800 2200}%
\special{fp}%
\special{pa 2800 2200}%
\special{pa 3200 2200}%
\special{fp}%
% LINE 2 2 3 0 Black Black  
% 2 3100 800 3100 3200
% 
\special{pn 8}%
\special{pa 3100 800}%
\special{pa 3100 3200}%
\special{dt 0.045}%
% LINE 2 0 3 0 Black White  
% 6 6400 2790 6600 2790 6600 2790 6600 2990 6600 2990 7000 2990
% 
\special{pn 8}%
\special{pa 6400 2790}%
\special{pa 6600 2790}%
\special{fp}%
\special{pa 6600 2790}%
\special{pa 6600 2990}%
\special{fp}%
\special{pa 6600 2990}%
\special{pa 7000 2990}%
\special{fp}%
% LINE 2 2 3 0 Black Black  
% 2 6900 800 6900 3200
% 
\special{pn 8}%
\special{pa 6900 800}%
\special{pa 6900 3200}%
\special{dt 0.045}%
% STR 2 0 3 0 Black White  
% 4 6840 3320 6840 3420 2 0 0 0
% $\LL{i}$
\put(68.4000,-34.2000){\makebox(0,0)[lb]{$\LL{i}$}}%
% CIRCLE 2 0 0 0 Black Black  
% 4 2900 2100 2920 2120 2920 2120 2920 2120
% 
\special{sh 1.000}%
\special{ia 2900 2100 28 28 0.0000000 6.2831853}%
\special{pn 8}%
\special{ar 2900 2100 28 28 0.0000000 6.2831853}%
% CIRCLE 2 0 0 0 Black Black  
% 4 2700 1900 2720 1920 2720 1920 2720 1920
% 
\special{sh 1.000}%
\special{ia 2700 1900 28 28 0.0000000 6.2831853}%
\special{pn 8}%
\special{ar 2700 1900 28 28 0.0000000 6.2831853}%
% CIRCLE 2 0 0 0 Black Black  
% 4 2500 1700 2520 1720 2520 1720 2520 1720
% 
\special{sh 1.000}%
\special{ia 2500 1700 28 28 0.0000000 6.2831853}%
\special{pn 8}%
\special{ar 2500 1700 28 28 0.0000000 6.2831853}%
% CIRCLE 2 0 0 0 Black Black  
% 4 5100 1300 5120 1320 5120 1320 5120 1320
% 
\special{sh 1.000}%
\special{ia 5100 1300 28 28 0.0000000 6.2831853}%
\special{pn 8}%
\special{ar 5100 1300 28 28 0.0000000 6.2831853}%
% CIRCLE 2 0 0 0 Black Black  
% 4 6500 2700 6520 2720 6520 2720 6520 2720
% 
\special{sh 1.000}%
\special{ia 6500 2700 28 28 0.0000000 6.2831853}%
\special{pn 8}%
\special{ar 6500 2700 28 28 0.0000000 6.2831853}%
% CIRCLE 2 0 0 0 Black Black  
% 4 6700 2900 6720 2920 6720 2920 6720 2920
% 
\special{sh 1.000}%
\special{ia 6700 2900 28 28 0.0000000 6.2831853}%
\special{pn 8}%
\special{ar 6700 2900 28 28 0.0000000 6.2831853}%
% STR 2 0 3 0 Black White  
% 4 3900 2040 3900 2140 2 0 0 0
% or
\put(39.0000,-21.4000){\makebox(0,0)[lb]{or}}%
% LINE 2 2 3 0 Black Black  
% 2 5100 800 5100 3200
% 
\special{pn 8}%
\special{pa 5100 800}%
\special{pa 5100 3200}%
\special{dt 0.045}%
% STR 2 0 3 0 Black White  
% 4 5080 3300 5080 3400 2 0 0 0
% $\hat{i}$
\put(50.8000,-34.0000){\makebox(0,0)[lb]{$\hat{i}$}}%
% LINE 2 0 3 0 Black White  
% 2 4800 1200 4800 1200
% 
\special{pn 8}%
\special{pa 4800 1200}%
\special{pa 4800 1200}%
\special{fp}%
% LINE 2 0 3 0 Black White  
% 4 4800 1200 4800 1000 4800 1000 4600 1000
% 
\special{pn 8}%
\special{pa 4800 1200}%
\special{pa 4800 1000}%
\special{fp}%
\special{pa 4800 1000}%
\special{pa 4600 1000}%
\special{fp}%
% CIRCLE 2 0 0 0 Black Black  
% 4 4900 1100 4920 1120 4920 1120 4920 1120
% 
\special{sh 1.000}%
\special{ia 4900 1100 28 28 0.0000000 6.2831853}%
\special{pn 8}%
\special{ar 4900 1100 28 28 0.0000000 6.2831853}%
\end{picture}}%
\caption{$i^{(+)}$ and $\LL{i}$.}
\label{pic:def of Li}
\end{figure}
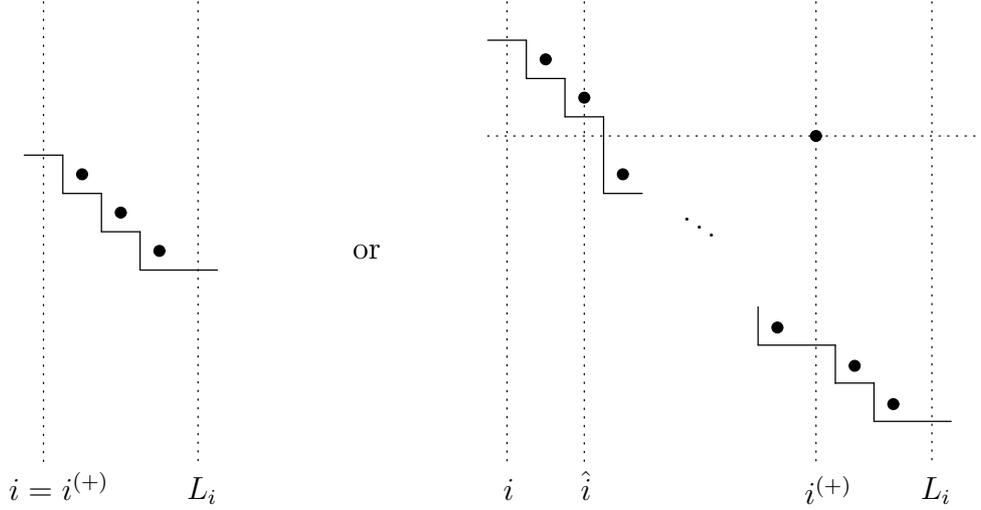

\begin{lemma}\label{lem: wh and L}
Let $h\colon[n]\rightarrow[n]$ be a nef Hessenberg function.
For $1 \le i \le n-1$ and $i +1 \le j < \LL{i} $, we have
\begin{align*}
\w{h} (i) < \w{h} (j).
\end{align*}
\end{lemma}

\begin{proof}
We first consider the case $i=i^{(+)}$.
In this case, we have 
\begin{align*}
h(k)=h(k-1)+1 \quad (i+1\le k< \LL{i}),
\end{align*}
which means that $\w{h}(k)=h(k)$ for $i+1\le k< \LL{i}$ by the definition of $\w{h}$. Taking this equality in the case $k=j$, we obtain
\begin{align*}
\w{h}(i)\le h(i)\le h(j)=\w{h}(j),
\end{align*}
where the first equality follows from the definition of $\w{h}$. Since $i\neq j$, we obtain the desired claim in this case.

We next consider the case $i<i^{(+)}$. In this case, it is clear that 
\begin{align}\label{eq: prop of L 10}
\w{h} (i) < \w{h} (k) \quad (i+1\le k\le i^{(+)})
\end{align}
by the definition of $i^{(+)}$, and the maximality of $\w{h}(k)$. Let us prove that we can extend the range of $k$ as
\begin{align}\label{eq: prop of L 20}
\w{h} (i) < \w{h} (k) \quad (i+1\le k\le (i^{(+)})^{(+)}).
\end{align}
We take cases.
If $i^{(+)}=(i^{(+)})^{(+)}$, then \eqref{eq: prop of L 20} is the same as \eqref{eq: prop of L 10}.
If $i^{(+)}<(i^{(+)})^{(+)}$, then we have
\begin{align}\label{eq: prop of L 30}
\w{h} (i^{(+)}) < \w{h} (k) \quad (i^{(+)}+1\le k\le (i^{(+)})^{(+)})
\end{align}
as we obtained \eqref{eq: prop of L 10}.
Combining \eqref{eq: prop of L 10} and \eqref{eq: prop of L 30}, we obtain \eqref{eq: prop of L 20} in this case as well. 
By continuing this argument, we see that 
\begin{align}\label{eq: prop of L 40}
\w{h} (i) < \w{h} (k) \quad (i+1\le k\le i^{(+\infty)}).
\end{align}
Hence we assume $i^{(+\infty)}+1\le j$ in the following.
Then, since $(i^{(+\infty)})^{(+)}=i^{(+\infty)}$, 
the same argument as in the case $i=i^{(+)}$ implies that $\w{h} (i^{(+\infty)})<\w{h}(j)$. Combining this with \eqref{eq: prop of L 40}, we obtain $\w{h} (i) < \w{h} (j)$.
\end{proof}

\begin{lemma}\label{l:connected_component_containing_1}
If $s_1 (\xi_h) = \xi_h$, then 
\[h(1) < h(2) < \cdots < h(\rb{1} +2).\] 
In particular, $\w{h} (k) = h(k)$ for $k\in J_1$.
\end{lemma}

\begin{proof}
Suppose that there exists $1\leq q\leq \rb{1} +1$ such that 
$h(q)=h(q+1)$.
Then the definition of $k_1$ implies that $s_q (\xi_h) = \xi_h$. Hence we have
\begin{align}\label{eq: additional 320}
h^*(n+1-q)=h^*(n-q)+2.
\end{align}
This in fact implies 
\begin{align*}
h^*(n)=h^*(n-1)+2
\end{align*}
as follows.
If $q=1$, then the claim is obvious. 
If not, then let $q'\coloneqq n+1-h^*(n+1-q)$. We then have $h(q')=h(q'+1)$ by \eqref{eq: additional 320}, and $1\le q'<q$ by
\[
q' = n+1-h^*(n+1-q) < n+1-(n+1-q) = q.
\]
This means that $q'\le \rb{1} +1$, and hence we have 
\begin{align*}
h^*(n+1-q')=h^*(n-q')+2
\end{align*}
as above.
By continuing this argument, it follows that 
$h^*(n)=h^*(n-1)+2$, as claimed above. However, this implies that $h(1)=h(2)=1$, which contradicts the definition of a Hessenberg function.
\end{proof}

%%%%%%%%%%%%%%%%%%
\subsection{Principle of similar shapes}\label{subsec: similar shapes}
%%%%%%%%%%%%%%%%%%
For each $1\le i\le n$, let 
\begin{align*}
\DD{i}\coloneqq n-h^*(n+1-i).
\end{align*} 
This measures the horizontal distance between the left-side wall and the boundary of $h$ on the $i$-th row; see Figure \ref{pic:def of D(i)}.
For example, $\DD{11}=3$ and $\DD{13}=6$ in the running example for $n=20$ in Section \ref{subsec: Example}.
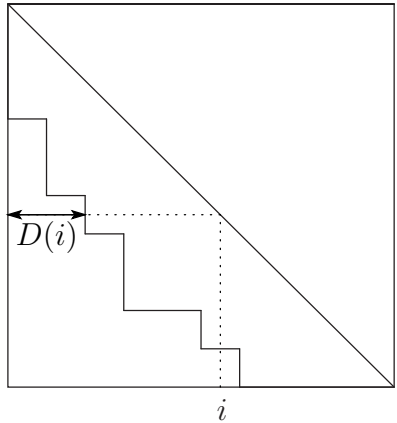
\begin{figure}[htbp]
%WinTpicVersion4.32a
{\unitlength 0.1in%
\begin{picture}(24.0000,24.0000)(30.0000,-42.7000)%
% LINE 2 0 3 0 Black Black  
% 2 3200 2000 5200 4000
% 
\special{pn 8}%
\special{pa 3200 2000}%
\special{pa 5200 4000}%
\special{fp}%
% BOX 2 0 3 0 Black Black  
% 2 3200 2000 5200 4000
% 
\special{pn 8}%
\special{pa 3200 2000}%
\special{pa 5200 2000}%
\special{pa 5200 4000}%
\special{pa 3200 4000}%
\special{pa 3200 2000}%
\special{pa 5200 2000}%
\special{fp}%
% LINE 2 0 3 0 Black Black  
% 24 3200 2000 3200 2600 3200 2600 3400 2600 3400 2600 3400 3000 3400 3000 3600 3000 3600 3000 3600 3200 3600 3200 3800 3200 3800 3200 3800 3600 3800 3600 4200 3600 4200 3600 4200 3800 4200 3800 4400 3800 4400 3800 4400 4000 4400 4000 5200 4000
% 
\special{pn 8}%
\special{pa 3200 2000}%
\special{pa 3200 2600}%
\special{fp}%
\special{pa 3200 2600}%
\special{pa 3400 2600}%
\special{fp}%
\special{pa 3400 2600}%
\special{pa 3400 3000}%
\special{fp}%
\special{pa 3400 3000}%
\special{pa 3600 3000}%
\special{fp}%
\special{pa 3600 3000}%
\special{pa 3600 3200}%
\special{fp}%
\special{pa 3600 3200}%
\special{pa 3800 3200}%
\special{fp}%
\special{pa 3800 3200}%
\special{pa 3800 3600}%
\special{fp}%
\special{pa 3800 3600}%
\special{pa 4200 3600}%
\special{fp}%
\special{pa 4200 3600}%
\special{pa 4200 3800}%
\special{fp}%
\special{pa 4200 3800}%
\special{pa 4400 3800}%
\special{fp}%
\special{pa 4400 3800}%
\special{pa 4400 4000}%
\special{fp}%
\special{pa 4400 4000}%
\special{pa 5200 4000}%
\special{fp}%
% STR 2 0 3 0 Black Black  
% 4 4280 4070 4280 4170 2 0 0 0
% $i$
\put(42.8000,-41.7000){\makebox(0,0)[lb]{$i$}}%
% VECTOR 2 0 3 0 Black Black  
% 4 3200 3100 3600 3100 3600 3100 3200 3100
% 
\special{pn 8}%
\special{pa 3200 3100}%
\special{pa 3600 3100}%
\special{fp}%
\special{sh 1}%
\special{pa 3600 3100}%
\special{pa 3533 3080}%
\special{pa 3547 3100}%
\special{pa 3533 3120}%
\special{pa 3600 3100}%
\special{fp}%
\special{pa 3600 3100}%
\special{pa 3200 3100}%
\special{fp}%
\special{sh 1}%
\special{pa 3200 3100}%
\special{pa 3267 3120}%
\special{pa 3253 3100}%
\special{pa 3267 3080}%
\special{pa 3200 3100}%
\special{fp}%
% STR 2 0 3 0 Black Black  
% 4 3240 3200 3240 3300 2 0 0 0
% $\DD{i}$
\put(32.4000,-33.0000){\makebox(0,0)[lb]{$\DD{i}$}}%
% BOX 2 5 3 0 Black Black  
% 2 3000 1870 5400 4270
% 
\special{pn 8}%
\special{pa 3000 1870}%
\special{pa 5400 1870}%
\special{pa 5400 4270}%
\special{pa 3000 4270}%
\special{pa 3000 1870}%
\special{ip}%
% LINE 2 2 3 0 Black Black  
% 4 4300 4000 4300 3100 4300 3100 3200 3100
% 
\special{pn 8}%
\special{pa 4300 4000}%
\special{pa 4300 3100}%
\special{dt 0.045}%
\special{pa 4300 3100}%
\special{pa 3200 3100}%
\special{dt 0.045}%
\end{picture}}%
\caption{The pictorial meaning of $\DD{i}$.}
\label{pic:def of D(i)}
\end{figure}

\begin{lemma}\label{lem: addition D and wh}
$\DD{i}<\w{h}^{-1}(j)$ for $1\le i\le j\le n$.
\end{lemma}

\begin{proof}
If $\DD{i}=0$, then the claim is obvious. Thus we may assume $\DD{i}\ge 1$.
It suffices to show that we have $\w{h}(l)<i$ for all $1\le l\le \DD{i}$. 
Suppose that $1\le l\le \DD{i}$.
Since we are assuming $\DD{i}\ge 1$, it is clear that we have $h(l)<i$. This implies that $\w{h}(l)\le h(l)<i$, as desired.
\end{proof}

Suppose that $s_i(\xi_h)=\xi_h$. Then we have 
\begin{align*}
h(i)-h(i+1)+2-h^*(n+1-i)+h^*(n-i) = 0, 
\end{align*} 
which is equivalent to
\begin{align}\label{eq: vanishing coefficient}
h(i)-h(i+1)+2+\DD{i}-\DD{i+1} = 0.
\end{align} 
This condition and Lemma~\ref{cor: nef condition} impose a strong restriction on the shape of $h$ as we observe in what follows.
We first consider the following relation among positions of $1\le i<n$ for which we have $h(i+1)=h(i)$.

\begin{lemma}\label{lem: principle of similar shapes 1}
Let $h\colon[n]\rightarrow[n]$ be a nef Hessenberg function, and suppose that $s_i(\xi_h)=\xi_h$.
If $h(i+1)=h(i)$, then $h(k+2)=h(k+1)$, where $k=\DD{i}$.
\end{lemma}

\begin{proof}
If $h(i+1)=h(i)$, then we have $\DD{i+1}-\DD{i} = 2$ since $s_i(\xi_h)=\xi_h$. This means that $h(k+2)=h(k+1)$ by the pictorial meaning of $k=\DD{i}$. 
\end{proof}

As the converse of Lemma 5.4, we obtain the following.
\begin{lemma}
Let $h \colon [n] \rightarrow [n]$ be a nef Hessenberg function, and $k = D(i)$ for some $1 \le i \le n$. If $k \ge 1$ and $h(k+2) = h(k+1)$, then we have
\[
{\rm (1)}\ D(i+1) = D(i) +2,\ {\it or}\ {\rm (2)}\ D(i+2) = D(i+1) +2=D(i) +2.
\]
In case {\rm (1)}, if $s_i (\xi_h) = \xi_h$ in addition, then $h(i+1) = h(i)$. In case {\rm (2)}, if $s_{i+1} (\xi_h) = \xi_h$ in addition, then $h(i+2) = h(i+1)$.
\end{lemma}

\begin{proof}
The former claim follows by the pictorial meaning of $\DD{i}$ and Lemma~\ref{cor: nef condition} (2).
The latter claim is a direct consequence of \eqref{eq: vanishing coefficient}.
\end{proof}

We next consider the following relation among positions of $1\le i<n$ for which we have $h(i+1)=h(i)+2$.

\begin{lemma}\label{lem: principle of similar shapes 2}
Let $h\colon[n]\rightarrow[n]$ be a nef Hessenberg function, and suppose that $s_i(\xi_h)=\xi_h$.
If $\DD{i}\ge1$ and $h(i+1)=h(i)+2$, then $h(k+1)=h(k)+2$, where $k=\DD{i}$.
\end{lemma}

\begin{proof}
Note first that $i>1$ since $\DD{1}=0$.
Since $s_i(\xi_h)=\xi_h$, the assumption $h(i+1)=h(i)+2$ means that $\DD{i}=\DD{i+1}$. Corollary \ref{cor: nef condition} now implies that we must have 
$\DD{i-1}<\DD{i}=\DD{i+1}<\DD{i+2}$, which means that $h(k+1)=h(k)+2$ by the pictorial meaning of $k=\DD{i}$.
\end{proof}

As the converse of Lemma~\ref{lem: principle of similar shapes 2}, we have the following claim.

\begin{lemma}\label{lem: principle of similar shapes 3}
Let $h\colon[n]\rightarrow[n]$ be a nef Hessenberg function, and $k=\DD{i}$ for some $1\le i\le n$.
If $k\ge1$ and $h(k+1)=h(k)+2$, then we have either 
\begin{align*}
\text{$(1)$ $\DD{i}=\DD{i+1}$, \  or \ $(2)$ $\DD{i-1}=\DD{i}$.}
\end{align*}
In case $(1)$, if $s_i(\xi_h)=\xi_h$ in addition, then $h(i+1)=h(i)+2$. In case $(2)$, if $s_{i-1}(\xi_h)=\xi_h$ in addition, then $h(i)=h(i-1)+2$.
\end{lemma}

\begin{proof}
The former claim is obvious by the pictorial meaning of $\DD{i}$.
The latter claim follows immediately by \eqref{eq: vanishing coefficient}.
\end{proof}

Let $I=[a-1,b]\subseteq [n]$ for some $1< a < b \le n$, and suppose that $s_i(\xi_h)=\xi_h$ for all $a-1\le i\le b-1$. 
Then the four lemmas above imply that if $\DD{a}\ge1$, $h(a)\ne h(a-1)+2$, $h(b) \neq h(b-1) +2$, then the information in what order the positions $i$ satisfying $h(i+1)=h(i)$ and the positions $j$ satisfying $h(j+1)=h(j)+2$ appear must be the same for the intervals $[a,b]$ and $[\DD{a},\DD{b}]$. 
Here, we need $h(a) \ne h(a-1)+2$ because of Lemma~\ref{lem: principle of similar shapes 3}. We also assume $h(b)\ne h(b-1)+2$ because of Lemma~\ref{lem: principle of similar shapes 2}. We call this \textbf{the principle of similar shapes} on $[a,b]$ and $[\DD{a},\DD{b}]$. We use the word ``similar" because we ignore the information how the positions $k$ satisfying $h(k+1)=h(k)+1$ appear when we consider this principle.
For example, if we take $[a,b]=\{10,11,12,13\}$ in the running example for $n=20$, then the shape of $h$ on $[a,b]$ and that of $h$ on $[\DD{a},\DD{b}]=\{1,2,3,4,5,6\}$ are similar in this sense. 

\begin{remark}\label{rem: the modified principle}
\normalfont{If $a= h(1)$, then we have $\DD{a}=0$. In this case, however, the principle of similar shapes on the intervals $[h(1),b]$ and $[1,\DD{b}]$ is valid if $s_i(\xi_h)=\xi_h$ for all $h(1)-1\le i\le b-1$ and $h(b)\ne h(b-1)+2$ without the assumptions $\DD{a}\ge1$ and $h(a)\ne h(a-1)+2$. This follows because we have $h(h(1)+1)\ne h(h(1))+2$ by $s_{h(1)}(\xi_h)=\xi_h$ and $h^*(n+1-h(1))>h^*(n-h(1))$. We need to treat this case as well later.}
\end{remark}

\begin{lemma}\label{lem: D and +}
Let $h\colon[n]\rightarrow[n]$ be a nef Hessenberg function.
If $[i-1,\LL{i}]\subseteq J_k$ for some $1\le k\le n-1$ and $h(i)\ne h(i-1)+2$, then $\DD{i^{(+)}}=\DD{i}^{(+)}$.
\end{lemma}

\begin{proof}
To begin with, we show that $\DD{i}\ge1$. If $\DD{i}=0$, then we also have $\DD{i-1}=0$, which means that $h^*(n+1-i)=h^*(n+1-(i-1))(=n)$. Since we have $[i-1,i]\subset [i-1,\LL{i}] \subseteq J_k$, this implies that $h(i)=h(i-1)+2$, which is a contradiction to our assumption. 

Let us prove that $D(i^{(+)}) = D(i)^{(+)}$. We first consider the case that $i$ satisfies condition \eqref{eq:def of L 10}. In this case, we have $i^{(+)}=i$ by definition, which means that
\begin{align*}
&h(j)=h(j-1)+1 \quad (i+1\le j< \LL{i}), \\
&h(\LL{i})=h(\LL{i}-1).
\end{align*}
The assumptions $[i-1,\LL{i}]\subseteq J_k$ and $h(i)\ne h(i-1)+2$ mean that we may apply the principle of similar shapes to $[i,\LL{i}]$ and $[\DD{i},\DD{\LL{i}}]$, and then the above equalities imply that
\begin{align*}
&h(l)=h(l-1)+1 \quad (\DD{i}+1\le l< m), \\
&h(m)=h(m-1),
\end{align*}
where $m=\DD{\LL{i}-1}+2$ by Lemma~\ref{lem: principle of similar shapes 1}. This means that $\DD{i}$ satisfies condition \eqref{eq:def of L 10} after replacing $i$ by $\DD{i}$. Thus it follows that $\DD{i}^{(+)}=\DD{i}$, and we obtain $\DD{i^{(+)}}=\DD{i}=\DD{i}^{(+)}$, as desired. 

Next, we consider the case that $i$ satisfies condition \eqref{eq:def of L 20}. By the assumption, we have $[i-1,i^{(+)}]\subset [i-1,\LL{i}] \subseteq J_k$, and hence we may apply the principle of similar shapes on $[i,i^{(+)}]$ and $[\DD{i},\DD{i^{(+)}}]$ in a way similar to above, and we see that $\DD{i}$ satisfies condition \eqref{eq:def of L 20} as well. Noticing this, it is straightforward to verify $\hat{\DD{i}} = \DD{\hat{i}}$. Hence the desired claim $\DD{i^{(+)}}=\DD{i}^{(+)}$ is equivalent to
\begin{align*}
\DD{\w{h}^{-1} (h(\hat{i})+1)} = \w{h}^{-1} (h(\DD{\hat{i}})+1).
\end{align*}
In addition, we have 
$
h(\hat{i}) \ne h(\hat{i}-1)+2
$
since the equality $h(\hat{i}) = h(\hat{i}-1)+2$ implies with the minimality of $\hat{i}\ (\ge i)$ that $h(i) = h(i-1)+2$, which contradicts our assumption.
Thus we may assume $i=\hat{i}$, that is, $h(i+1)=h(i)+2$, to prove $\DD{i^{(+)}}=\DD{i}^{(+)}$ in what follows. 
Notice that 
\begin{align}\label{eq: D and + 10}
h(\DD{i}+1)=h(\DD{i})+2,
\end{align}
which follows by $h(i+1)=h(i)+2$ and Lemma~\ref{lem: principle of similar shapes 2}. Also, since we have $h(i^{(+)})=h(i^{(+)}-1)$, it follows that
\begin{align}\label{eq: D and + 20}
h(\DD{i^{(+)}})=h(\DD{i^{(+)}}-1)
\end{align}
by Lemma~\ref{lem: principle of similar shapes 1} and $\DD{i^{(+)}-1}=\DD{i^{(+)}}-2$. See Figure \ref{pic:D(i) and D(ihat)} which visualizes the equalities \eqref{eq: D and + 10} and \eqref{eq: D and + 20}.
\begin{figure}[htbp]
%WinTpicVersion4.32a
{\unitlength 0.1in%
\begin{picture}(38.0000,22.2000)(22.0000,-35.0000)%
% CIRCLE 2 0 0 0 Black Black  
% 4 3700 1900 3720 1920 3720 1920 3720 1920
% 
\special{sh 1.000}%
\special{ia 3700 1900 28 28 0.0000000 6.2831853}%
\special{pn 8}%
\special{ar 3700 1900 28 28 0.0000000 6.2831853}%
% LINE 2 2 3 0 Black Black  
% 2 3500 1350 3500 3150
% 
\special{pn 8}%
\special{pa 3500 1350}%
\special{pa 3500 3150}%
\special{dt 0.045}%
% LINE 2 2 3 0 Black Black  
% 2 4700 1340 4700 3140
% 
\special{pn 8}%
\special{pa 4700 1340}%
\special{pa 4700 3140}%
\special{dt 0.045}%
% STR 2 0 3 0 Black White  
% 4 3320 3330 3320 3430 2 0 0 0
% $\DD{i}$
\put(33.2000,-34.3000){\makebox(0,0)[lb]{$\DD{i}$}}%
% STR 2 0 3 0 Black White  
% 4 4420 3330 4420 3430 2 0 0 0
% $\DD{i^{(+)}}$
\put(44.2000,-34.3000){\makebox(0,0)[lb]{$\DD{i^{(+)}}$}}%
% LINE 2 0 3 0 Black White  
% 4 4800 2800 4400 2800 4400 2800 4400 2600
% 
\special{pn 8}%
\special{pa 4800 2800}%
\special{pa 4400 2800}%
\special{fp}%
\special{pa 4400 2800}%
\special{pa 4400 2600}%
\special{fp}%
% CIRCLE 2 0 0 0 Black Black  
% 4 4500 2700 4520 2720 4520 2720 4520 2720
% 
\special{sh 1.000}%
\special{ia 4500 2700 28 28 0.0000000 6.2831853}%
\special{pn 8}%
\special{ar 4500 2700 28 28 0.0000000 6.2831853}%
% STR 2 0 3 0 Black White  
% 4 4000 2250 4000 2350 2 0 0 0
% $\ddots$
\put(40.0000,-23.5000){\makebox(0,0)[lb]{$\ddots$}}%
% LINE 2 0 3 0 Black Black  
% 2 4800 2800 4800 3000
% 
\special{pn 8}%
\special{pa 4800 2800}%
\special{pa 4800 3000}%
\special{fp}%
% LINE 2 0 3 0 Black Black  
% 6 3400 1600 3600 1600 3600 1600 3600 2000 3600 2000 3800 2000
% 
\special{pn 8}%
\special{pa 3400 1600}%
\special{pa 3600 1600}%
\special{fp}%
\special{pa 3600 1600}%
\special{pa 3600 2000}%
\special{fp}%
\special{pa 3600 2000}%
\special{pa 3800 2000}%
\special{fp}%
% BOX 2 5 3 0 Black Black  
% 2 2200 1280 6000 3500
% 
\special{pn 8}%
\special{pa 2200 1280}%
\special{pa 6000 1280}%
\special{pa 6000 3500}%
\special{pa 2200 3500}%
\special{pa 2200 1280}%
\special{ip}%
\end{picture}}%
\caption{The picture of \eqref{eq: D and + 10} and \eqref{eq: D and + 20}.}
\label{pic:D(i) and D(ihat)}
\end{figure}
From these, it suffices to show that
$\w{h}$ takes $h(\DD{i})+1$ as its value at $\DD{i^{(+)}}$ since this is equivalent to the desired claim $\DD{i^{(+)}}=\DD{i}^{(+)}$ by the definition of the $(+)$-operation. 
For $1\le k< \ell \le n$, we define
\begin{align*}
&\stable[k,\ell] \coloneqq |\{c\in[k,\ell-1] \mid h(c+1)=h(c) \}|, \\
&\twosteps[k,\ell] \coloneqq |\{c\in[k,\ell-1] \mid h(c+1)=h(c)+2 \}|.
\end{align*}
Then, by the definition of the $(+)$-operation and the maximality of the values of $\w{h}$, we have
\begin{equation}\label{eq: D and + 30}
\begin{split}
&\twosteps[i,j]> \stable[i,j] \quad \text{for each $i< j< i^{(+)}$,\ {\rm and}} \\
&\twosteps[i,i^{(+)}] = \stable[i,i^{(+)}].
\end{split}
\end{equation}
By the principle of similar shapes on $[i,i^{(+)}]\subseteq J_k$ and $[\DD{i},\DD{i^{(+)}}]$, we see that \eqref{eq: D and + 30} still hold when we replace $i$ and $i^{(+)}$ by $\DD{i}$ and $\DD{i^{(+)}}$, respectively. Namely, we have
\begin{align*}
&\twosteps[\DD{i},k] > \stable[\DD{i},k] \quad \text{for each $\DD{i}< k< \DD{i^{(+)}}$,\ {\rm and}}, \\
&\twosteps[\DD{i},\DD{i^{(+)}}] = \stable[\DD{i},\DD{i^{(+)}}].
\end{align*}
This in particular implies that $w_h (D(i^{(+)})) = h(D(i)) +1$ by the maximality of $\w{h}$, as desired. 
\end{proof}

\begin{lemma}\label{lem: D and L}
Let $h\colon[n]\rightarrow[n]$ be a nef Hessenberg function.
If $[i-1,\LL{i}]\subseteq J_k$ for some $1\le k\le n-1$ and $h(i)\ne h(i-1)+2$, then $\DD{\LL{i}}=\LL{{\DD{i}}}$.
\end{lemma}

\begin{proof}
By the assumption, the previous lemma shows that $\DD{i^{(+)}}=\DD{i}^{(+)}$. By taking $(+)$ on both sides, we obtain $\DD{i^{(+)}}^{(+)}=\DD{i}^{(+)(+)}$. Note that we have $[i^{(+)}-1,\LL{i^{(+)}} ]\subseteq[i-1,\LL{i}]\subseteq J_k$, and $h(i^{(+)})\ne h(i^{(+)}-1)+2$. Here, the latter claim follows because if $i=i^{(+)}$, then the claim is precisely the assumption $h(i)\neq h(i-1)+2$, and if  $i<i^{(+)}$, then $h$ is stable at $i^{(+)}$, that is, $h(i^{(+)})=h(i^{(+)}-1)$, which implies the claim.
Thus we obtain $\DD{i^{(+)(+)}}=\DD{i^{(+)}}^{(+)}$ by the previous lemma. Combining this with the previous equality above, we obtain
\begin{align*}
\DD{i^{(+)(+)}}=\DD{i}^{(+)(+)}.
\end{align*}
By continuing this process sufficiently many times, we obtain
\begin{align}\label{eq: D and L 10}
\DD{i^{(+\infty)}}=\DD{i}^{(+\infty)}.
\end{align}
We also have
\begin{align*}
h(i^{(+\infty)})\ne h(i^{(+\infty)}-1)+2
\end{align*}
by an argument similar to that above. Thus, by \eqref{eq: D and L 10} and $[i^{(+\infty)}-1,\LL{i^{(+\infty)}}]\subseteq [i-1,\LL{i}]\subseteq J_k$, we may assume $i=i^{(+\infty)}$ to prove $\DD{\LL{i}}=\LL{\DD{i}}$ in what follows.

Since we have $i=i^{(+\infty)}$, we know that $i$ satisfies condition \eqref{eq:def of L 10}, which means that we have 
\begin{align*}
&h(j)=h(j-1)+1 \quad (i+1\le j< \LL{i}),\ {\rm and} \\
&h(\LL{i})=h(\LL{i}-1).
\end{align*}
Because of \eqref{eq: D and L 10}, we also have $\DD{i}=\DD{i}^{(+\infty)}$. Thus $\DD{i}$ also satisfies condition \eqref{eq:def of L 10}, which means that we have 
\begin{align*}
&h(k)=h(k-1)+1 \quad (\DD{i}+1\le k< \LL{\DD{i}}), \\
&h(\LL{\DD{i}})=h(\LL{\DD{i}}-1).
\end{align*}
Thus, by $[i-1,\LL{i}]\subseteq J_k$ and Lemma~\ref{lem: principle of similar shapes 1}, it follows that
\begin{align*}
\DD{\LL{i}-1}+1=\LL{\DD{i}}-1.
\end{align*}
Since we have $h(\LL{i})=h(\LL{i}-1)$ and $[\LL{i}-1,\LL{i}]\subseteq [i-1,\LL{i}]\subseteq J_k$, it follows that $\DD{\LL{i}-1}=\DD{\LL{i}}-2$ by \eqref{eq: vanishing coefficient}. Combining this with the above equality, we obtain $\DD{\LL{i}}=\LL{\DD{i}}$, as desired.
\end{proof}

%%%%%%%%%%%%%%%%%%
\subsection{Proof of Theorem~B}\label{subsect: Proof of Thm B}
%%%%%%%%%%%%%%%%%%
Let $h\colon[n]\rightarrow [n]$ be a nef Hessenberg function satisfying the assumption \eqref{eq: assumption on h}, that is, $h(i)\geq i+1$ for $1\le i<n$.
In this subsection, we give a proof of Theorem~B which is stated in Section \ref{sec: intro}.

We already established the equivalence of (i) and (iii) in Theorem~B by Proposition~\ref{prop: equiv condition for nef}. Recalling that the anti-canonical bundle of $\Hess{S}{h}$ is isomorphic to $\LB_{\xi_h}$ by Proposition~\ref{prop: anti-can of Hess}, it suffices to prove that if $\LB_{\xi_h}$ on $\Hess{S}{h}$ is nef, then it is in fact big.
By Proposition~\ref{cor: using Richardson 20} and Lemma~\ref{cor: using Richardson 30} together with the notations given in Section \ref{subsec: preliminary notations}, it is enough to show that there exists $u \in \mathfrak{S}_{J}$ such that $\ell(u)+\ell(\w{h})=\ell (u\w{h})$ and $u_{J}=(u\w{h})_{J}$. 

Our proof is induction on $n$. To control induction, we require two additional conditions as seen below. Namely, we prove the following, where we say that $h$ is \textbf{strictly increasing} on an interval $[a,b]\subseteq [n]$ $(a<b)$ if 
\[
h(a)<h(a+1)<\cdots< h(b).
\]

\begin{theorem}
Let $h\colon[n]\rightarrow [n]$ be a nef Hessenberg function satisfying \eqref{eq: assumption on h}.
Then there exists $u \in \mathfrak{S}_J$ such that the following conditions hold:
\begin{enumerate}
\item[{\rm (i)}] $\ell (u\w{h})=\ell(u)+\ell(\w{h})$;
\item[{\rm (ii)}] $(u\w{h})_J=u_J$; 
\item[{\rm (iii)}] 
if $s_i (\xi_h) = \xi_h$ and $h$ is strictly increasing on $J_i$, then 
$u(j) = j$ for all $j\in J_i$;
\item[{\rm (iv)}] for $1 \le i \le n-1$ and $i +1 \le j < \LL{i} $, we have $u \w{h} (i) < u \w{h} (j)$.
\end{enumerate}
\end{theorem}

\begin{remark}
\normalfont{
In addition to the original conditions {\rm (i)} and {\rm (ii)}, we require two additional conditions {\rm (iii)} and {\rm (iv)} by the following reasons. If $h$ is strictly increasing on some $J_i$, then $w_h$ is also strictly increasing on $J_i$. Hence, by condition (ii), it is natural to seek for $u \in \mathfrak{S}_J$ under condition (iii) on $J_i$, that is, $u$ is the identity on $J_i$. This condition will be used in the proof of Lemma $\ref{l:fixed_part_by_u_prime}$. Condition (iv) is inspired by Lemma $\ref{lem: wh and L}$. This condition will be used to control the positions of $(h(1) -k)^{(+)}$ and $L_{h(1) -k}$ for $0 \le k \le k_{h(1), -}$ in the proofs of Lemmas \ref{l:computation_of_w(h)_for_h(1)-k} and \ref{lem: L and q}.}
\end{remark}

We proceed by induction on $n$. 
When $n = 2$, then the assumption \eqref{eq: assumption on h} implies that we must have $h=(2,2)$, which shows that $J=\emptyset$, so that the assertion is obvious by taking $u=e$. Let $n\ge 3$.
If $h(1)=n$, then we must have $h(i)=n$ for all $1\le i\le n$. In this case, we have $J=\emptyset$, and the assertion is obvious. Hence we may assume $h(1)<n$ in what follows so that $s_{h(1)}\in\mathfrak{S}_n$ makes sense. 

Define a function $h' \colon \{1, 2, \ldots, n-1\} \rightarrow \{1, 2, \ldots, n-1\}$ by 
\[
h' (i) \coloneqq h(i +1) -1 \qquad (1 \le i \le n-1).
\]
Then it is a Hessenberg function, and it is obtained from $h$ by removing all the boxes in the $h(1)$-st row and those in the $1$-st column. See the running example in Section \ref{subsec: Example}.
Note that we have
\[h^{\prime \ast} (i) = \begin{cases}
h^\ast (i) \quad&({\rm if}\ 1 \le i < n +1 -h(1)),\\
h^\ast (i) -1\ (= n -1)\quad&({\rm if}\ n +1 -h(1) \le i \le n-1).
\end{cases}\]
By the definition of $\w{h}$ and $\w{h'}$, we see that 
\begin{align}\label{eq: wh in terms of wh'}
\w{h} (i) = 
\begin{cases}
h(1) \quad&({\rm if}\ i = 1),\\
\w{h'} (i -1) \quad&({\rm if}\ \w{h'} (i -1) < h(1)),\\
\w{h'} (i -1) +1 \quad&({\rm if}\ \w{h'} (i -1) \ge h(1)),
\end{cases}
\end{align} 
which implies that
\begin{align}\label{eq: additional 330}
\ell(\w{h}) = \ell (\w{h'}) + (h(1)-1).
\end{align}
Recalling that we are assuming $h(1)<n$, we have the following.

\begin{lemma}\label{l:calculation_of_xi_for_h_prime}
Write $\xi_h = \sum_{i = 1} ^{n -1} d_i \varpi_i$. Then $\xi_{h'}$ can be written as follows: 
\[\xi_{h'} = \varpi_{h(1) -1} +\sum_{i = 1} ^{n -2} d_{i +1} \varpi_i.\]
In particular, $h'$ is also nef. 
\end{lemma}

\begin{proof}
Writing $\xi_{h'} = \sum_{i = 1} ^{n -2} d'_i \varpi_i$, we have 
\begin{align*}
d'_i &= h'(i) - h'(i+1) + 2 - h'^*(n-i) + h'^*(n-1-i).
\end{align*}
By the definition of $h'$ and the description of $h'^*$ above, we can rewrite this as follows.
If $i< h(1)-1$ or $i> h(1)-1$, then $d'_i$ is equal to 
\begin{align*}
h(i+1) - h(i+2) + 2 - h^*(n-i) + h^*(n-1-i),
\end{align*}
which is $d_{i+1}$.
If $i=h(1)-1\ (\le n-2)$, then $d'_i$ is equal to 
\begin{align*}
h(i+1) - h(i+2) + 2 - h^*(n-i) + h^*(n-1-i)+1,
\end{align*} 
which is $d_{h(1)}+1$. This proves the claim.
\end{proof}

By Lemma~\ref{l:calculation_of_xi_for_h_prime}, if $s_i (\xi_{h'}) = \xi_{h'}$ for some $1 \le i \le n-2$, then $s_{i +1} (\xi_h) = \xi_h$. This means that under the injective group homomorphism $\iota \colon \mathfrak{S}_{n-1} \hookrightarrow \mathfrak{S}_n$ given by $s_i \mapsto s_{i+1}$ for $1\le i\le n-2$, we have 
\begin{equation}\label{eq:inclusion_of_parabolic_subgroup}
\begin{aligned}
\mathfrak{S}_{J'} \hookrightarrow \mathfrak{S}_J, 
\end{aligned}
\end{equation}
where $J=J(h)$ and $J'=J(h')$.
By induction hypothesis, there exists $u' \in \mathfrak{S}_{J'}$ such that conditions {\rm (i)}--{\rm (iv)} hold. 
We denote by $\bar{u}' \in\mathfrak{S}_J$ the image of $u'$ under the embedding \eqref{eq:inclusion_of_parabolic_subgroup}. 
Namely, we have 
\begin{align}\label{eq: two permutations}
\bar{u}' (i) = 
\begin{cases}
1 \quad&({\rm if}\ i=1),\\
u'(i-1)+1\quad&({\rm if}\ i > 1),
\end{cases}
\end{align}
which implies that
\begin{align}\label{eq: two permutations 2}
u'(i) = \bar{u}' (i+1)-1 \quad (1\le i\le n-1).
\end{align}
Since $h$ is nef, it follows by Lemma~\ref{l:calculation_of_xi_for_h_prime} that $s_{h(1) -1} (\xi_{h'}) \neq \xi_{h'}$, and hence that $s_{h(1) -1}\notin \mathfrak{S}_{J'}$.
We observe that condition (iii) for $u'$ ensures the following property of its image $\bar{u}' $ in $\mathfrak{S}_J$.

\begin{lemma}\label{l:fixed_part_by_u_prime}
The equality $\bar{u}'  (k) = k$ holds for $1 \le k \le h(1)$.
\end{lemma} 

\begin{proof}
Since we have $\bar{u}'  (1) = 1$ by definition, it suffices to prove that
\begin{align}\label{eq: additional 1000}
u' (k) = k \quad (1 \le k \le h(1)-1).
\end{align}
Since $s_{h(1) -1} (\xi_{h'}) \neq \xi_{h'}$, we know that each $J'_i$ is contained in either $\{1, 2, \ldots, h(1) -1\}$ or $\{h(1),h(1)+1,\ldots,n-1\}$, where $J'_i=J_i(h')$.
If there are no $J' _i$ such that $J' _i \subseteq \{1, 2, \ldots, h(1) -1\}$, 
then the claim is obvious since $u'\in\mathfrak{S}_{J'}$. If there exists $1 \le i \le h(1) -1$ such that $J' _i \subseteq \{1, 2, \ldots, h(1) -1\}$, then we have $h^{\prime \ast} (n -j) = n -1$ for all $j \in J' _i$ since $h(1)-1\le h(2)-1= h'(1)$. 
Hence the equalities 
\[
h'(j) - h'(j+1) + 2 - h'^*(n-j) + h'^*(n-1-j) = 0 \quad (i-k' _{i, -} \le j < i+ k' _i +1)
\]
now imply that $h' (j+1) = h' (j) +2$ for $i-k' _{i, -} \le j < i+ k' _i +1$, where $k^\prime _{i, -}$ and $k^\prime _i$ are $k_{i, -}$ and $k_i$ for $h^\prime$, respectively. In particular, we have
\[h' (i -k' _{i, -}) < \cdots < h' (i) < \cdots < h' (i +k' _i +1).\]
Hence we see by condition {\rm (iii)}  for $u'$ that $u' (j) = j \ (j \in J' _i)$. 
Since this holds for all $J'_i\subseteq \{1, 2, \ldots, h(1) -1\}$, \eqref{eq: additional 1000} follows from $u' \in\mathfrak{S}_{J'}$.
\end{proof}

The relation \eqref{eq: wh in terms of wh'} between $\w{h}$ and $\w{h'}$ implies the following relation between $\bar{u}'\w{h}$ and $u' \w{h'}$ in a similar form by Lemma~\ref{l:fixed_part_by_u_prime}.

\begin{corollary}\label{c:computation_of_v_first_step}
The following equalities hold: 
\begin{align*}
\bar{u}'  \w{h} (i) = 
\begin{cases}
h(1) \quad&({\rm if}\ i = 1),\\
u' \w{h'} (i -1) \quad&({\rm if}\ u' \w{h'} (i -1) < h(1)),\\
u' \w{h'} (i -1)+1 \quad&({\rm if}\ u' \w{h'} (i -1) \ge h(1)).
\end{cases}
\end{align*} 
\end{corollary}

\begin{proof}
We compute the values $\bar{u}'  \w{h} (i)$ for $1\le i\le n$.
If $i=1$, then $\bar{u}'  \w{h} (i)=\bar{u}'  (h(1))=h(1)$ by \eqref{eq: wh in terms of wh'} and Lemma~\ref{l:fixed_part_by_u_prime}. Hence we may assume $i>1$ in the following.

If $u' \w{h'} (i -1) < h(1)$, then by \eqref{eq: two permutations 2} we have $\bar{u}'  (\w{h'} (i -1)+1) \le h(1)$ so that 
\begin{align*}
\w{h'} (i -1) < h(1)
\end{align*} 
by Lemma~\ref{l:fixed_part_by_u_prime}. Thus \eqref{eq: wh in terms of wh'} and Lemma~\ref{l:fixed_part_by_u_prime} again show that 
\[
\bar{u}'  \w{h} (i) = \bar{u}'  \w{h'} (i-1) = \w{h'} (i-1). 
\]
But this is further equal to $u' \w{h'} (i -1)$ 
by \eqref{eq: additional 1000}.

If $u' \w{h'} (i -1) \ge h(1)$, then $\bar{u}'  (\w{h'} (i -1)+1)-1 \ge h(1)$ by \eqref{eq: two permutations 2}. Hence it follows that $\w{h'} (i -1)+1\ge h(1)+1$ by Lemma~\ref{l:fixed_part_by_u_prime}, and hence that
\[
\w{h'} (i -1)\ge h(1). 
\]
Thus, by \eqref{eq: wh in terms of wh'} and \eqref{eq: two permutations 2}, we have $\bar{u}'  \w{h} (i) = \bar{u}'  (\w{h'} (i-1)+1) = u' \w{h'} (i-1)+1$, as desired.
\end{proof}

\begin{lemma}\label{l:length_easiest_case}
The equality $\ell(\bar{u}'  \w{h}) = \ell(\bar{u}' ) + \ell(\w{h})$ holds.
\end{lemma} 

\begin{proof}
The claim follows from the following direct computations: 
\begin{align*}
\ell(\bar{u}'  \w{h}) 
&= \ell(u' \w{h'}) + (h(1)-1)  \quad \text{(by Corollary \ref{c:computation_of_v_first_step})} \\
&= \ell(u')+\ell(\w{h'}) + (h(1)-1) \quad \text{(by condition (i) for $u' $)} \\
&= \ell(\bar{u}' )+\ell(\w{h}) \quad \text{(by \eqref{eq: additional 330})}.
\end{align*}
\end{proof}

To construct $u\in\mathfrak{S}_J$ which satisfies conditions (i)-(iv), we now take cases.\\

\noindent \underline{\textbf{Case 1:} $s_{h(1)} (\xi_h) \neq \xi_h$.}\\

In this case, we set $u \coloneqq \bar{u}' \in\mathfrak{S}_J$. Then condition (i) holds for $u$ by Lemma~\ref{l:length_easiest_case}. 
The assumption $s_{h(1)} (\xi_h) \neq \xi_h$ implies the following assertions on $J_i ^\prime$ by Lemma~\ref{l:calculation_of_xi_for_h_prime}. If $s_i (\xi_{h^\prime}) = \xi_{h^\prime}$, then $J_i ^\prime = \{j -1 \mid j \in J_{i +1}\} \setminus \{0\}$. If $s_1 (\xi_h) = \xi_h$, then $J_1 ^\prime$ is defined if and only if $k_1 >0$. In this case, $J_1 = \{j +1 \mid j \in J_1 ^\prime\} \cup \{1\}$.
We will use this observation to prove conditions (ii)-(iv) in the following.

\begin{proposition}\label{p:correction_on_parabolic_completed}
Condition {\rm (ii)}  holds for $u$. That is, the equality $(u \w{h})_J=u_J$ holds.
\end{proposition} 

\begin{proof}
Take $1 \le i \le n -1$ such that $s_i (\xi_h) = \xi_h$. 
It suffices to prove for $j_1, j_2 \in J_i$ that $u \w{h} (j_1) < u \w{h} (j_2)$ if and only if $u (j_1) < u (j_2)$, as we observed in the beginning of Section \ref{subsec: preliminary notations}.

First, we consider the case $1 \notin J_i$. In this case, we have $j_1, j_2\ge 2$, and hence Corollary \ref{c:computation_of_v_first_step} implies that $u \w{h} (j_1) < u \w{h} (j_2)$ if and only if $u' \w{h'} (j_1 -1) < u' \w{h'} (j_2 -1)$. In addition, since $u = \bar{u}' $, we have $u (j_1) < u (j_2)$ if and only if $u' (j_1 -1) < u' (j_2 -1)$ by \eqref{eq: two permutations 2}. From these and condition {\rm (ii)}  for $u'$, we conclude the assertion.

We next consider the case $1 \in J_i$. In this case, we have $J_i = J_1 = \{1, 2, \ldots, \rb{1} +2\}$, and the same argument as above implies that for $j_1, j_2 \in J_1 \setminus \{1\}$, we have $u \w{h} (j_1) < u \w{h} (j_2)$ if and only if $u (j_1) < u (j_2)$. Since $u = \bar{u}'$, it follows from \eqref{eq: two permutations} that 
\[u (1) = 1 < u (j)\]
for all $j \in J_1 \setminus \{1\}$. Thus it suffices to prove that $u \w{h} (1)<u \w{h} (j)$ for all $j \in J_1 \setminus \{1\}$. By Lemma~\ref{l:connected_component_containing_1}, we deduce for $j \in J_1 \setminus \{1\}$ that $h(1) < h(j) = \w{h} (j)$.
Since this means $h(1)<\bar{u}'  \w{h} (j)$
by Lemma~\ref{l:fixed_part_by_u_prime}, it follows by Lemma~\ref{l:fixed_part_by_u_prime} again that
\begin{align*}
u \w{h} (1)
=u (h(1)) 
=h(1)
<\bar{u}'  \w{h} (j) 
=u \w{h} (j),
\end{align*}
as desired.
\end{proof}

\begin{proposition}\label{p:condition_iv_for_u}
Condition {\rm (iii)} holds for $u$.
That is,
if $s_i (\xi_h) = \xi_h$ and $h$ is strictly increasing on $J_i$, then 
$u(j) = j$
for all $j\in J_i$.
\end{proposition} 

\begin{proof}
Note that $i \neq h(1)$ since we are assuming $s_{h(1)} (\xi_h) \neq \xi_h$ in Case 1. We first consider the case $1 \notin J_i$. In this case, we have $i\ge2$, and the assumption $s_i (\xi_h) = \xi_h$ implies $s_{i -1} (\xi_{h'}) = \xi_{h'}$ by Lemma~\ref{l:calculation_of_xi_for_h_prime} since $i\ne h(1)$. In addition, we have
\begin{align*}
\lb{i -1}' = \lb{i},\ \rb{i -1}' = \rb{i},\ J' _{i -1} = \{j -1 \mid j \in J_i\}.
\end{align*}
Hence the assumption
\begin{align*}
h (i -\lb{i}) < \cdots < h (i) < \cdots < h (i +\rb{i} +1)
\end{align*}
means by the definition of $h'$ that
\begin{align*}
h' (i -1 -k' _{i -1, -}) < \cdots < h' (i -1) < \cdots < h' (i -1 +k' _{i -1} +1).
\end{align*}
Thus, by condition {\rm (iii)}  for $u'$, we obtain
\[u (j) = \bar{u}'  (j) =u'(j-1)+1 = (j-1)+1=j\]
for all $j\in J_i$, where the second equality follows from \eqref{eq: two permutations} and $j\ge2$.

We next consider the case $1 \in J_i$. In this case, we have $J_i=J_1=\{1, 2, \ldots, \rb{1}+2\}$.
If $\rb{1}=0$, then the claim is obvious since we have $u (j) = \bar{u}'  (j) = j$ for $1\le j\le 2$ $(\le h(1))$ by Lemma~\ref{l:fixed_part_by_u_prime}.
Hence we may assume that $\rb{1}\ge1$. This means that $s_{2}(\xi_h)=\xi_h$, and hence that $s_{1}(\xi_{h'})=\xi_{h'}$ by Lemma~\ref{l:calculation_of_xi_for_h_prime} 
since the assumption $s_{h(1)} (\xi_h) \neq \xi_h$ implies $h(1) \neq 2$.
In particular, $J'_1$ is defined, and we have $J'_1=\{1, 2, \ldots, \rb{1}'+2\}=\{1, 2, \ldots, \rb{1}+1\}$.
Thus, by an argument similar to that above, we obtain $u (j) = j$ for $j\in J_i\setminus\{1\}$. 
Since we also have $u (1) = \bar{u}'  (1) = 1$ by \eqref{eq: two permutations}, it follows that $u (j) = j$ for all $j\in J_i$.
\end{proof}

\begin{proposition}\label{p:condition_v_easy_case}
Condition {\rm (iv)} holds for $u$.
That is, for $1 \le i \le n-1$ and $i +1 \le j < \LL{i} $, we have
$u \w{h} (i) < u \w{h} (j)$.
\end{proposition} 

\begin{proof}
If $i \neq 1$, then we have $\LL{i} = \LL{i-1}^\prime +1$, where $\LL{i}=\LL{i}(h)$ and $\LL{i-1}^\prime=\LL{i-1}(h')$.
Hence the assumption $i +1 \le j < \LL{i} $ means that $(i-1)+1\le j-1 < \LL{i-1}^\prime$, and we see by condition {\rm (iv)}  for $u'$ that 
\[u' \w{h'} (i -1) < u' \w{h'} (j -1),\]
which implies by Corollary \ref{c:computation_of_v_first_step} that 
$\bar{u}'  \w{h} (i) < \bar{u}'  \w{h} (j).$

If $i = 1$, then the assumption $i +1 \le j < \LL{i} $ implies that 
$h(1)=\w{h} (1)< \w{h} (j)$ by Lemma~\ref{lem: wh and L}. Hence it follows from Lemma~\ref{l:fixed_part_by_u_prime} that
\begin{align*}
u \w{h} (1) = h(1) < u \w{h} (j),
\end{align*}
as desired.
\end{proof}

\noindent \underline{\textbf{Case 2:} $s_{h(1)} (\xi_h) = \xi_h$.}\\

Let us write $\lp{h} \coloneqq \lb{h(1)}$ and $\rp{h} \coloneqq \rb{h(1)}$ for simplicity, that is, 
\begin{align*}
J_{h(1)} = \{h(1) -\lp{h}, \ldots, h(1) -1, h(1), h(1) +1, \ldots, h(1) +\rp{h} +1\}.
\end{align*}
We first consider the orders of the numbers of $\bar{u}' $ and $\bar{u}' \w{h}$ on $J_{h(1)}$ in one-line notation. We  set $J_{h(1)}^-\coloneqq \{h(1) -\lp{h}, \ldots, h(1)-1, h(1)\}$ and $J_{h(1)}^+\coloneqq \{h(1)+1, h(1)+2 , \ldots, h(1) +\rp{h} +1\}$ so that 
\[J_{h(1)} = J_{h(1)}^-\sqcup J_{h(1)}^+.\]
Since $\bar{u}'\in\mathfrak{S}_{J}$, it preserves the subset $J_{h(1)}$ of $[n]$. By Lemma~\ref{l:fixed_part_by_u_prime}, it also preserves the smaller subset $J_{h(1)}^-$. Since $J_{h(1)}^+$ is the complement of $J_{h(1)}^-$ in $J_{h(1)}$, we see that
\begin{align}\label{eq: preserving - and +}
\bar{u}'  (J_{h(1)}^-) = J_{h(1)}^-, \quad
\bar{u}'  (J_{h(1)}^+) = J_{h(1)}^+.
\end{align}
For example, we have $J_{h(1)}^-=\{7,8,9\}$ and $J_{h(1)}^+=\{10,11,\ldots,14\}$ in the running example of $n=20$.

We now use condition (ii) for $u'$ to study the orders of the numbers of $\bar{u}'$ and $\bar{u}' \w{h}$ on $J_{h(1)}$. 
Since $J'_{h(1)} = \{h(1), h(1)+1,\ldots, h(1) +\rp{h}\}$ if $\rp{h}>0$, we have
\begin{align}\label{eq: orders on J+}
\text{$\bar{u}' (j_1) < \bar{u}' (j_2)$ if and only if $\bar{u}' \w{h} (j_1) < \bar{u}' \w{h} (j_2)$} \qquad (j_1, j_2\in J_{h(1)}^+)
\end{align}
from condition (ii) for $u'$ on $J'_{h(1)}$ (cf.\ the proof of Proposition~\ref{p:correction_on_parabolic_completed}). 
Similarly (but from a slightly complicated argument as we explain below), it follows that
\begin{align}\label{eq: orders on J-}
\text{$\bar{u}' (j_1) < \bar{u}' (j_2)$ if and only if $\bar{u}' \w{h} (j_1) < \bar{u}' \w{h} (j_2)$} \qquad (j_1, j_2\in J_{h(1)}^-).
\end{align}
We prove this as follows.
By Lemma~\ref{l:fixed_part_by_u_prime}, 
it suffices to show 
\begin{align}\label{eq: order of u'bar wh}
\bar{u}' \w{h} (h(1)-\lp{h}) < \ \cdots < \bar{u}' \w{h} (h(1)-1) < \bar{u}' \w{h} (h(1)).
\end{align}
We may assume that $\lp{h}\ge 1$ since if not, then there is nothing to prove. Recall from the definition of $J_{h(1)}$ that we have $s_{h(1)-\lp{h}}\in\mathfrak{S}_J$. We now take cases.
If $h(1) -\lp{h}\ge2$, then the property $s_{h(1)-\lp{h}}\in\mathfrak{S}_J$ implies that $s_{h(1)-\lp{h}-1}\in\mathfrak{S}_{J'}$. This means that $J'_{h(1)-\lp{h}-1} = \{h(1)-\lp{h}-1, \ldots, h(1)-2, h(1)-1\}$ is defined, where we used $s_{h(1)-1}\notin\mathfrak{S}_{J'}$. Thus Lemma~\ref{l:fixed_part_by_u_prime} and condition (ii) for $u'$ on $J'_{h(1)-\lp{h}-1}$ imply \eqref{eq: order of u'bar wh} in this case. 
If $h(1) -\lp{h}=1$, then we have $h(1)<h(2)$ 
by Lemma~\ref{l:connected_component_containing_1}, and this implies that $\w{h}(1)=h(1)<h(2)=\w{h}(2)$. Hence Lemma~\ref{l:fixed_part_by_u_prime} shows that $\bar{u}' \w{h}(1) <\bar{u}' \w{h}(2)$.
This means that if $h(1)=2$, then we already have \eqref{eq: order of u'bar wh}, and hence we may assume $h(1)\ge3$ in what follows.
We then have $h(1)-\lp{h}=1< 2< h(1)$, and hence the definition of $J_{h(1)}$ implies that $s_2(\xi_{h})=\xi_{h}$, which means that $s_1(\xi_{h'})=\xi_{h'}$ so that $J'_{1} = \{1,2, \ldots, h(1)-2, h(1)-1\}$ is defined. Thus the same argument as above implies that
\begin{align*}
\bar{u}' \w{h} (2) < \bar{u}' \w{h} (3) < \ \cdots < \bar{u}' \w{h} (h(1)-1) < \bar{u}' \w{h} (h(1)).
\end{align*}
Combining this with $\bar{u}' \w{h}(1) <\bar{u}' \w{h}(2)$ proved above, we obtain \eqref{eq: order of u'bar wh} in this case. Hence  \eqref{eq: orders on J-} follows.

Recall that we seek for a permutation $u\in \mathfrak{S}_J$ which satisfies $(u\w{h})_J=u_J$. 
We observed in \eqref{eq: orders on J+} and \eqref{eq: orders on J-} above that the numbers of $\bar{u}' $ and $\bar{u}' \w{h}$ on $J_{h(1)}^{\pm}$ are ordered in the same way.
For $\bar{u}' $, we also have the following property on the whole $J_{h(1)}$:
\begin{align}\label{eq:additional 30}
\bar{u}' (j_1) < \bar{u}' (j_2) \qquad (j_1\in J_{h(1)}^-, \ j_2\in J_{h(1)}^+)
\end{align}
by \eqref{eq: preserving - and +}. If \eqref{eq:additional 30} is also satisfied for $\bar{u}' \w{h}$ (after replacing $\bar{u}'$ by $\bar{u}' \w{h}$), then we may take $u$ to be $\bar{u}' $ as we will see in Case 2-a below, but this is not the case in general. To find the desired permutation $u\in \mathfrak{S}_J$, we encode the information how the numbers of $\bar{u}' \w{h}$ on the whole $J_{h(1)}$ are ordered; in other words, how \eqref{eq:additional 30} is violated for $\bar{u}' \w{h}$ on $J_{h(1)}$. Recalling the inequalities \eqref{eq: order of u'bar wh} for $\bar{u}' \w{h}$ on $J_{h(1)}^-$, we set $r_k \coloneqq  \bar{u}' \w{h} (h(1) -k)$ for $0 \le k \le \lp{h}$, that is, we have
\begin{align}\label{eq: one-line notation on J-}
 r_{\lp{h}} < \cdots < r_1 < r_0
\end{align}
in one-line notation of $\bar{u}' \w{h}$ on $J_{h(1)}^-$.
We define $0\le\m{\lp{h}}\le \cdots\le \m{1}\le \m{0}\le \rp{h} + 1$
and 
$1 \le \q{1}, \q{2} , \ldots, \q{\m{0}} \le \rp{h} +1$ by 
\begin{equation}\label{eq: observation on J 10}
\{1 \le q \le \rp{h} +1 \mid \bar{u}'  \w{h} (h(1) +q) < r_k\} 
= \{\q{1}, \q{2}, \ldots, \q{\m{k}}\}
\end{equation}
for $0\le k\le \lp{h}$, and by
\begin{align*}
\bar{u}'  \w{h} (h(1) +\q{1}) < \bar{u}'  \w{h} (h(1) +\q{2}) < \cdots  < \bar{u}'  \w{h} (h(1) +\q{\m{0}}).
\end{align*}
Here, we mean $\{\q{1}, \q{2}, \ldots, \q{\m{k}}\}=\emptyset$ when $\m{k}=0$.
The definition \eqref{eq: observation on J 10} is well-defined because of \eqref{eq: one-line notation on J-}. Let 
\begin{align*}
\Delta_k\coloneqq\m{k}-\m{k+1} 
= |\{1 \le q \le \rp{h} +1 \mid r_{k+1} < \bar{u}'  \w{h} (h(1) +q) < r_k\} |
\end{align*}
for $0 \le k \le \lp{h}$, where we take $r_{\lp{h}+1}=\m{\lp{h}+1}=0$ as conventions so that $\Delta_{\lp{h}}=\m{\lp{h}}$.
In the running example of $n=20$, 
we have $m_2=0$, $m_1=1$, and $m_0=2$ (see also Figure~\ref{pic:example_mi}).\\

\noindent \underline{\textbf{Case 2-a:} $\Delta_k = 0$ for all $0 \le k \le \lp{h}$.}\\

In this case, the inequalities \eqref{eq:additional 30} are also satisfied for $\bar{u}' \w{h}$ on the whole $J_{h(1)}$ by the definition of $\Delta_k$, and this leads us to set $u \coloneqq \bar{u}' \in\mathfrak{S}_{J}$. Then, as in Case 1, condition {\rm (i)} follows by Lemma~\ref{l:length_easiest_case}, and condition {\rm (iv)} follows by the same arguments as that in the proof of Proposition~\ref{p:condition_v_easy_case}.

\begin{proposition}\label{prop: case 2-a condition (ii)}
The equality $(u \w{h})_J = u_J$ holds, that is, condition {\rm (ii)}  holds for $u$.
\end{proposition} 

\begin{proof}
Take $1 \le i \le n -1$ such that $s_i (\xi_h) = \xi_h$. It suffices to prove that for $j_1, j_2 \in J_i$, $u \w{h} (j_1) < u \w{h} (j_2)$ if and only if $u(j_1) < u (j_2)$. If $h(1) \notin J_i$, then the proof of Proposition~\ref{p:correction_on_parabolic_completed} in Case 1 implies the assertion (see also the paragraph before Proposition~\ref{p:correction_on_parabolic_completed}). 

Hence we may assume that $i = h(1)$. If $j_1, j_2 \le h(1)$ or $j_1, j_2 \ge h(1) +1$, then we have $u \w{h} (j_1) < u \w{h} (j_2)$ if and only if $u (j_1) < u (j_2)$ because of \eqref{eq: orders on J+} and \eqref{eq: orders on J-}. Hence it is enough to consider the case $j_1 \le h(1)$ and $h(1) +1\le j_2$. In this case, we have 
\[
u (j_1) < u (j_2)
\] 
by \eqref{eq:additional 30}. Since $\Delta_k = 0$ for all $k$, we know that \eqref{eq:additional 30} holds for $\bar{u}' \w{h}$ as well after replacing $\bar{u}'$ by $\bar{u}' \w{h}$, that is, we have 
\[ \bar{u}'  \w{h} (j_1) < \bar{u}'  \w{h} (j_2),\]
which implies the assertion.
\end{proof}

\begin{proposition}
Condition {\rm (iii)} holds for $u$. That is, if $s_i (\xi_h) = \xi_h$ and $h$ is strictly increasing on $J_i$, then $u(j) = j$ for all $j\in J_i$.
\end{proposition} 

\begin{proof}
If $h(1) \notin J_i$, then the proof of Proposition~\ref{p:condition_iv_for_u} implies the assertion. Hence we may assume that $i = h(1)$ in the assumption of condition (iii). If $\rp{h} = 0$, then we have $J_{h(1)}=J_{h(1)}^-\sqcup\{h(1)+1\}$. Since $u=\bar{u}'$, we see by \eqref{eq: preserving - and +} that
$u(h(1) +1) = h(1) +1$. This and Lemma~\ref{l:fixed_part_by_u_prime} imply the desired claim in this case.

If $\rp{h} \ge 1$, then we have $h(1)+\rp{h}+1 \ge h(1)+2$, which implies that $s_{h(1)+1} (\xi_{h}) = \xi_{h}$. This means by Lemma~\ref{l:calculation_of_xi_for_h_prime} that $s_{h(1)} (\xi_{h'}) = \xi_{h'}$, and 
\[J' _{h(1)} = \{h(1), h(1) +1, \ldots, h(1) +\rp{h}\},
\]
where we used $s_{h(1)-1}(\xi_{h'})\ne \xi_{h'}$.
Since we have
\[
h (h(1) +1) < h (h(1) +2) < \cdots < h (h(1) +\rp{h} +1)\]
by the assumption of condition (iii), it follows from 
the definition of $h'$ that 
\[
h' (h(1)) < h' (h(1) +1) < \cdots < h' (h(1) +\rp{h}).\]
Hence condition {\rm (iii)}  for $u'$ implies that $u'(j) = j$ for $h(1)\le j\le h(1)+\rp{h}$. Thus it follows from \eqref{eq: two permutations} that
\begin{align*} 
u(j) = \bar{u}' (j) = u'(j-1)+1=(j-1)+1=j
\end{align*}
for $h(1)+1\le j\le h(1)+\rp{h}+1$. From this and Lemma~\ref{l:fixed_part_by_u_prime}, we conclude the proposition.
\end{proof}

\vspace{10pt}

\noindent \underline{\textbf{Case 2-b:} $\Delta_k \ge 1$ for some $0 \le k \le \lp{h}$.}\\

In this case, the property \eqref{eq:additional 30} does not hold for $\bar{u}' \w{h}$ on $J_{h(1)}$ (after replacing $\bar{u}'$ by $\bar{u}' \w{h}$) by the definition of $\Delta_k$, which means that the numbers of $\bar{u}' $ on $J_{h(1)}$ and the numbers of $\bar{u}' \w{h}$ on $J_{h(1)}$ are in different orders.
Hence we cannot take $u$ to be $\bar{u}' $ to have the desired property $u_J=(u\w{h})_J$.
To resolve this, we define a certain permutation $v\in\mathfrak{S}_J$ which makes the order of the numbers of $u\coloneqq v\bar{u}' $ on $J_{h(1)}$ is the same as the order of the numbers of $\bar{u}' \w{h}$ on $J_{h(1)}$. It is not immediately clear that this implies $u_J=(u\w{h})_J$ since $\bar{u}' \w{h}$ is changed as $u \w{h}$ simultaneously, but we prove that the equality in fact follows. 

To find such $v\in\mathfrak{S}_J$, we focus on the numbers of $\bar{u}' $ on $J_{h(1)}^-$ which are given by 
\begin{align*}
\bar{u}'  (h(1) -k) = h(1) -k \qquad (0 \le k \le \lp{h})
\end{align*}
by Lemma~\ref{l:fixed_part_by_u_prime}.
We set
\[\Mh \coloneqq \max\{0 \le k \le \lp{h} \mid \Delta_{k}=\m{k}-\m{k+1} \ge 1\}\]
so that we have 
\begin{align}\label{eq:additional 1010}
0=\m{\lp{h}} = \cdots =\m{\Mh+1} < \m{\Mh}\le \cdots \le \m{1}\le \m{0} \ (\le  \rp{h}+1)
\end{align}
by definition. Let
\begin{align*}
v_k \coloneqq s_{h(1)-k +\m{k}-1} \cdots s_{h(1)-k +1} s_{h(1)-k} \qquad (0 \le k \le \Mh).
\end{align*}
Note that $v_k$ is a cyclic permutation of length $\m{k}$ which is visualized in Figure \ref{pic:def of vk}.
\begin{figure}[htbp]
%WinTpicVersion4.32a
{\unitlength 0.1in%
\begin{picture}(42.0000,6.2000)(8.5000,-15.7000)%
% BOX 2 5 3 0 Black White  
% 2 850 950 5050 1570
% 
\special{pn 8}%
\special{pa 850 950}%
\special{pa 5050 950}%
\special{pa 5050 1570}%
\special{pa 850 1570}%
\special{pa 850 950}%
\special{ip}%
% LINE 2 0 3 0 Black White  
% 2 1490 1250 1550 1250
% 
\special{pn 8}%
\special{pa 1490 1250}%
\special{pa 1550 1250}%
\special{fp}%
% LINE 2 0 3 1 Black White  
% 2 1520 1450 1520 1250
% 
\special{pn 8}%
\special{pa 1520 1450}%
\special{pa 1520 1250}%
\special{fp}%
% LINE 2 0 3 0 Black White  
% 2 1520 1450 4210 1450
% 
\special{pn 8}%
\special{pa 1520 1450}%
\special{pa 4210 1450}%
\special{fp}%
% STR 2 0 3 0 Black White  
% 4 1200 1100 1200 1200 2 0 0 0
% \footnotesize{$h(1)-k$}
\put(12.0000,-12.0000){\makebox(0,0)[lb]{\footnotesize{$h(1)-k$}}}%
% STR 2 0 3 0 Black White  
% 4 2050 1100 2050 1200 2 0 0 0
% \footnotesize{$h(1)-k+1$}
\put(20.5000,-12.0000){\makebox(0,0)[lb]{\footnotesize{$h(1)-k+1$}}}%
% STR 2 0 3 0 Black White  
% 4 3750 1100 3750 1200 2 0 0 0
% \footnotesize{$h(1)-k+\m{k}$}
\put(37.5000,-12.0000){\makebox(0,0)[lb]{\footnotesize{$h(1)-k+\m{k}$}}}%
% STR 2 0 3 0 Black White  
% 4 3180 1060 3180 1160 2 0 0 0
% $\cdots$
\put(31.8000,-11.6000){\makebox(0,0)[lb]{$\cdots$}}%
% STR 2 0 3 0 Black White  
% 4 1770 1070 1770 1170 2 0 0 0
% $\leftarrow$
\put(17.7000,-11.7000){\makebox(0,0)[lb]{$\leftarrow$}}%
% LINE 2 0 3 1 Black White  
% 2 1930 1100 1930 1160
% 
\special{pn 8}%
\special{pa 1930 1100}%
\special{pa 1930 1160}%
\special{fp}%
% STR 2 0 3 0 Black White  
% 4 2870 1070 2870 1170 2 0 0 0
% $\leftarrow$
\put(28.7000,-11.7000){\makebox(0,0)[lb]{$\leftarrow$}}%
% LINE 2 0 3 1 Black White  
% 2 3030 1100 3030 1160
% 
\special{pn 8}%
\special{pa 3030 1100}%
\special{pa 3030 1160}%
\special{fp}%
% STR 2 0 3 0 Black White  
% 4 3520 1070 3520 1170 2 0 0 0
% $\leftarrow$
\put(35.2000,-11.7000){\makebox(0,0)[lb]{$\leftarrow$}}%
% LINE 2 0 3 1 Black White  
% 2 3680 1100 3680 1160
% 
\special{pn 8}%
\special{pa 3680 1100}%
\special{pa 3680 1160}%
\special{fp}%
% STR 2 0 3 0 Black White  
% 4 4170 1290 4170 1390 2 0 0 0
% $\uparrow$
\put(41.7000,-13.9000){\makebox(0,0)[lb]{$\uparrow$}}%
% LINE 2 0 3 0 Black White  
% 2 4210 1450 4210 1270
% 
\special{pn 8}%
\special{pa 4210 1450}%
\special{pa 4210 1270}%
\special{fp}%
\end{picture}}%
\caption{The cyclic permutation $v_k$.}
\label{pic:def of vk}
\end{figure}
Hence we have $\ell(v_{\Mh})\le \cdots \le \ell(v_1)\le \ell(v_0)$. Now we set
\begin{align*}
u\coloneqq v_{\Mh} \cdots v_1 v_0 \bar{u}' .
\end{align*}
Note that each $v_k$ $(0 \le k \le \Mh)$ is a permutation on $J_{h(1)}\subseteq [n]$ since we have $h(1)-\lp{h}\le h(1)-k$ and $h(1)-k+\m{k}\le h(1)+\rp{h}+1$ in Figure \ref{pic:def of vk}. Thus it follows that
$u=v_{\Mh} \cdots v_1 v_0 \bar{u}'\in\mathfrak{S}_J$. 
In the running example of $n = 20$, we have $J_{h(1)}=\{7, 8, \ldots, 14\}$ and $\Mh=1 < 2 = \lp{h}$, and it follows that $u = v_1 v_0 \bar{u}' = (s_8) (s_{10}s_9) \bar{u}'$. Figures \ref{pic:example_mi} and \ref{pic:example_modification} for this example visualize the idea of the definition of $u$ which we stated at the beginning of Case 2-b: we defined $u=v_{\Mh} \cdots v_1 v_0 \bar{u}' $ so that the order of the numbers of $u$ on $J_{h(1)}$ is the same as the order of the numbers of $\bar{u}' \w{h}$ on $J_{h(1)}$. To see this in general, we prepare the following lemma.

\begin{lemma}\label{lem: the modifying permutation preserves the order}
For $1 \le j_1 < j_2 \le n$, we have $v_{\Mh} \cdots v_1 v_0 (j_1) > v_{\Mh} \cdots v_1 v_0 (j_2)$ if and only if $j_1 = h(1) - k$ and $j_2 = h(1) + l$ for some $0 \le k \le \Mh$ and $1 \le l \le m_k$. In particular, the permutation $v_{\Mh} \cdots v_1 v_0$ preserves the order of the numbers in $[n]\setminus \{h(1)-k \mid 0\le k\le \Mh\}$.
\end{lemma}

\begin{proof}
We first assume that $v_{\Mh} \cdots v_1 v_0 (j_1) > v_{\Mh} \cdots v_1 v_0 (j_2)$. Then there exists
$0\leq k\leq \Mh$ such that the following inequalities hold: 
\begin{equation}\label{equ:5.20-1}
\begin{split}
&v_{k-1}\cdots v_0 (j_1) < v_{k-1}\cdots v_0 (j_2), \\
&v_{k}v_{k-1}\cdots v_0 (j_1) > v_{k}v_{k-1}\cdots v_0 (j_2).
\end{split}
\end{equation}
Then, by the definition of $v_k$, it follows that
\begin{equation*}
\begin{split}
&v_{k-1}\cdots v_0 (j_1) = h(1)-k, \\
&v_{k}v_{k-1}\cdots v_0 (j_1) = h(1)-k+\m{k}.
\end{split}
\end{equation*}
The first equality implies that $j_1 = h(1)-k$ since the permutations $v_0, v_1, \ldots, v_{k-1}$ preserve $h(1)-k$. The definition of $v_k$ and inequalities (\ref{equ:5.20-1}) also imply that  
\[
h(1)-(k-1)\le v_{k-1}\cdots v_0 (j_2) \le h(1)-k+\m{k}.\]
By the definition of $v_{k-1}$ and $h(1)-k+\m{k}< h(1)-(k-1)+\m{k-1}$, this means that 
\[
h(1)-(k-2)\le v_{k-2}\cdots v_0 (j_2) \le h(1)-(k-1)+\m{k}.\]
By continuing this argument, we obtain that
\[ 
h(1)+1 \le j_2 \le h(1)+\m{k}.
\]
Thus we have $j_2 = h(1)+ l$ for some $1\le l \le \m{k}$.

Conversely, assume that $j_1 = h(1) - k$ and $j_2 = h(1) + l$ for some $0 \le k \le \Mh$ and $1 \le l \le m_k$. By reversing the argument above, we see that 
\begin{align*}
&v_{k}v_{k-1}\cdots v_0 (j_1) = h(1)-k+\m{k},\\
&h(1)-(k-1) \le v_{k-1}\cdots v_0 (j_2) \le h(1)-k+\m{k}.
\end{align*}
In particular, it follows that 
\[
v_{k}v_{k-1}\cdots v_0 (j_2) = v_{k-1}\cdots v_0 (j_2) -1 < h(1)-k+\m{k} = v_{k}v_{k-1}\cdots v_0 (j_1).
\]
Hence we see by \eqref{eq:additional 1010} that 
\[
v_{\Mh} \cdots v_1 v_0 (j_2) < h(1)-k+\m{k} = v_{\Mh} \cdots v_1 v_0 (j_1). 
\]
\end{proof}

Since the definition of $q_1, q_2,\ldots,q_{\m{0}}$ implies that $\bar{u}' \w{h} (h(1) +q_{l})$ is the $l$-th smallest number in $\bar{u}' \w{h}(J_{h(1)}^+)$ for $1 \le l \le \m{0}$, we see by \eqref{eq: preserving - and +} and \eqref{eq: orders on J+} that $\bar{u}'  (h(1) +q_{l})$ is the $l$-th smallest number in $\bar{u}' (J_{h(1)}^+) = J_{h(1)}^+$, which implies that
\begin{equation}\label{eq:additional 60}
\begin{split}
\bar{u}'  (h(1) +q_{l}) = h(1) +l \quad{\rm for}\ 1 \le l \le \m{0}.
\end{split}
\end{equation}
Combining this with Lemmas \ref{l:fixed_part_by_u_prime} and \ref{lem: the modifying permutation preserves the order}, it follows that for $j_1, j_2 \in [n]$, we have $\bar{u}' (j_1) < \bar{u}' (j_2)$ and $u (j_1) > u (j_2)$ if and only if $j_1 = h(1) - k$ and $j_2 = h(1) +q_{l}$ for some $0 \le k \le \Mh$ and $1 \le l \le m_k$. Hence, by \eqref{eq: orders on J+} and \eqref{eq: orders on J-}, the definition of $q_1, q_2,\ldots,q_{\m{0}}$ implies that 
\begin{align}\label{eq: additional 380}
u (j_1) < u (j_2) \ \text{if and only if} \ 
\bar{u}'  \w{h} (j_1) < \bar{u}'  \w{h} (j_2)
\quad \text{for $j_1, j_2\in J_{h(1)}$;}
\end{align}
see Figures \ref{pic:example_mi} and \ref{pic:example_modification} for the pictorial meaning of this argument. Indeed, we defined the permutation $u\in\mathfrak{S}_J$ so that this holds as we claimed above. We will prove that the latter inequality is in fact equivalent to $u \w{h} (j_1) < u \w{h} (j_2)$ to see condition (ii) for $u$. 

Our first aim is to prove condition (i) for $u$. For this purpose, we make a few observations in what follows.

\begin{lemma}\label{lem: length of modified u}
$\ell(u) = \ell(\bar{u}' ) + (\m{0}+\m{1}+\cdots +\m{\Mh})$.
\end{lemma} 

\begin{proof}
In one-line notation of $\bar{u}' $, the numbers $h(1)-\Mh, h(1)-\Mh+1,\ldots, h(1)$ appear before the numbers $h(1)+1, h(1)+2, \ldots, h(1)+\m{0}$ by Lemma~\ref{l:fixed_part_by_u_prime}. Hence the assertion follows by Lemma \ref{lem: the modifying permutation preserves the order}.
\end{proof}

Notice that
\[
\DD{h(1)-k}=0
\qquad (0\le k\le \lp{h}).
\] 
Since $s_{h(1) -k} (\xi_h) = \xi_h$ for all $1 \le k \le \lp{h}$, this means from \eqref{eq: vanishing coefficient} that
\begin{equation}\label{eq:shape_of_h_smaller_than_h(1)}
\begin{aligned}
h(h(1) -k) = \jh -2k \qquad \text{($0 \le k \le \lp{h}$)}.
\end{aligned}
\end{equation}
This leads us to define $t_1,t_2,\ldots,t_{\lp{h}}>0$ by
\begin{align}\label{eq:additional 70}
(h(1)-k)^{(+)} = \tk{k}   \qquad (1 \le k \le \lp{h}), 
\end{align}
which is equivalent to
\begin{align}\label{eq:additional 75}
\w{h}(\tk{k}) 
= h(h(1)-k)+1 = \jh -(2k -1) \qquad (1 \le k \le \lp{h})
\end{align}
by \eqref{eq:shape_of_h_smaller_than_h(1)}. See Figure \ref{pic:def tk}. Since we have 
$\jh -1>\jh -3>\cdots> \jh -(2\lp{h}-1)$, 
the maximality of the values of $\w{h}$ implies that
\[
t_1 < t_2 < \cdots < t_{\lp{h}}.
\]
Also, \eqref{eq:additional 75} implies that $h$ is stable at $\tk{k}$:
\[
h(\tj{k}) = h(\tk{k}) \qquad (1 \le k \le \lp{h}).
\]
For example, we have $\tk{1}=11$ and $\tk{2}=13$ in the running example of $n=20$ in Section \ref{subsec: Example}.

\begin{figure}[htbp]
%WinTpicVersion4.32a
{\unitlength 0.1in%
\begin{picture}(31.4000,26.9000)(12.5000,-38.9000)%
% LINE 2 0 3 0 Black White  
% 6 1800 1400 2000 1400 2000 1400 2000 1800 2000 1800 2200 1800
% 
\special{pn 8}%
\special{pa 1800 1400}%
\special{pa 2000 1400}%
\special{fp}%
\special{pa 2000 1400}%
\special{pa 2000 1800}%
\special{fp}%
\special{pa 2000 1800}%
\special{pa 2200 1800}%
\special{fp}%
% LINE 2 0 3 0 Black White  
% 4 3200 3200 3200 3400 3200 3400 3600 3400
% 
\special{pn 8}%
\special{pa 3200 3200}%
\special{pa 3200 3400}%
\special{fp}%
\special{pa 3200 3400}%
\special{pa 3600 3400}%
\special{fp}%
% CIRCLE 2 0 0 0 Black Black  
% 4 2100 1700 2120 1720 2120 1720 2120 1720
% 
\special{sh 1.000}%
\special{ia 2100 1700 28 28 0.0000000 6.2831853}%
\special{pn 8}%
\special{ar 2100 1700 28 28 0.0000000 6.2831853}%
% CIRCLE 2 0 0 0 Black Black  
% 4 3300 3300 3320 3320 3320 3320 3320 3320
% 
\special{sh 1.000}%
\special{ia 3300 3300 28 28 0.0000000 6.2831853}%
\special{pn 8}%
\special{ar 3300 3300 28 28 0.0000000 6.2831853}%
% CIRCLE 2 0 0 0 Black Black  
% 4 3500 2300 3520 2320 3520 2320 3520 2320
% 
\special{sh 1.000}%
\special{ia 3500 2300 28 28 0.0000000 6.2831853}%
\special{pn 8}%
\special{ar 3500 2300 28 28 0.0000000 6.2831853}%
% STR 2 0 3 0 Black White  
% 4 2850 2800 2850 2900 2 0 0 0
% $\ddots$
\put(28.5000,-29.0000){\makebox(0,0)[lb]{$\ddots$}}%
% LINE 2 0 3 0 Black White  
% 6 2200 1800 2200 2200 2200 2200 2400 2200 2400 2200 2400 2600
% 
\special{pn 8}%
\special{pa 2200 1800}%
\special{pa 2200 2200}%
\special{fp}%
\special{pa 2200 2200}%
\special{pa 2400 2200}%
\special{fp}%
\special{pa 2400 2200}%
\special{pa 2400 2600}%
\special{fp}%
% LINE 2 0 3 0 Black White  
% 2 2400 2600 2600 2600
% 
\special{pn 8}%
\special{pa 2400 2600}%
\special{pa 2600 2600}%
\special{fp}%
% CIRCLE 2 0 0 0 Black Black  
% 4 2500 2500 2520 2520 2520 2520 2520 2520
% 
\special{sh 1.000}%
\special{ia 2500 2500 28 28 0.0000000 6.2831853}%
\special{pn 8}%
\special{ar 2500 2500 28 28 0.0000000 6.2831853}%
% CIRCLE 2 0 0 0 Black Black  
% 4 2300 2100 2320 2120 2320 2120 2320 2120
% 
\special{sh 1.000}%
\special{ia 2300 2100 28 28 0.0000000 6.2831853}%
\special{pn 8}%
\special{ar 2300 2100 28 28 0.0000000 6.2831853}%
% LINE 2 2 3 0 Black White  
% 2 4300 2300 1800 2300
% 
\special{pn 8}%
\special{pa 4300 2300}%
\special{pa 1800 2300}%
\special{dt 0.045}%
% LINE 2 2 3 0 Black White  
% 2 3500 1300 3500 3600
% 
\special{pn 8}%
\special{pa 3500 1300}%
\special{pa 3500 3600}%
\special{dt 0.045}%
% STR 2 0 3 0 Black Black  
% 4 3200 3730 3200 3830 2 0 0 0
% $\tk{k}$
\put(32.0000,-38.3000){\makebox(0,0)[lb]{$\tk{k}$}}%
% LINE 2 2 3 0 Black White  
% 2 2300 1300 2300 3600
% 
\special{pn 8}%
\special{pa 2300 1300}%
\special{pa 2300 3600}%
\special{dt 0.045}%
% STR 2 0 3 0 Black Black  
% 4 1970 3730 1970 3830 2 0 0 0
% $h(1)-k$
\put(19.7000,-38.3000){\makebox(0,0)[lb]{$h(1)-k$}}%
% BOX 2 5 3 0 Black Black  
% 2 1250 1200 4390 3890
% 
\special{pn 8}%
\special{pa 1250 1200}%
\special{pa 4390 1200}%
\special{pa 4390 3890}%
\special{pa 1250 3890}%
\special{pa 1250 1200}%
\special{ip}%
\end{picture}}%
\caption{The definition of $t_k$.}
\label{pic:def tk}
\end{figure}
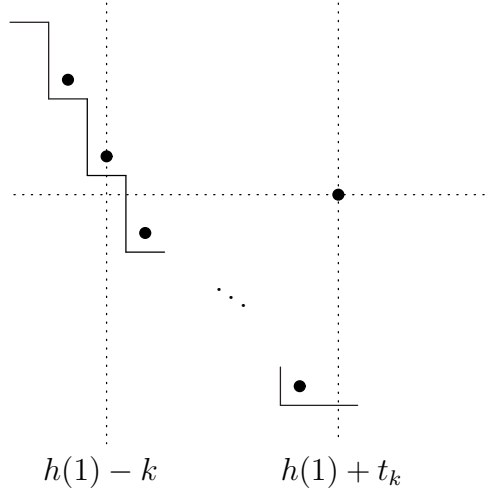

We use condition (iv) for $u'$ to prove the following lemma.

\begin{lemma}\label{p:not_containing_1}
The set $J_{h(1)}$ does not contain $1$. In particular, we have $
h(1)-\lp{h}>1$.
\end{lemma}

\begin{proof}
If $1 \in J_{h(1)}$, then we have $J_{h(1)}=J_1$ and $s_1(\xi_h)=\xi_h$.
Lemma~\ref{l:connected_component_containing_1} then implies that
\[h(1) < h(2) < \cdots < h(h(1)) < \cdots < h(h(1) +\rp{h} +1).\]
Since $h$ is stable at $\LL{i} $ by definition, this means that
\[h(1) +\rp{h} +1 < \LL{i} \]
for all $1 \le i \le h(1)$.
Hence condition (iv) for $u'$ implies (as in the proof of Proposition~\ref{p:condition_v_easy_case}) that
\[
\bar{u}'  \w{h} (i) < \bar{u}'  \w{h} (j)
\]
for all $1 \le i \le h(1)$ and $i +1 \le j \le h(1) +\rp{h} +1$. This means by definition that $\Delta_{k}=\m{k}-\m{k+1} = 0$ for all $k$, which is a contradiction.
\end{proof}

We note by \eqref{eq:additional 1010} that
the set in \eqref{eq: observation on J 10} is non-empty for $0\le k\le \Mh$:
\begin{align}\label{eq: additional 530}
\{\q{1}, \q{2}, \ldots, \q{\m{k}}\}\ne\emptyset \qquad (0\le k\le \Mh).
\end{align}

\begin{lemma}\label{l:computation_of_w(h)_for_h(1)-k}
For $1 \le k \le \Mh$, the following equality holds:
\[\w{h}^{-1} (h(1) -k) = \DD{\tk{k}}.\]
\end{lemma}

\begin{proof}
We first show that $[h(1),\tk{k}]\subseteq J_{h(1)}$.
By the previous lemma, 
we have $h(1)-(k+1)\ge1$.
Hence it follows from condition {\rm (iv)}  for $u'$ with $i=h(1)-(k+1)\ge1$ that
\begin{align*}
u' \w{h'} (h(1) -(k +1)) < u' \w{h'} (j -1)
\end{align*}
for $h(1) -k +1 \le j < i^{(+)'}+1$ $(\le \LL{i}^{\prime})$, 
where we mean $i^{(+)'}=i^{(+)_{h'}}$ and $\LL{i}^{\prime}=\LL{i}(h')$.
Since we have $i^{(+)'}+1 = (i+1)^{(+)} = \tk{k}$ by the definition of $i=h(1)-(k+1)$, this inequality and Corollary \ref{c:computation_of_v_first_step} imply that
\begin{align*}
\bar{u}'  \w{h} (h(1) -k) < \bar{u}'  \w{h} (j)
\end{align*}
for $h(1) -k +1 \le j < \tk{k}$. This means that $h(1) + q_{m}$ cannot belong to the interval $[h(1) -k +1, \ \tj{k}]$ for any $1\le m\le \m{k}$ because of \eqref {eq: observation on J 10}, where we know that $\{\q{1}, \q{2}, \ldots, \q{\m{k}}\}\ne\emptyset$ as pointed out above.
Since we have $h(1)+q_{m}\in J_{h(1)}^+$ by definition, this means that
\begin{align}\label{eq: additional 520}
\tk{k} 
\le h(1) +q_{m} \le h(1) +\rp{h} +1
\end{align}
for $1\le m\le \m{k}$.
From this, we obtain 
\begin{align}\label{eq: additional 370}
[h(1),\tk{k}]\subseteq J_{h(1)}.
\end{align}
This in fact implies the claim of this lemma as we prove in what follows. 

We start with the case $k=1$.
Since the definition of $t_1$ means that $\w{h}(\tk{1})= \w{h}(h(1))-1$ which is smaller than $\w{h}(h(1))$ by $1$, we must have
\begin{equation*}
\begin{split}
&\stable[h(1),j]\le \twosteps[h(1),j] \quad \text{for $h(1)< j< \tk{1}$}, \\
&\stable[h(1),\tk{1}] = \twosteps[h(1),\tk{1}]+1,
\end{split}
\end{equation*}
where $\stable[k,\ell]$ and $\twosteps[k,\ell]$ are the numbers defined in the proof of Lemma~\ref{lem: D and +}. 
This is because, if this does not hold, then 
$\w{h} (h(1) +t_1) \neq \w{h} (h(1)) -1$ by the maximality of the values of $\w{h}$. The principle of similar shapes in the sense of Remark \ref{rem: the modified principle} now implies that the same holds on $[1,\DD{\tk{1}}]$.
Namely, 
\begin{equation*}
\begin{split}
&\stable[1,j]\le \twosteps[1,j] \quad \text{for $1< j< \DD{\tk{1}}$}, \\
&\stable[1,\DD{\tk{1}}] = \twosteps[1,\DD{\tk{1}}]+1.
\end{split}
\end{equation*}
This means that 
\begin{align*}
\w{h} ( \DD{\tk{1}} ) 
= \w{h} (1) -1
= h(1) -1.
\end{align*}
That is, $\DD{\tk{1}}
= \w{h}^{-1}(h(1)-1)$,
as desired. 

For $k\ge 2$, we can argue in a similar way on $[\tk{k-1}, \tk{k}]$ with the principle of similar shapes.
\end{proof}

\begin{proposition}\label{prop: condition (i) for u}
Condition {\rm (i)} holds for $u$. That is, the equality $\ell(u \w{h}) = \ell(u) + \ell(\w{h})$ holds.
\end{proposition} 

\begin{proof}
By Lemmas~\ref{l:length_easiest_case} and \ref{lem: length of modified u}, it suffices to show
\[\ell(u \w{h}) = \ell(\bar{u}'  \w{h}) +(\m{0} +\m{1} + \cdots +\m{\Mh}).\]
Since $u=v_{\Mh} \cdots v_1 v_0 \bar{u}'$, Lemma \ref{lem: the modifying permutation preserves the order} implies that it is enough to prove 
\[
(\bar{u}'  \w{h})^{-1} (h(1) -k) < (\bar{u}'  \w{h})^{-1} (h(1) +m)
\]
for $0\le k\le \Mh$ and $1\le m\le \m{k}$.

When $k=0$, we have $(\bar{u}'  \w{h})^{-1} (h(1)) = 1$ by Lemma~\ref{l:fixed_part_by_u_prime}. In particular, we see that
\begin{align*}
(\bar{u}'  \w{h})^{-1} (h(1) ) < (\bar{u}'  \w{h})^{-1} (h(1) +m) 
\end{align*}
for $1 \le m \le \m{0}$, as desired.
Hence we may assume that $1\le k\le \Mh$ in the following. By Lemma~\ref{l:computation_of_w(h)_for_h(1)-k}, we have
\[
\w{h}^{-1} (h(1) -k) = \DD{\tk{k}}.
\]
This means from 
Lemma~\ref{lem: addition D and wh} that
\begin{align}\label{eq:additional 110}
\w{h}^{-1} (h(1) -k) < 
\w{h}^{-1} (j) 
\quad \text{for all $j \ge \tk{k}$}.
\end{align}
Recalling that we are assuming $1\le k\le \Mh$, we know from \eqref{eq: additional 520} that
\begin{align}\label{eq:additional 130}
\tk{k}\le h(1) +q_m \qquad (1\le m\le \m{k}).
\end{align}
Thus, for $1\le m\le \m{k}$, we have
\begin{align*}
(\bar{u}'  \w{h})^{-1} (h(1) -k) 
&= \w{h}^{-1} (h(1) -k) \qquad \text{(by Lemma~\ref{l:fixed_part_by_u_prime})}\\
&< \w{h}^{-1} (h(1) +q_m) \qquad \text{(by \eqref{eq:additional 110} and \eqref{eq:additional 130})}\\
&= (\bar{u}'  \w{h})^{-1} (h(1) +m) \qquad \text{(by \eqref{eq:additional 60} and $m \le \m{k}\le\m{0})$},
\end{align*}
as desired.
\end{proof}

Our next aim is to prove condition {\rm (ii)} for $u$. To give a proof, we need to know that  $\bar{u}'  \w{h}(J_{h(1)})$ does not contain $h(1)-k$ for any $0 \le k \le \Mh$, which will be proved in Proposition~\ref{l:not_contain_changing_points}. We prepare two lemmas for this.

\begin{lemma}\label{l:position_of_changing_values}
The following hold:
\begin{align*}
(\bar{u}'  \w{h})^{-1} (h(1)) < (\bar{u}'  \w{h})^{-1} (h(1) -1) < \cdots < (\bar{u}'  \w{h})^{-1} (h(1) -\Mh) \le h(1) -\lp{h}.
\end{align*}
\end{lemma}

\begin{proof}
Note that the left-most number is equal to $1$.
This means that if $\Mh=0$, then there is nothing to prove. 
Hence we may assume $\Mh\ge1$ in the following.
For $1 \le k \le \Mh$, we have
\begin{equation}\label{eq: first equations}
\begin{split}
(\bar{u}'  \w{h})^{-1} (h(1) -k) 
&= \w{h}^{-1} (h(1) -k) \qquad \text{(by Lemma~\ref{l:fixed_part_by_u_prime})} \\
&= \DD{\tk{k}} \qquad \text{(by Lemma~\ref{l:computation_of_w(h)_for_h(1)-k})}, 
\end{split}
\end{equation}
which implies that 
\[(\bar{u}'  \w{h})^{-1} (h(1) -1) < (\bar{u}'  \w{h})^{-1} (h(1) -2) < \cdots < (\bar{u}'  \w{h})^{-1} (h(1) -\Mh).\]

Next, we write
\[
r\coloneqq (\bar{u}'  \w{h})^{-1} (h(1) -\Mh), 
\]
and prove the right-most inequality $r \le h(1) -\lp{h}$ in the claim.
The same computation as above shows that
\begin{align*}
r = \w{h}^{-1} (h(1) -\Mh) = \DD{\tk{\Mh}} < \tk{\Mh}, 
\end{align*}
where the left-most equality shows that $\w{h}(r)<h(1)$, which implies that $h$ must be stable at $r$ by the definition of $\w{h}(r)$.
Since $h$ is not stable at $k$ for all $h(1)-\lp{h}<k\le h(1)$ by \eqref{eq:shape_of_h_smaller_than_h(1)}, it suffices to show that
\begin{align}\label{eq: additional 360}
r \notin \{h(1)+1, h(1)+2, \ldots, \tk{\Mh}\}.
\end{align}
Recall from \eqref{eq:additional 75} that 
\[
\w{h}(\tk{\Mh}) = \jh - (2\Mh -1) = h(h(1)-\Mh) +1,
\]
which is depicted in Figure \ref{pic:value at tMh}.
\begin{figure}[htbp]
%WinTpicVersion4.32a
{\unitlength 0.1in%
\begin{picture}(50.0000,26.8000)(7.2000,-34.9000)%
% STR 2 0 3 0 Black White  
% 4 2500 3340 2500 3440 2 0 0 0
% $h(1)-\Mh$
\put(25.0000,-34.4000){\makebox(0,0)[lb]{$h(1)-\Mh$}}%
% STR 2 0 3 0 Black White  
% 4 3710 3340 3710 3440 2 0 0 0
% $h(1)+t_{\Mh}$
\put(37.1000,-34.4000){\makebox(0,0)[lb]{$h(1)+t_{\Mh}$}}%
% BOX 2 5 3 0 Black White  
% 2 720 810 5720 3490
% 
\special{pn 8}%
\special{pa 720 810}%
\special{pa 5720 810}%
\special{pa 5720 3490}%
\special{pa 720 3490}%
\special{pa 720 810}%
\special{ip}%
% STR 2 0 3 0 Black White  
% 4 1040 1870 1040 1970 2 0 0 0
% $h(h(1)-\Mh)+1$
\put(10.4000,-19.7000){\makebox(0,0)[lb]{$h(h(1)-\Mh)+1$}}%
% LINE 2 0 3 0 Black White  
% 6 2400 1000 2600 1000 2600 1000 2600 1400 2600 1400 2800 1400
% 
\special{pn 8}%
\special{pa 2400 1000}%
\special{pa 2600 1000}%
\special{fp}%
\special{pa 2600 1000}%
\special{pa 2600 1400}%
\special{fp}%
\special{pa 2600 1400}%
\special{pa 2800 1400}%
\special{fp}%
% LINE 2 0 3 0 Black White  
% 4 3800 2800 3800 3000 3800 3000 4200 3000
% 
\special{pn 8}%
\special{pa 3800 2800}%
\special{pa 3800 3000}%
\special{fp}%
\special{pa 3800 3000}%
\special{pa 4200 3000}%
\special{fp}%
% CIRCLE 2 0 0 0 Black Black  
% 4 2700 1300 2720 1320 2720 1320 2720 1320
% 
\special{sh 1.000}%
\special{ia 2700 1300 28 28 0.0000000 6.2831853}%
\special{pn 8}%
\special{ar 2700 1300 28 28 0.0000000 6.2831853}%
% CIRCLE 2 0 0 0 Black Black  
% 4 3900 2900 3920 2920 3920 2920 3920 2920
% 
\special{sh 1.000}%
\special{ia 3900 2900 28 28 0.0000000 6.2831853}%
\special{pn 8}%
\special{ar 3900 2900 28 28 0.0000000 6.2831853}%
% CIRCLE 2 0 0 0 Black Black  
% 4 4100 1900 4120 1920 4120 1920 4120 1920
% 
\special{sh 1.000}%
\special{ia 4100 1900 28 28 0.0000000 6.2831853}%
\special{pn 8}%
\special{ar 4100 1900 28 28 0.0000000 6.2831853}%
% STR 2 0 3 0 Black White  
% 4 3450 2400 3450 2500 2 0 0 0
% $\ddots$
\put(34.5000,-25.0000){\makebox(0,0)[lb]{$\ddots$}}%
% LINE 2 0 3 0 Black White  
% 6 2800 1400 2800 1800 2800 1800 3000 1800 3000 1800 3000 2200
% 
\special{pn 8}%
\special{pa 2800 1400}%
\special{pa 2800 1800}%
\special{fp}%
\special{pa 2800 1800}%
\special{pa 3000 1800}%
\special{fp}%
\special{pa 3000 1800}%
\special{pa 3000 2200}%
\special{fp}%
% LINE 2 0 3 0 Black White  
% 2 3000 2200 3200 2200
% 
\special{pn 8}%
\special{pa 3000 2200}%
\special{pa 3200 2200}%
\special{fp}%
% CIRCLE 2 0 0 0 Black Black  
% 4 3100 2100 3120 2120 3120 2120 3120 2120
% 
\special{sh 1.000}%
\special{ia 3100 2100 28 28 0.0000000 6.2831853}%
\special{pn 8}%
\special{ar 3100 2100 28 28 0.0000000 6.2831853}%
% CIRCLE 2 0 0 0 Black Black  
% 4 2900 1700 2920 1720 2920 1720 2920 1720
% 
\special{sh 1.000}%
\special{ia 2900 1700 28 28 0.0000000 6.2831853}%
\special{pn 8}%
\special{ar 2900 1700 28 28 0.0000000 6.2831853}%
% LINE 2 2 3 0 Black White  
% 2 4900 1900 2400 1900
% 
\special{pn 8}%
\special{pa 4900 1900}%
\special{pa 2400 1900}%
\special{dt 0.045}%
% LINE 2 2 3 0 Black White  
% 2 4100 900 4100 3200
% 
\special{pn 8}%
\special{pa 4100 900}%
\special{pa 4100 3200}%
\special{dt 0.045}%
% LINE 2 2 3 0 Black White  
% 2 2900 900 2900 3200
% 
\special{pn 8}%
\special{pa 2900 900}%
\special{pa 2900 3200}%
\special{dt 0.045}%
\end{picture}}%
\caption{The value $\w{h}(\tk{\Mh})$.}
\label{pic:value at tMh}
\end{figure}
Now, $h$ is stable at both $r$ and $\tk{\Mh}$ where $\w{h}$ takes the values $h(1) -\Mh$ and $h(h(1) -\Mh) +1$, respectively.
Since $h$ is a Hessenberg function, the former value $h(1) -\Mh$ is less than the latter value $h(h(1) -\Mh) +1$. Thus, by the maximality of the values of $\w{h}$, it follows that
\[
r \notin \{h(1) -\Mh +1, h(1) -\Mh +2, \ldots, \tk{\Mh}\}
\]
(see Figure \ref{pic:value at tMh}), which implies \eqref{eq: additional 360}, as desired.
\end{proof}

As the following lemma indicates, condition (iv) ensures that $L_{h(1)-k}\in J_{h(1)}$ for $0 \le k \le \Mh$.

\begin{lemma}\label{lem: L and q}
For $0 \le k \le \Mh$, we have
\[
\LL{h(1)-k}\leq h(1)+\rp{h}+1.
\]
\end{lemma}

\begin{proof}
We know from \eqref{eq: additional 530} that $\{q_1,q_2,\ldots,q_{\m{k}} \}\ne\emptyset$ in \eqref{eq: observation on J 10}.
Hence it suffices to show for $q \in \{q_1,q_2,\ldots,q_{\m{k}} \}$ that
\begin{align}\label{eq: additional 300}
\LL{h(1)-k}\leq h(1)+q.
\end{align}
We first note that $h(1)-k\ge1$. Since $q \in \{q_1,q_2,\ldots,q_{\m{k}} \}$, we have
\begin{align*}
\bar{u}'  \w{h} (h(1) +q) < \bar{u}'  \w{h} (h(1) -k).
\end{align*}
Condition {\rm (iv)} for $u'$ with $i=h(1)-k$ now implies that 
\[
\LL{h(1)-k} \leq h(1)+q
\]
(cf.\ the proof of Proposition~\ref{p:condition_v_easy_case}). 
\end{proof}

In the proof of the next proposition, we will implicitly use the fact that if $h(i+1)=h(i)+2$, then
\[
\w{h}(i^{(+)}) > 1. 
\]
This is because $\w{h}(i^{(+)})=h(i)+1>1$ by the definition of $(+)$-operation. 

\begin{proposition}\label{l:not_contain_changing_points}
The set $J_{h(1)}$ does not contain $(\bar{u}'  \w{h})^{-1} (h(1) -k)$ for any $0 \le k \le \Mh$.
\end{proposition}

\begin{proof}
The claim for $k=0$ follows from Lemma~\ref{p:not_containing_1}.
Suppose $(\bar{u}'  \w{h})^{-1} (h(1) -k) \in J_{h(1)}$ for some $1 \le k \le \Mh$, 
and we will deduce a contradiction. 
Lemma~\ref{l:position_of_changing_values} implies that we must have $k = \Mh$ and  
\begin{align}\label{eq:additional 200}
(\bar{u}'  \w{h})^{-1} (h(1) -\Mh) = h(1) -\lp{h}.
\end{align}
By \eqref{eq: first equations}, this means that
\begin{align}\label{eq:additional 210}
\DD{\tk{\Mh}} = h(1) -\lp{h}, 
\end{align}
which implies that 
\begin{align}\label{eq: additional 540}
\LL{\DD{\tk{\Mh}}} = \LL{h(1) -\lp{h}}.
\end{align} 
The equality $\tk{\Mh}=(h(1)-\Mh)^{(+)}$ from \eqref{eq:additional 70} implies that $\LL{\tk{\Mh}}=\LL{h(1)-\Mh}$ by the definition \eqref{eq: def of Li}. This means that $[\tk{\Mh}-1, \LL{\tk{\Mh}}]\subseteq J_{h(1)}$ by Lemma~\ref{lem: L and q}. In addition, we have $h(\tk{\Mh})\ne h(\tk{\Mh}-1)+2$ since $h$ is stable at $\tk{\Mh}$ by definition. Thus we deduce from Lemma~\ref{lem: D and L} that
\begin{align*}
\DD{\LL{\tk{\Mh}}} = \LL{\DD{\tk{\Mh}}} = \LL{h(1) -\lp{h}},
\end{align*}
where the second equality follows from \eqref{eq: additional 540}.
Since we have $(h(1)-\lp{h})^{(+)} < \LL{h(1) -\lp{h}}$ by the definition \eqref{eq: def of Li}, this implies that
\begin{align*}
(h(1)-\lp{h})^{(+)} < \LL{h(1) -\lp{h}} = \DD{\LL{\tk{\Mh}}} < \LL{\tk{\Mh}}. 
\end{align*}
By \eqref{eq:additional 70} and the definition \eqref{eq: def of Li} again, we deduce by this that $\tk{\lp{h}}< \LL{h(1)-\Mh}$.
Thus, by Lemma~\ref{lem: L and q}, we obtain
\begin{align*}
\tk{\lp{h}} < h(1)+\rp{h}+1.
\end{align*}
This means that \eqref{eq: additional 370} holds for $1\le k\le\lp{h}$. Hence by the argument after \eqref{eq: additional 370}, the claim of Lemma~\ref{l:computation_of_w(h)_for_h(1)-k} holds for $1\le k\le\lp{h}$. This implies that \eqref{eq: first equations} holds for $1\le k\le \lp{h}$, and thus we obtain
\[
(\bar{u}'  \w{h})^{-1} (h(1) -\Mh) \le (\bar{u}'  \w{h})^{-1} (h(1) -\lp{h})
\]
(cf.\ the inequalities below \eqref{eq: first equations}).
Since \eqref{eq: first equations} for $k=\lp{h}$ is now valid, the latter half of the proof of Lemma~\ref{l:position_of_changing_values} proves 
\[ (\bar{u}'  \w{h})^{-1} (h(1) -\lp{h})\le  h(1) -\lp{h}\] 
by replacing $\Mh$ by $\lp{h}$ in the argument. 
From the last two inequalities and \eqref{eq:additional 200}, it now follows that $(\bar{u}'  \w{h})^{-1} (h(1) -\Mh)=(\bar{u}'  \w{h})^{-1} (h(1) -\lp{h})$, and hence $\Mh=\lp{h}$.

We now deduce a contradiction by using $\Mh=\lp{h}$.
Let $i=h(1)-\Mh$.
Then \eqref{eq:additional 210} now says that 
\begin{align}\label{eq:additional 310}
\DD{i^{(+)}} = i
\end{align}
(see Figure \ref{pic: D(i+)=i}).
\begin{figure}[htbp]
\hspace{-10pt}
%WinTpicVersion4.32a
{\unitlength 0.1in%
\begin{picture}(68.0000,22.7000)(12.0000,-34.0000)%
% CIRCLE 2 0 0 0 Black Black  
% 4 3700 1900 3720 1920 3720 1920 3720 1920
% 
\special{sh 1.000}%
\special{ia 3700 1900 28 28 0.0000000 6.2831853}%
\special{pn 8}%
\special{ar 3700 1900 28 28 0.0000000 6.2831853}%
% LINE 2 0 3 0 Black Black  
% 2 4200 1200 5600 2600
% 
\special{pn 8}%
\special{pa 4200 1200}%
\special{pa 5600 2600}%
\special{fp}%
% LINE 2 2 3 0 Black Black  
% 2 3500 1200 3500 3200
% 
\special{pn 8}%
\special{pa 3500 1200}%
\special{pa 3500 3200}%
\special{dt 0.045}%
% LINE 2 2 3 0 Black Black  
% 2 2800 1700 4700 1700
% 
\special{pn 8}%
\special{pa 2800 1700}%
\special{pa 4700 1700}%
\special{dt 0.045}%
% LINE 2 2 3 0 Black Black  
% 2 4700 1700 4700 3200
% 
\special{pn 8}%
\special{pa 4700 1700}%
\special{pa 4700 3200}%
\special{dt 0.045}%
% STR 2 0 3 0 Black White  
% 4 2990 3300 2990 3400 2 0 0 0
% $i=h(1)-\Mh$
\put(29.9000,-34.0000){\makebox(0,0)[lb]{$i=h(1)-\Mh$}}%
% STR 2 0 3 0 Black White  
% 4 4400 3300 4400 3400 2 0 0 0
% $i^{(+)}=(h(1)-\Mh)^{(+)}$
\put(44.0000,-34.0000){\makebox(0,0)[lb]{$i^{(+)}=(h(1)-\Mh)^{(+)}$}}%
% BOX 2 5 3 0 Black White  
% 2 1200 1130 8000 3400
% 
\special{pn 8}%
\special{pa 1200 1130}%
\special{pa 8000 1130}%
\special{pa 8000 3400}%
\special{pa 1200 3400}%
\special{pa 1200 1130}%
\special{ip}%
% CIRCLE 2 0 0 0 Black Black  
% 4 4700 1700 4720 1720 4720 1720 4720 1720
% 
\special{sh 1.000}%
\special{ia 4700 1700 28 28 0.0000000 6.2831853}%
\special{pn 8}%
\special{ar 4700 1700 28 28 0.0000000 6.2831853}%
% STR 2 0 3 0 Black White  
% 4 5300 2650 5300 2750 2 0 0 0
% the diagonal line
\put(53.0000,-27.5000){\makebox(0,0)[lb]{the diagonal line}}%
% LINE 2 0 3 0 Black White  
% 4 4800 2800 4400 2800 4400 2800 4400 2600
% 
\special{pn 8}%
\special{pa 4800 2800}%
\special{pa 4400 2800}%
\special{fp}%
\special{pa 4400 2800}%
\special{pa 4400 2600}%
\special{fp}%
% CIRCLE 2 0 0 0 Black Black  
% 4 4500 2700 4520 2720 4520 2720 4520 2720
% 
\special{sh 1.000}%
\special{ia 4500 2700 28 28 0.0000000 6.2831853}%
\special{pn 8}%
\special{ar 4500 2700 28 28 0.0000000 6.2831853}%
% STR 2 0 3 0 Black White  
% 4 4000 2350 4000 2450 2 0 0 0
% $\Huge{\text{*}}$
\put(40.0000,-24.5000){\makebox(0,0)[lb]{$\Huge{\text{*}}$}}%
% LINE 2 0 3 0 Black Black  
% 2 4800 2800 4800 3000
% 
\special{pn 8}%
\special{pa 4800 2800}%
\special{pa 4800 3000}%
\special{fp}%
% LINE 2 0 3 0 Black Black  
% 6 3400 1600 3600 1600 3600 1600 3600 2000 3600 2000 3800 2000
% 
\special{pn 8}%
\special{pa 3400 1600}%
\special{pa 3600 1600}%
\special{fp}%
\special{pa 3600 1600}%
\special{pa 3600 2000}%
\special{fp}%
\special{pa 3600 2000}%
\special{pa 3800 2000}%
\special{fp}%
\end{picture}}%
\caption{The picture of $\DD{i^{(+)}} = i$, where $i=h(1)-\Mh$.}
\label{pic: D(i+)=i}
\end{figure}
By Lemma~\ref{lem: L and q}, we know that 
\begin{align*}
[i, i^{(+)}+1]\subseteq J_{h(1)}
\end{align*}
since $i^{(+)}+1\le  \LL{i}=\LL{h(1)-\Mh}$.
This means that $s_{i^{(+)}} \in \mathfrak{S}_J$, and hence we have
\[h(i^{(+)}+1) = h(i^{(+)}) +2\]
by the case (1) of Lemma~\ref{lem: principle of similar shapes 3} (see Figure \ref{pic: D(i+)=i}).
This means from the definition of the $(+)$-operation that \[i^{(+)}<(i^{(+)})^{(+)}.\]
We know that 
\[
\DD{(i^{(+)})^{(+)}}  = \DD{i^{(+)}}^{(+)} = i^{(+)}
\]
by \eqref{eq:additional 310} and Lemma~\ref{lem: D and +}.
Namely, \eqref{eq:additional 310} holds after replacing $i$ by $i^{(+)}$.
By Lemma~\ref{lem: L and q} again, we know that 
\begin{align*}
[i^{(+)}, (i^{(+)})^{(+)}+1]\subseteq J_{h(1)}
\end{align*}
since $(i^{(+)})^{(+)}+1\le  \LL{i}=\LL{h(1)-\Mh}$.
This means that $s_{(i^{(+)})^{(+)}}\in\mathfrak{S}_J$, and we obtain that
\[h((i^{(+)})^{(+)}+1) = h((i^{(+)})^{(+)}) +2\]
as above, so that we have $(i^{(+)})^{(+)}<((i^{(+)})^{(+)})^{(+)}$ by the definition of the $(+)$-operation.
Continuing this argument, we conclude that $\LL{h(1) -\Mh}$ $(=\LL{i^{(+\infty)}})$ exceeds $h(1) +\rp{h} +1$, that is, $h(1) +\rp{h} +1 < \LL{h(1) -\Mh}$, which contradicts Lemma~\ref{lem: L and q}.
\end{proof}

The previous proposition gives us information on the positions where $\bar{u}' \w{h}$ takes values less than $h(1)$. 
In contrast, the next lemma gives us lower bounds for the positions where $\bar{u}' \w{h}$ takes values greater than $h(1)$. 
We will use both properties to prove condition (ii) for $u$. 

\begin{lemma}\label{l:evaluation_of_L_for_changing_points}
For $0 \le k \le \Mh$ and 
$1\le l\le \m{k}$, 
the following holds:
\begin{align*}
\LL{(\bar{u}'  \w{h})^{-1} (h(1) -k)} \le (\bar{u}' \w{h})^{-1} (h(1) +l).
\end{align*}
\end{lemma}

\begin{proof}
By \eqref{eq:additional 60}, our claim is the same as 
\begin{align*}
\LL{(\bar{u}'  \w{h})^{-1} (h(1) -k)} \le \w{h}^{-1} (h(1) +q)
\end{align*}
for $q \in \{q_1,q_2,\ldots,q_{\m{k}} \}$. We first consider the case $k \ge 1$. In this case, we know from \eqref{eq: first equations} that
\begin{align*}
(\bar{u}'  \w{h})^{-1} (h(1) -k) = \DD{\tk{k}}.
\end{align*}
Hence, letting $r\coloneqq\tk{k}$, what we need to prove is 
\begin{align}\label{eq:additional 150}
\LL{\DD{r}} \le \w{h}^{-1} (h(1) +q).
\end{align}
We know from Lemma~\ref{lem: L and q} and \eqref{eq:additional 70} that $[r-1, \LL{r}]=[r-1,\LL{h(1)-k}] \subseteq J_{h(1)}$, and that $h$ is stable at $r$
so that we have $\LL{\DD{r}} = \DD{\LL{r}}$ by Lemma~\ref{lem: D and L}. Since $\LL{r} =\LL{h(1)-k} \le h(1) +q$ by \eqref{eq:additional 70} and \eqref{eq: additional 300}, it follows that
\begin{align}\label{eq:additional 170}
\LL{\DD{r}} = \DD{\LL{r}} \le \DD{h(1)+q} \le \w{h}^{-1}(h(1)+q),
\end{align}
where the right-most inequality follows from Lemma~\ref{lem: addition D and wh}.
This is exactly \eqref{eq:additional 150}, as desired.

We next consider the case $k = 0$. In this case, we have $(\bar{u}'  \w{h})^{-1} (h(1))=1$ by Lemma~\ref{l:fixed_part_by_u_prime}, and hence what we need to prove is 
\begin{align}\label{eq:additional 185}
\LL{1} \le \w{h}^{-1} (h(1) +q)
\end{align}
for $q \in \{q_1,q_2,\ldots,q_{\m{0}} \}$ as above. 
Recall that we have $[h(1),\LL{h(1)}]\subseteq J_{h(1)}$ from Lemma~\ref{lem: L and q}. 
Thus if we have
\begin{align}\label{eq:additional 180}
\LL{1} = \DD{\LL{h(1)}},
\end{align}
then the argument used in  the case $k \ge 1$ (i.e.\ the argument proving \eqref{eq:additional 170}) works to conclude \eqref{eq:additional 185} in this case as well.
Hence let us prove \eqref{eq:additional 180} in what follows.

If $h(1)<h(2)$, then we have $\DD{h(1)}=0$ and $\DD{h(1)+1}=1$. Since 
$s_{h(1)} (\xi_h) = \xi_h$, this and \eqref{eq: vanishing coefficient} means that $h(h(1)+1)=h(h(1))+1$. Hence we obtain $\LL{h(1)}=\LL{h(1)+1}$ by the definition \eqref{eq: def of Li}. This implies that $\DD{\LL{h(1)}}=\DD{\LL{h(1)+1}}$, and hence \eqref{eq:additional 180} follows from Lemma~\ref{lem: D and L} in this case.

If $h(1)=h(2)$, then we have $\LL{1}=2$.  Also, 
it follows that $\DD{h(1)+1}=\DD{h(1)}+2$ in this case, and hence that $h$ is stable at $h(1)+1$ by (5.9) since $s_{h(1)} (\xi_h) = \xi_h$.
Thus we obtain $\LL{h(1)}=h(1)+1$ by the definition \eqref{eq: def of Li}, and the right-hand side of \eqref{eq:additional 180} is equal to $\DD{h(1)+1}=2$ which agrees with the left-hand side $\LL{1}=2$ so that \eqref{eq:additional 180} follows in this case as well.
\end{proof}

We now prove condition (ii) for $u$ by using 
Proposition~\ref{l:not_contain_changing_points} and Lemma~\ref{l:evaluation_of_L_for_changing_points} as declared.

\begin{proposition}\label{p:correction_on_parabolic_completed_Case_2}
Condition {\rm (ii)}  holds for $u$, that is, the equality $(u \w{h})_J=u_J$ holds.
\end{proposition} 

\begin{proof}
Take $1 \le i \le n -1$ such that $s_i (\xi_h) = \xi_h$. It suffices to prove that for $j_1, j_2 \in J_i$, $u \w{h} (j_1) < u \w{h} (j_2)$ if and only if $u (j_1) < u (j_2)$.

We first consider the case $h(1) \in J_i$. 
Recall from \eqref{eq: additional 380} that we have $u (j_1) < u (j_2)$ if and only if $\bar{u}'  \w{h} (j_1) < \bar{u}'  \w{h} (j_2)$. 
Now, Proposition~\ref{l:not_contain_changing_points} shows that $\bar{u}'  \w{h}(J_i)$ does not contain the numbers $h(1)-k$ for $0\le k\le \Mh$. This means that the order of the numbers in $\bar{u}'  \w{h}(J_i)$ are the same as that of $u \w{h}(J_i)$ by Lemma~\ref{lem: the modifying permutation preserves the order}. Hence it follows that $u \w{h} (j_1) < u \w{h} (j_2)$ if and only if $u (j_1) < u (j_2)$, as desired.

We next consider the case $h(1) \notin J_i$. In this case, the proof of Proposition~\ref{p:correction_on_parabolic_completed} implies that $\bar{u}'  \w{h} (j_1) < \bar{u}'  \w{h} (j_2)$ if and only if $\bar{u}'  (j_1) < \bar{u}'  (j_2)$. We need to prove that the permutation $v_{\Mh} \cdots v_1 v_0$ preserves this equivalence. 
Since $J_i\cap J_{h(1)}=\emptyset$ in this case, Lemma~\ref{l:fixed_part_by_u_prime} implies that
$\bar{u}'  (J_i)$ does not contain $h(1) -k$ for any $0 \le k \le \Mh$. This means that $\bar{u}'  (j_1) < \bar{u}'  (j_2)$ if and only if $u (j_1) < u (j_2)$ by Lemma~\ref{lem: the modifying permutation preserves the order}. Hence if 
$\bar{u}'  \w{h}(J_i)$ also does not contain $h(1) -k$ for any $0 \le k \le \Mh$, then the assertion follows immediately. Thus we may assume that 
\begin{align}\label{eq:additional 260}
h(1) -k \in \bar{u}'  \w{h}(J_i) \quad \text{for some $0 \le k \le \Mh$}
\end{align}
in the following.
This means that $(\bar{u}'  \w{h})(J_i)$ contains some numbers which are increased by $v_{\Mh}\cdots v_1 v_0$. Under this assumption, let us prove that 
\begin{equation}\label{eq:additional 270}
\begin{split}
&(v_{k-1}\cdots v_1v_0) \bar{u}'  \w{h} (j) 
\notin \{(h(1)-k)+1, (h(1)-k)+2, \ldots, h(1) -k +\m{k}\} \\
&\hspace{160pt} \text{for all $j\in J_i$ and $0 \le k \le \Mh$ satisfying \eqref{eq:additional 260}}.
\end{split}
\end{equation}
Equivalently, if $k$ satisfies \eqref{eq:additional 260}, then the set $(v_{k-1}\cdots v_1v_0)\bar{u}'  \w{h}(J_i)$ does not contain any number $\ell$ satisfying $h(1)-k< \ell\le v_k(h(1)-k)$ so that $v_k$ preserves the order of the numbers in $(v_{k-1}\cdots v_1v_0)\bar{u}'  \w{h}(J_i)$.
Hence if \eqref{eq:additional 270} is proved, then it follows that $v_{\Mh}\cdots v_1 v_0$ preserves the order of the numbers in $\bar{u}'  \w{h}(J_i)$, and we obtain the assertion of this proposition.

We prove \eqref{eq:additional 270} by using Lemma~\ref {l:evaluation_of_L_for_changing_points} in what follows.
By \eqref{eq:additional 260} and Lemma~\ref{l:position_of_changing_values}, we know that $J_i$ lies left to $J_{h(1)}$ in the standard listing of $[n]$, where it holds that $J_i\cap J_{h(1)}=\emptyset$. In particular, we have $J_i
\subseteq [1,h(1)-1]$. This implies that
\begin{equation}\label{eq:strictly_increasing_not_containing_h(1)}
\begin{aligned}
h(j+1)=h(j)+2 
\qquad 
(i -\lb{i} \le j \le i + \rb{i})
\end{aligned}
\end{equation}
by \eqref{eq: vanishing coefficient} since $\DD{a}=0$ for $1\le a \le h(1)$.
\begin{figure}[htbp]
%WinTpicVersion4.32a
{\unitlength 0.1in%
\begin{picture}(44.0000,23.0000)(11.2000,-33.1000)%
% CIRCLE 2 0 0 0 Black Black  
% 4 2700 1500 2720 1520 2720 1520 2720 1520
% 
\special{sh 1.000}%
\special{ia 2700 1500 28 28 0.0000000 6.2831853}%
\special{pn 8}%
\special{ar 2700 1500 28 28 0.0000000 6.2831853}%
% CIRCLE 2 0 0 0 Black Black  
% 4 2900 1900 2920 1920 2920 1920 2920 1920
% 
\special{sh 1.000}%
\special{ia 2900 1900 28 28 0.0000000 6.2831853}%
\special{pn 8}%
\special{ar 2900 1900 28 28 0.0000000 6.2831853}%
% CIRCLE 2 0 0 0 Black Black  
% 4 4100 2700 4120 2720 4120 2720 4120 2720
% 
\special{sh 1.000}%
\special{ia 4100 2700 28 28 0.0000000 6.2831853}%
\special{pn 8}%
\special{ar 4100 2700 28 28 0.0000000 6.2831853}%
% LINE 2 2 3 0 Black White  
% 2 2500 1100 2500 3030
% 
\special{pn 8}%
\special{pa 2500 1100}%
\special{pa 2500 3030}%
\special{dt 0.045}%
% STR 2 0 3 0 Black White  
% 4 2240 3140 2240 3240 2 0 0 0
% $i-k_{i,-}$
\put(22.4000,-32.4000){\makebox(0,0)[lb]{$i-k_{i,-}$}}%
% STR 2 0 3 0 Black White  
% 4 3840 3140 3840 3240 2 0 0 0
% $i+k_i+1$
\put(38.4000,-32.4000){\makebox(0,0)[lb]{$i+k_i+1$}}%
% BOX 2 5 3 0 Black White  
% 2 1120 1010 5520 3310
% 
\special{pn 8}%
\special{pa 1120 1010}%
\special{pa 5520 1010}%
\special{pa 5520 3310}%
\special{pa 1120 3310}%
\special{pa 1120 1010}%
\special{ip}%
% STR 2 0 3 0 Black White  
% 4 3400 2220 3400 2320 2 0 0 0
% $\ddots$
\put(34.0000,-23.2000){\makebox(0,0)[lb]{$\ddots$}}%
% LINE 2 0 3 0 Black White  
% 4 4000 2400 4000 2800 4000 2800 4200 2800
% 
\special{pn 8}%
\special{pa 4000 2400}%
\special{pa 4000 2800}%
\special{fp}%
\special{pa 4000 2800}%
\special{pa 4200 2800}%
\special{fp}%
% LINE 2 2 3 0 Black White  
% 2 4100 1100 4100 3030
% 
\special{pn 8}%
\special{pa 4100 1100}%
\special{pa 4100 3030}%
\special{dt 0.045}%
% LINE 2 0 3 0 Black White  
% 10 2400 1200 2600 1200 2600 1200 2600 1600 2600 1600 2800 1600 2800 1600 2800 2000 2800 2000 3000 2000
% 
\special{pn 8}%
\special{pa 2400 1200}%
\special{pa 2600 1200}%
\special{fp}%
\special{pa 2600 1200}%
\special{pa 2600 1600}%
\special{fp}%
\special{pa 2600 1600}%
\special{pa 2800 1600}%
\special{fp}%
\special{pa 2800 1600}%
\special{pa 2800 2000}%
\special{fp}%
\special{pa 2800 2000}%
\special{pa 3000 2000}%
\special{fp}%
\end{picture}}%
\caption{The shape of $h$ on $J_i=[i -\lb{i}, i+\rb{i}+1]$.}
\label{pic:h for Ji}
\end{figure}
See Figure \ref{pic:h for Ji}.
We claim that 
\begin{align}\label{eq:additional 280}
(\bar{u}'  \w{h})^{-1} (h(1) -k) = i -\lb{i}.
\end{align}
To see this, we take cases. 
If $k\ge 1$, then we see from Lemma~\ref{l:fixed_part_by_u_prime} that
\begin{align*}
(\bar{u}'  \w{h})^{-1} (h(1) -k)=\w{h}^{-1} (h(1) -k),
\end{align*}
and we know that $h$ is stable at $\w{h}^{-1} (h(1) -k)$ since $h(1) -k<h(1)$. Thus \eqref{eq:strictly_increasing_not_containing_h(1)} (see Figure \ref{pic:h for Ji}) now implies that 
$\w{h}^{-1} (h(1) -k)=i -\lb{i}$ since $i -\lb{i}$ is the unique position in $J_i$ where $h$ can be stable. Hence we obtain \eqref{eq:additional 280} in this case.
If $k=0$, then \eqref{eq:additional 260} means that $1\in J_i$ since $h(1)=\bar{u}'  \w{h}(1)$ by Lemma~\ref{l:fixed_part_by_u_prime}, and hence we have $(\bar{u}'  \w{h})^{-1} (h(1) -k) = 1 = i -\lb{i}$. Namely, \eqref{eq:additional 280} holds in this case as well.

From \eqref{eq:strictly_increasing_not_containing_h(1)}, it also follows that $i +\rb{i} +1 < \LL{i -\lb{i}}$.
This is because  $\LL{i -\lb{i}}$ is greater than $i -\lb{i}$ and $h$ must be stable at $\LL{i -\lb{i}}$.
Now, \eqref{eq:additional 280} means that 
\[
i +\rb{i} +1 <\LL{i-\lb{i}}=\LL{(\bar{u}'  \w{h})^{-1}(h(1)-k)}.
\]
Hence Lemma~\ref{l:evaluation_of_L_for_changing_points} now implies that \begin{align*}
(\bar{u}' \w{h})^{-1} (h(1) +l) \notin J_i
\quad \text{for $1\le l\le \m{k}$}.
\end{align*}
We can rewrite this as
\begin{align}\label{eq: additional 560}
 \bar{u}'  \w{h} (j) \notin \{h(1)+1, h(1) +2, \ldots, h(1) +\m{k}\}
\quad \text{for $j \in J_i$}.
\end{align}
If $k = 0$, then this is precisely \eqref{eq:additional 270}. Hence we may assume that $k \ge 1$. Then \eqref{eq: additional 560} implies that
\begin{align}\label{eq: additional 570}
v_0 (\bar{u}'  \w{h} (j)) \notin \{h(1), h(1)+1, \ldots, h(1)-1 +\m{k}\}
\end{align}
since $h(1) +\m{k}\le h(1) +\m{0}$.
If $k = 1$, then this proves \eqref{eq:additional 270}. Hence we may assume that $k \ge 2$. Then \eqref{eq: additional 570} implies that
\begin{align*}
v_1 v_0 (\bar{u}'  \w{h} (j)) \notin \{h(1)-1, h(1), \ldots, h(1)-2 +\m{k}\}
\end{align*}
since we have $h(1)-1 +\m{k}\le h(1)-1 +\m{1}$. 
It is clear that we can continue this argument to see \eqref{eq:additional 270}.
\end{proof} 

\begin{proposition} 
Condition {\rm (iii)} holds for $u$.
That is, if $s_i (\xi_h) = \xi_h$ and $h$ is strictly increasing on $J_i$, then 
$u(j) = j$ for all $j\in J_i$.
\end{proposition}
\begin{proof}
To begin with, we show that $J_{h(1)}$ does not satisfy the assumption of condition (iii). For that purpose, assume that the assumption of condition (iii) holds on $J_{h(1)}$, that is, 
\[h(h(1)-\lp{h})<h(h(1)-\lp{h}+1)<\cdots <h(h(1)+\rp{h})<h(h(1)+\rp{h}+1).\]
This means that 
\[
h(1)+\rp{h}+1< \LL{h(1)-k} \quad (0\le k\le \lp{h})
\]
since $h(1)-k<\LL{h(1)-k}$ and $h$ must be stable at $L_{h(1)-k}$, but this contradicts Lemma~\ref{lem: L and q}. 

Assume that the assumption of condition (iii) holds on $J_{i}\ (\ne J_{h(1)})$.
The same argument as that in the proof of Proposition~\ref{p:condition_iv_for_u} shows that 
\begin{align}\label{eq: additional 390}
\bar{u}'(j)=j \quad\text{for} \quad j\in J_i.
\end{align}
Now recall that $u=v_{\Mh}\cdots v_1v_0\bar{u}'$ by definition, and that $v_{\Mh}\cdots v_1v_0\in\mathfrak{S}_J$ is in fact a permutation on $J_{h(1)}$ by definition. Since $J_{i}\cap J_{h(1)}=\emptyset$, it follows that $v_{\Mh}\cdots v_1v_0$ is trivial on $J_{i}$. This and  \eqref{eq: additional 390} mean that $u(j)=j$ for $j\in J_i$.
\end{proof}

\begin{proposition}
Condition {\rm (iv)} holds for $u$. That is, 
for $1 \le i \le n-1$ and $i +1 \le j < \LL{i} $, we have
$u \w{h} (i) < u \w{h} (j)$.
\end{proposition}

\begin{proof}
The same argument as that in the proof of Proposition~\ref{p:condition_v_easy_case} shows that 
\begin{align*}
\bar{u}' \w{h}(i)<\bar{u}' \w{h}(j)\quad\text{for}\quad 1 \le i \le n-1 \text{ and } i+1\leq j< \LL{i}.
\end{align*}

We now prove that $u \w{h} (i) < u \w{h} (j)$. Assume for a contradiction that this does not hold. Since $u \w{h} = v_M \cdots v_1 v_0 \bar{u}' \w{h}$, we see by Lemma \ref{lem: the modifying permutation preserves the order} that $\bar{u}' \w{h}(i) = h(1)-k$ and $\bar{u}' \w{h}(j)=h(1)+ l$ for some $0\leq k\leq \Mh$ and $1\leq l\leq \m{k}$, which implies that 
\begin{align*}
&i = (\bar{u}' \w{h})^{-1}(h(1)-k),\\
&j = (\bar{u}' \w{h})^{-1}(h(1)+ l).
\end{align*}
Hence it follows from Lemma~\ref{l:evaluation_of_L_for_changing_points} that $\LL{i} \leq j$ which contradicts the assumption $j<\LL{i}$ of condition (iv), as desired.
\end{proof}

\subsection{A pair of illustrating examples}\label{subsec: Example}
In this subsection, we give an example of a pair of nef Hessenberg functions which illustrate the argument in Case 2-b in Section \ref{subsect: Proof of Thm B}. 
Let $n=20$, and $h\colon[20]\rightarrow[20]$ the Hessenberg function depicted in Figure \ref{pic:example}, that is, 
\[
h = (9,10,10,11,12,12,13,15,17,18,18,19,19, 20, 20, 20, 20, 20, 20, 20).
\]
Let $h'\colon[19]\rightarrow[19]$ be the Hessenberg function given by $h' (i) \coloneqq h(i +1) -1$ for $1 \le i \le 19$ as in Section \ref{subsect: Proof of Thm B} (see Figure \ref{pic:example}):
\[
h' = (9, 9, 10,11,11,12,14,16,17,17,18,18,19,19,19,19,19,19,19).
\]

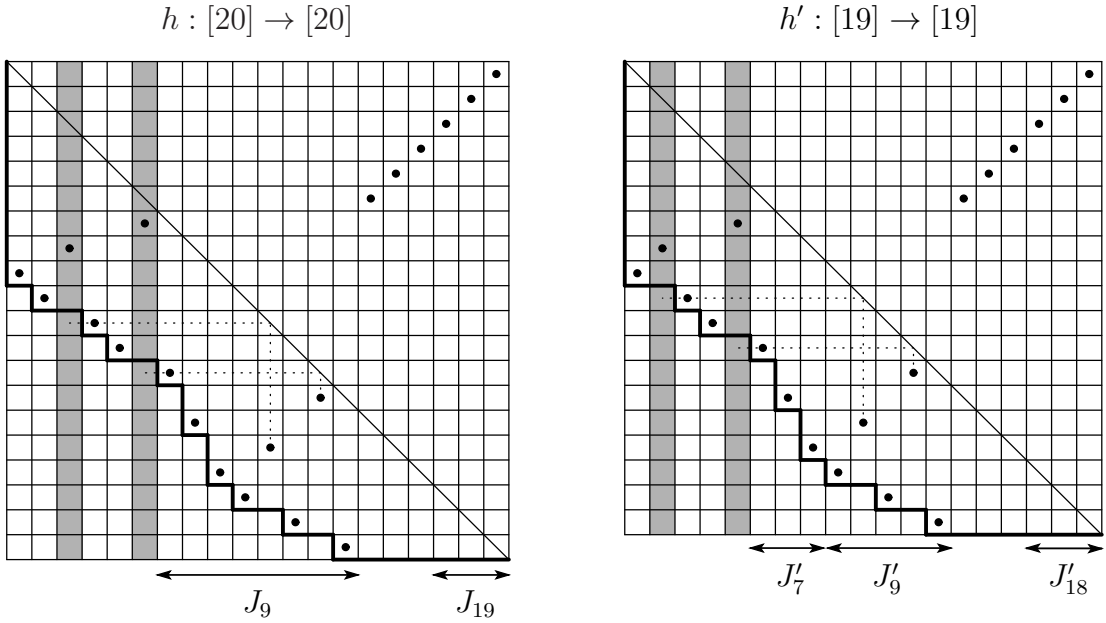
\begin{figure}[htbp]
%WinTpicVersion4.32a
{\unitlength 0.1in%
\begin{picture}(62.0000,31.0000)(8.0000,-40.1000)%
% STR 2 0 3 0 Black White  
% 4 1800 970 1800 1070 2 0 0 0
% $h:[20]\rightarrow[20]$
\put(18.0000,-10.7000){\makebox(0,0)[lb]{$h:[20]\rightarrow[20]$}}%
% BOX 2 5 1 0 Black Black  
% 2 1650 1200 1780 3800
% 
\special{pn 0}%
\special{sh 0.300}%
\special{pa 1650 1200}%
\special{pa 1780 1200}%
\special{pa 1780 3800}%
\special{pa 1650 3800}%
\special{pa 1650 1200}%
\special{ip}%
\special{pn 8}%
\special{pa 1650 1200}%
\special{pa 1780 1200}%
\special{pa 1780 3800}%
\special{pa 1650 3800}%
\special{pa 1650 1200}%
\special{ip}%
% BOX 2 5 1 0 Black Black  
% 2 1260 1200 1390 3800
% 
\special{pn 0}%
\special{sh 0.300}%
\special{pa 1260 1200}%
\special{pa 1390 1200}%
\special{pa 1390 3800}%
\special{pa 1260 3800}%
\special{pa 1260 1200}%
\special{ip}%
\special{pn 8}%
\special{pa 1260 1200}%
\special{pa 1390 1200}%
\special{pa 1390 3800}%
\special{pa 1260 3800}%
\special{pa 1260 1200}%
\special{ip}%
% LINE 2 0 3 0 Black White  
% 8 1000 1200 3600 1200 3600 1200 3600 3800 3600 3800 1000 3800 1000 3800 1000 1200
% 
\special{pn 8}%
\special{pa 1000 1200}%
\special{pa 3600 1200}%
\special{fp}%
\special{pa 3600 1200}%
\special{pa 3600 3800}%
\special{fp}%
\special{pa 3600 3800}%
\special{pa 1000 3800}%
\special{fp}%
\special{pa 1000 3800}%
\special{pa 1000 1200}%
\special{fp}%
% LINE 2 0 3 0 Black White  
% 2 1130 1200 1130 3800
% 
\special{pn 8}%
\special{pa 1130 1200}%
\special{pa 1130 3800}%
\special{fp}%
% LINE 2 0 3 0 Black White  
% 2 1260 1200 1260 3800
% 
\special{pn 8}%
\special{pa 1260 1200}%
\special{pa 1260 3800}%
\special{fp}%
% LINE 2 0 3 0 Black White  
% 2 1390 1200 1390 3800
% 
\special{pn 8}%
\special{pa 1390 1200}%
\special{pa 1390 3800}%
\special{fp}%
% LINE 2 0 3 0 Black White  
% 2 1520 1200 1520 3800
% 
\special{pn 8}%
\special{pa 1520 1200}%
\special{pa 1520 3800}%
\special{fp}%
% LINE 2 0 3 0 Black White  
% 2 1650 1200 1650 3800
% 
\special{pn 8}%
\special{pa 1650 1200}%
\special{pa 1650 3800}%
\special{fp}%
% LINE 2 0 3 0 Black White  
% 2 1780 1200 1780 3800
% 
\special{pn 8}%
\special{pa 1780 1200}%
\special{pa 1780 3800}%
\special{fp}%
% LINE 2 0 3 0 Black White  
% 2 1910 1200 1910 3800
% 
\special{pn 8}%
\special{pa 1910 1200}%
\special{pa 1910 3800}%
\special{fp}%
% LINE 2 0 3 0 Black White  
% 2 2040 1200 2040 3800
% 
\special{pn 8}%
\special{pa 2040 1200}%
\special{pa 2040 3800}%
\special{fp}%
% LINE 2 0 3 0 Black White  
% 2 2170 1200 2170 3800
% 
\special{pn 8}%
\special{pa 2170 1200}%
\special{pa 2170 3800}%
\special{fp}%
% LINE 2 0 3 0 Black White  
% 2 2300 1200 2300 3800
% 
\special{pn 8}%
\special{pa 2300 1200}%
\special{pa 2300 3800}%
\special{fp}%
% LINE 2 0 3 0 Black White  
% 2 2430 1200 2430 3800
% 
\special{pn 8}%
\special{pa 2430 1200}%
\special{pa 2430 3800}%
\special{fp}%
% LINE 2 0 3 0 Black White  
% 2 2560 1200 2560 3800
% 
\special{pn 8}%
\special{pa 2560 1200}%
\special{pa 2560 3800}%
\special{fp}%
% LINE 2 0 3 0 Black White  
% 2 2690 1200 2690 3800
% 
\special{pn 8}%
\special{pa 2690 1200}%
\special{pa 2690 3800}%
\special{fp}%
% LINE 2 0 3 0 Black White  
% 2 2820 1200 2820 3800
% 
\special{pn 8}%
\special{pa 2820 1200}%
\special{pa 2820 3800}%
\special{fp}%
% LINE 2 0 3 0 Black White  
% 2 2950 1200 2950 3800
% 
\special{pn 8}%
\special{pa 2950 1200}%
\special{pa 2950 3800}%
\special{fp}%
% LINE 2 0 3 0 Black White  
% 2 3080 1200 3080 3800
% 
\special{pn 8}%
\special{pa 3080 1200}%
\special{pa 3080 3800}%
\special{fp}%
% LINE 2 0 3 0 Black White  
% 2 3210 1200 3210 3800
% 
\special{pn 8}%
\special{pa 3210 1200}%
\special{pa 3210 3800}%
\special{fp}%
% LINE 2 0 3 0 Black White  
% 2 3340 1200 3340 3800
% 
\special{pn 8}%
\special{pa 3340 1200}%
\special{pa 3340 3800}%
\special{fp}%
% LINE 2 0 3 0 Black White  
% 2 3470 1200 3470 3800
% 
\special{pn 8}%
\special{pa 3470 1200}%
\special{pa 3470 3800}%
\special{fp}%
% LINE 2 0 3 0 Black White  
% 2 1000 1330 3600 1330
% 
\special{pn 8}%
\special{pa 1000 1330}%
\special{pa 3600 1330}%
\special{fp}%
% LINE 2 0 3 0 Black White  
% 2 1000 1460 3600 1460
% 
\special{pn 8}%
\special{pa 1000 1460}%
\special{pa 3600 1460}%
\special{fp}%
% LINE 2 0 3 0 Black White  
% 2 1000 1590 3600 1590
% 
\special{pn 8}%
\special{pa 1000 1590}%
\special{pa 3600 1590}%
\special{fp}%
% LINE 2 0 3 0 Black White  
% 2 1000 1720 3600 1720
% 
\special{pn 8}%
\special{pa 1000 1720}%
\special{pa 3600 1720}%
\special{fp}%
% LINE 2 0 3 0 Black White  
% 2 1000 1850 3600 1850
% 
\special{pn 8}%
\special{pa 1000 1850}%
\special{pa 3600 1850}%
\special{fp}%
% LINE 2 0 3 0 Black White  
% 2 1000 1980 3600 1980
% 
\special{pn 8}%
\special{pa 1000 1980}%
\special{pa 3600 1980}%
\special{fp}%
% LINE 2 0 3 0 Black White  
% 2 1000 2110 3600 2110
% 
\special{pn 8}%
\special{pa 1000 2110}%
\special{pa 3600 2110}%
\special{fp}%
% LINE 2 0 3 0 Black White  
% 2 1000 2240 3600 2240
% 
\special{pn 8}%
\special{pa 1000 2240}%
\special{pa 3600 2240}%
\special{fp}%
% LINE 2 0 3 0 Black White  
% 2 1000 2370 3600 2370
% 
\special{pn 8}%
\special{pa 1000 2370}%
\special{pa 3600 2370}%
\special{fp}%
% LINE 2 0 3 0 Black White  
% 2 1000 2500 3600 2500
% 
\special{pn 8}%
\special{pa 1000 2500}%
\special{pa 3600 2500}%
\special{fp}%
% LINE 2 0 3 0 Black White  
% 2 1000 2630 3600 2630
% 
\special{pn 8}%
\special{pa 1000 2630}%
\special{pa 3600 2630}%
\special{fp}%
% LINE 2 0 3 0 Black White  
% 2 1000 2760 3600 2760
% 
\special{pn 8}%
\special{pa 1000 2760}%
\special{pa 3600 2760}%
\special{fp}%
% LINE 2 0 3 0 Black White  
% 2 1000 2890 3600 2890
% 
\special{pn 8}%
\special{pa 1000 2890}%
\special{pa 3600 2890}%
\special{fp}%
% LINE 2 0 3 0 Black White  
% 2 1000 3020 3600 3020
% 
\special{pn 8}%
\special{pa 1000 3020}%
\special{pa 3600 3020}%
\special{fp}%
% LINE 2 0 3 0 Black White  
% 2 1000 3150 3600 3150
% 
\special{pn 8}%
\special{pa 1000 3150}%
\special{pa 3600 3150}%
\special{fp}%
% LINE 2 0 3 0 Black White  
% 2 1000 3280 3600 3280
% 
\special{pn 8}%
\special{pa 1000 3280}%
\special{pa 3600 3280}%
\special{fp}%
% LINE 2 0 3 0 Black White  
% 2 1000 3410 3600 3410
% 
\special{pn 8}%
\special{pa 1000 3410}%
\special{pa 3600 3410}%
\special{fp}%
% LINE 2 0 3 0 Black White  
% 2 1000 3540 3600 3540
% 
\special{pn 8}%
\special{pa 1000 3540}%
\special{pa 3600 3540}%
\special{fp}%
% LINE 2 0 3 0 Black White  
% 2 1000 3670 3600 3670
% 
\special{pn 8}%
\special{pa 1000 3670}%
\special{pa 3600 3670}%
\special{fp}%
% LINE 0 0 3 0 Black White  
% 40 1000 1200 1000 2370 1000 2370 1130 2370 1130 2370 1130 2500 1130 2500 1390 2500 1390 2500 1390 2630 1390 2630 1520 2630 1520 2630 1520 2760 1520 2760 1780 2760 1780 2760 1780 2890 1780 2890 1910 2890 1910 2890 1910 3150 1910 3150 2040 3150 2040 3150 2040 3410 2040 3410 2170 3410 2170 3410 2170 3540 2170 3540 2430 3540 2430 3540 2430 3670 2430 3670 2690 3670 2690 3670 2690 3800 2690 3800 3600 3800
% 
\special{pn 20}%
\special{pa 1000 1200}%
\special{pa 1000 2370}%
\special{fp}%
\special{pa 1000 2370}%
\special{pa 1130 2370}%
\special{fp}%
\special{pa 1130 2370}%
\special{pa 1130 2500}%
\special{fp}%
\special{pa 1130 2500}%
\special{pa 1390 2500}%
\special{fp}%
\special{pa 1390 2500}%
\special{pa 1390 2630}%
\special{fp}%
\special{pa 1390 2630}%
\special{pa 1520 2630}%
\special{fp}%
\special{pa 1520 2630}%
\special{pa 1520 2760}%
\special{fp}%
\special{pa 1520 2760}%
\special{pa 1780 2760}%
\special{fp}%
\special{pa 1780 2760}%
\special{pa 1780 2890}%
\special{fp}%
\special{pa 1780 2890}%
\special{pa 1910 2890}%
\special{fp}%
\special{pa 1910 2890}%
\special{pa 1910 3150}%
\special{fp}%
\special{pa 1910 3150}%
\special{pa 2040 3150}%
\special{fp}%
\special{pa 2040 3150}%
\special{pa 2040 3410}%
\special{fp}%
\special{pa 2040 3410}%
\special{pa 2170 3410}%
\special{fp}%
\special{pa 2170 3410}%
\special{pa 2170 3540}%
\special{fp}%
\special{pa 2170 3540}%
\special{pa 2430 3540}%
\special{fp}%
\special{pa 2430 3540}%
\special{pa 2430 3670}%
\special{fp}%
\special{pa 2430 3670}%
\special{pa 2690 3670}%
\special{fp}%
\special{pa 2690 3670}%
\special{pa 2690 3800}%
\special{fp}%
\special{pa 2690 3800}%
\special{pa 3600 3800}%
\special{fp}%
% CIRCLE 2 0 0 0 Black Black  
% 4 1065 2305 1078 2318 1078 2318 1078 2318
% 
\special{sh 1.000}%
\special{ia 1065 2305 18 18 0.0000000 6.2831853}%
\special{pn 8}%
\special{ar 1065 2305 18 18 0.0000000 6.2831853}%
% CIRCLE 2 0 0 0 Black Black  
% 4 1195 2435 1208 2448 1208 2448 1208 2448
% 
\special{sh 1.000}%
\special{ia 1195 2435 18 18 0.0000000 6.2831853}%
\special{pn 8}%
\special{ar 1195 2435 18 18 0.0000000 6.2831853}%
% CIRCLE 2 0 0 0 Black Black  
% 4 1455 2565 1468 2578 1468 2578 1468 2578
% 
\special{sh 1.000}%
\special{ia 1455 2565 18 18 0.0000000 6.2831853}%
\special{pn 8}%
\special{ar 1455 2565 18 18 0.0000000 6.2831853}%
% CIRCLE 2 0 0 0 Black Black  
% 4 1585 2695 1598 2708 1598 2708 1598 2708
% 
\special{sh 1.000}%
\special{ia 1585 2695 18 18 0.0000000 6.2831853}%
\special{pn 8}%
\special{ar 1585 2695 18 18 0.0000000 6.2831853}%
% CIRCLE 2 0 0 0 Black Black  
% 4 1845 2825 1858 2838 1858 2838 1858 2838
% 
\special{sh 1.000}%
\special{ia 1845 2825 18 18 0.0000000 6.2831853}%
\special{pn 8}%
\special{ar 1845 2825 18 18 0.0000000 6.2831853}%
% CIRCLE 2 0 0 0 Black Black  
% 4 1975 3085 1988 3098 1988 3098 1988 3098
% 
\special{sh 1.000}%
\special{ia 1975 3085 18 18 0.0000000 6.2831853}%
\special{pn 8}%
\special{ar 1975 3085 18 18 0.0000000 6.2831853}%
% CIRCLE 2 0 0 0 Black Black  
% 4 2105 3345 2118 3358 2118 3358 2118 3358
% 
\special{sh 1.000}%
\special{ia 2105 3345 18 18 0.0000000 6.2831853}%
\special{pn 8}%
\special{ar 2105 3345 18 18 0.0000000 6.2831853}%
% CIRCLE 2 0 0 0 Black Black  
% 4 2235 3475 2248 3488 2248 3488 2248 3488
% 
\special{sh 1.000}%
\special{ia 2235 3475 18 18 0.0000000 6.2831853}%
\special{pn 8}%
\special{ar 2235 3475 18 18 0.0000000 6.2831853}%
% CIRCLE 2 0 0 0 Black Black  
% 4 2495 3605 2508 3618 2508 3618 2508 3618
% 
\special{sh 1.000}%
\special{ia 2495 3605 18 18 0.0000000 6.2831853}%
\special{pn 8}%
\special{ar 2495 3605 18 18 0.0000000 6.2831853}%
% CIRCLE 2 0 0 0 Black Black  
% 4 2755 3735 2768 3748 2768 3748 2768 3748
% 
\special{sh 1.000}%
\special{ia 2755 3735 18 18 0.0000000 6.2831853}%
\special{pn 8}%
\special{ar 2755 3735 18 18 0.0000000 6.2831853}%
% CIRCLE 2 0 0 0 Black Black  
% 4 1325 2175 1338 2188 1338 2188 1338 2188
% 
\special{sh 1.000}%
\special{ia 1325 2175 18 18 0.0000000 6.2831853}%
\special{pn 8}%
\special{ar 1325 2175 18 18 0.0000000 6.2831853}%
% CIRCLE 2 0 0 0 Black Black  
% 4 1715 2045 1728 2058 1728 2058 1728 2058
% 
\special{sh 1.000}%
\special{ia 1715 2045 18 18 0.0000000 6.2831853}%
\special{pn 8}%
\special{ar 1715 2045 18 18 0.0000000 6.2831853}%
% CIRCLE 2 0 0 0 Black Black  
% 4 2365 3215 2378 3228 2378 3228 2378 3228
% 
\special{sh 1.000}%
\special{ia 2365 3215 18 18 0.0000000 6.2831853}%
\special{pn 8}%
\special{ar 2365 3215 18 18 0.0000000 6.2831853}%
% CIRCLE 2 0 0 0 Black Black  
% 4 2625 2955 2638 2968 2638 2968 2638 2968
% 
\special{sh 1.000}%
\special{ia 2625 2955 18 18 0.0000000 6.2831853}%
\special{pn 8}%
\special{ar 2625 2955 18 18 0.0000000 6.2831853}%
% CIRCLE 2 0 0 0 Black Black  
% 4 2885 1915 2898 1928 2898 1928 2898 1928
% 
\special{sh 1.000}%
\special{ia 2885 1915 18 18 0.0000000 6.2831853}%
\special{pn 8}%
\special{ar 2885 1915 18 18 0.0000000 6.2831853}%
% CIRCLE 2 0 0 0 Black Black  
% 4 3015 1785 3028 1798 3028 1798 3028 1798
% 
\special{sh 1.000}%
\special{ia 3015 1785 18 18 0.0000000 6.2831853}%
\special{pn 8}%
\special{ar 3015 1785 18 18 0.0000000 6.2831853}%
% CIRCLE 2 0 0 0 Black Black  
% 4 3145 1655 3158 1668 3158 1668 3158 1668
% 
\special{sh 1.000}%
\special{ia 3145 1655 18 18 0.0000000 6.2831853}%
\special{pn 8}%
\special{ar 3145 1655 18 18 0.0000000 6.2831853}%
% CIRCLE 2 0 0 0 Black Black  
% 4 3275 1525 3288 1538 3288 1538 3288 1538
% 
\special{sh 1.000}%
\special{ia 3275 1525 18 18 0.0000000 6.2831853}%
\special{pn 8}%
\special{ar 3275 1525 18 18 0.0000000 6.2831853}%
% CIRCLE 2 0 0 0 Black Black  
% 4 3405 1395 3418 1408 3418 1408 3418 1408
% 
\special{sh 1.000}%
\special{ia 3405 1395 18 18 0.0000000 6.2831853}%
\special{pn 8}%
\special{ar 3405 1395 18 18 0.0000000 6.2831853}%
% CIRCLE 2 0 0 0 Black Black  
% 4 3535 1265 3548 1278 3548 1278 3548 1278
% 
\special{sh 1.000}%
\special{ia 3535 1265 18 18 0.0000000 6.2831853}%
\special{pn 8}%
\special{ar 3535 1265 18 18 0.0000000 6.2831853}%
% LINE 2 0 3 0 Black Black  
% 2 1000 1200 3600 3800
% 
\special{pn 8}%
\special{pa 1000 1200}%
\special{pa 3600 3800}%
\special{fp}%
% STR 2 0 3 0 Black White  
% 4 2220 4000 2220 4100 2 0 0 0
% $J_{9}$
\put(22.2000,-41.0000){\makebox(0,0)[lb]{$J_{9}$}}%
% LINE 2 2 3 0 Black White  
% 4 2365 3215 2365 2565 2365 2565 1325 2565
% 
\special{pn 8}%
\special{pa 2365 3215}%
\special{pa 2365 2565}%
\special{dt 0.045}%
\special{pa 2365 2565}%
\special{pa 1325 2565}%
\special{dt 0.045}%
% LINE 2 2 3 0 Black White  
% 4 2625 2955 2625 2825 2625 2825 1715 2825
% 
\special{pn 8}%
\special{pa 2625 2955}%
\special{pa 2625 2825}%
\special{dt 0.045}%
\special{pa 2625 2825}%
\special{pa 1715 2825}%
\special{dt 0.045}%
% VECTOR 2 0 3 0 Black White  
% 2 1780 3880 2820 3880
% 
\special{pn 8}%
\special{pa 1780 3880}%
\special{pa 2820 3880}%
\special{fp}%
\special{sh 1}%
\special{pa 2820 3880}%
\special{pa 2753 3860}%
\special{pa 2767 3880}%
\special{pa 2753 3900}%
\special{pa 2820 3880}%
\special{fp}%
% VECTOR 2 0 3 1 Black White  
% 2 2820 3880 1780 3880
% 
\special{pn 8}%
\special{pa 2820 3880}%
\special{pa 1780 3880}%
\special{fp}%
\special{sh 1}%
\special{pa 1780 3880}%
\special{pa 1847 3900}%
\special{pa 1833 3880}%
\special{pa 1847 3860}%
\special{pa 1780 3880}%
\special{fp}%
% BOX 2 5 1 0 Black Black  
% 2 4720 1200 4850 3670
% 
\special{pn 0}%
\special{sh 0.300}%
\special{pa 4720 1200}%
\special{pa 4850 1200}%
\special{pa 4850 3670}%
\special{pa 4720 3670}%
\special{pa 4720 1200}%
\special{ip}%
\special{pn 8}%
\special{pa 4720 1200}%
\special{pa 4850 1200}%
\special{pa 4850 3670}%
\special{pa 4720 3670}%
\special{pa 4720 1200}%
\special{ip}%
% BOX 2 5 1 0 Black Black  
% 2 4330 1200 4460 3670
% 
\special{pn 0}%
\special{sh 0.300}%
\special{pa 4330 1200}%
\special{pa 4460 1200}%
\special{pa 4460 3670}%
\special{pa 4330 3670}%
\special{pa 4330 1200}%
\special{ip}%
\special{pn 8}%
\special{pa 4330 1200}%
\special{pa 4460 1200}%
\special{pa 4460 3670}%
\special{pa 4330 3670}%
\special{pa 4330 1200}%
\special{ip}%
% LINE 2 0 3 0 Black White  
% 2 4200 1200 4200 3670
% 
\special{pn 8}%
\special{pa 4200 1200}%
\special{pa 4200 3670}%
\special{fp}%
% LINE 2 0 3 0 Black White  
% 2 4330 1200 4330 3670
% 
\special{pn 8}%
\special{pa 4330 1200}%
\special{pa 4330 3670}%
\special{fp}%
% LINE 2 0 3 0 Black White  
% 2 4460 1200 4460 3670
% 
\special{pn 8}%
\special{pa 4460 1200}%
\special{pa 4460 3670}%
\special{fp}%
% LINE 2 0 3 0 Black White  
% 2 4590 1200 4590 3670
% 
\special{pn 8}%
\special{pa 4590 1200}%
\special{pa 4590 3670}%
\special{fp}%
% LINE 2 0 3 0 Black White  
% 2 4720 1200 4720 3670
% 
\special{pn 8}%
\special{pa 4720 1200}%
\special{pa 4720 3670}%
\special{fp}%
% LINE 2 0 3 0 Black White  
% 2 4850 1200 4850 3670
% 
\special{pn 8}%
\special{pa 4850 1200}%
\special{pa 4850 3670}%
\special{fp}%
% LINE 2 0 3 0 Black White  
% 2 4980 1200 4980 3670
% 
\special{pn 8}%
\special{pa 4980 1200}%
\special{pa 4980 3670}%
\special{fp}%
% LINE 2 0 3 0 Black White  
% 2 5110 1200 5110 3670
% 
\special{pn 8}%
\special{pa 5110 1200}%
\special{pa 5110 3670}%
\special{fp}%
% LINE 2 0 3 0 Black White  
% 2 5240 1200 5240 3670
% 
\special{pn 8}%
\special{pa 5240 1200}%
\special{pa 5240 3670}%
\special{fp}%
% LINE 2 0 3 0 Black White  
% 2 5370 1200 5370 3670
% 
\special{pn 8}%
\special{pa 5370 1200}%
\special{pa 5370 3670}%
\special{fp}%
% LINE 2 0 3 0 Black White  
% 2 5500 1200 5500 3670
% 
\special{pn 8}%
\special{pa 5500 1200}%
\special{pa 5500 3670}%
\special{fp}%
% LINE 2 0 3 0 Black White  
% 2 5630 1200 5630 3670
% 
\special{pn 8}%
\special{pa 5630 1200}%
\special{pa 5630 3670}%
\special{fp}%
% LINE 2 0 3 0 Black White  
% 2 5760 1200 5760 3670
% 
\special{pn 8}%
\special{pa 5760 1200}%
\special{pa 5760 3670}%
\special{fp}%
% LINE 2 0 3 0 Black White  
% 2 5890 1200 5890 3670
% 
\special{pn 8}%
\special{pa 5890 1200}%
\special{pa 5890 3670}%
\special{fp}%
% LINE 2 0 3 0 Black White  
% 2 6020 1200 6020 3670
% 
\special{pn 8}%
\special{pa 6020 1200}%
\special{pa 6020 3670}%
\special{fp}%
% LINE 2 0 3 0 Black White  
% 2 6150 1200 6150 3670
% 
\special{pn 8}%
\special{pa 6150 1200}%
\special{pa 6150 3670}%
\special{fp}%
% LINE 2 0 3 0 Black White  
% 2 6280 1200 6280 3670
% 
\special{pn 8}%
\special{pa 6280 1200}%
\special{pa 6280 3670}%
\special{fp}%
% LINE 2 0 3 0 Black White  
% 2 6410 1200 6410 3670
% 
\special{pn 8}%
\special{pa 6410 1200}%
\special{pa 6410 3670}%
\special{fp}%
% LINE 2 0 3 0 Black White  
% 2 6540 1200 6540 3670
% 
\special{pn 8}%
\special{pa 6540 1200}%
\special{pa 6540 3670}%
\special{fp}%
% CIRCLE 2 0 0 0 Black Black  
% 4 4265 2305 4278 2318 4278 2318 4278 2318
% 
\special{sh 1.000}%
\special{ia 4265 2305 18 18 0.0000000 6.2831853}%
\special{pn 8}%
\special{ar 4265 2305 18 18 0.0000000 6.2831853}%
% CIRCLE 2 0 0 0 Black Black  
% 4 4525 2435 4538 2448 4538 2448 4538 2448
% 
\special{sh 1.000}%
\special{ia 4525 2435 18 18 0.0000000 6.2831853}%
\special{pn 8}%
\special{ar 4525 2435 18 18 0.0000000 6.2831853}%
% CIRCLE 2 0 0 0 Black Black  
% 4 4655 2565 4668 2578 4668 2578 4668 2578
% 
\special{sh 1.000}%
\special{ia 4655 2565 18 18 0.0000000 6.2831853}%
\special{pn 8}%
\special{ar 4655 2565 18 18 0.0000000 6.2831853}%
% CIRCLE 2 0 0 0 Black Black  
% 4 4915 2695 4928 2708 4928 2708 4928 2708
% 
\special{sh 1.000}%
\special{ia 4915 2695 18 18 0.0000000 6.2831853}%
\special{pn 8}%
\special{ar 4915 2695 18 18 0.0000000 6.2831853}%
% CIRCLE 2 0 0 0 Black Black  
% 4 5045 2955 5058 2968 5058 2968 5058 2968
% 
\special{sh 1.000}%
\special{ia 5045 2955 18 18 0.0000000 6.2831853}%
\special{pn 8}%
\special{ar 5045 2955 18 18 0.0000000 6.2831853}%
% CIRCLE 2 0 0 0 Black Black  
% 4 5175 3215 5188 3228 5188 3228 5188 3228
% 
\special{sh 1.000}%
\special{ia 5175 3215 18 18 0.0000000 6.2831853}%
\special{pn 8}%
\special{ar 5175 3215 18 18 0.0000000 6.2831853}%
% CIRCLE 2 0 0 0 Black Black  
% 4 5305 3345 5318 3358 5318 3358 5318 3358
% 
\special{sh 1.000}%
\special{ia 5305 3345 18 18 0.0000000 6.2831853}%
\special{pn 8}%
\special{ar 5305 3345 18 18 0.0000000 6.2831853}%
% CIRCLE 2 0 0 0 Black Black  
% 4 5565 3475 5578 3488 5578 3488 5578 3488
% 
\special{sh 1.000}%
\special{ia 5565 3475 18 18 0.0000000 6.2831853}%
\special{pn 8}%
\special{ar 5565 3475 18 18 0.0000000 6.2831853}%
% CIRCLE 2 0 0 0 Black Black  
% 4 5825 3605 5838 3618 5838 3618 5838 3618
% 
\special{sh 1.000}%
\special{ia 5825 3605 18 18 0.0000000 6.2831853}%
\special{pn 8}%
\special{ar 5825 3605 18 18 0.0000000 6.2831853}%
% CIRCLE 2 0 0 0 Black Black  
% 4 4395 2175 4408 2188 4408 2188 4408 2188
% 
\special{sh 1.000}%
\special{ia 4395 2175 18 18 0.0000000 6.2831853}%
\special{pn 8}%
\special{ar 4395 2175 18 18 0.0000000 6.2831853}%
% CIRCLE 2 0 0 0 Black Black  
% 4 4785 2045 4798 2058 4798 2058 4798 2058
% 
\special{sh 1.000}%
\special{ia 4785 2045 18 18 0.0000000 6.2831853}%
\special{pn 8}%
\special{ar 4785 2045 18 18 0.0000000 6.2831853}%
% CIRCLE 2 0 0 0 Black Black  
% 4 5435 3085 5448 3098 5448 3098 5448 3098
% 
\special{sh 1.000}%
\special{ia 5435 3085 18 18 0.0000000 6.2831853}%
\special{pn 8}%
\special{ar 5435 3085 18 18 0.0000000 6.2831853}%
% CIRCLE 2 0 0 0 Black Black  
% 4 5695 2825 5708 2838 5708 2838 5708 2838
% 
\special{sh 1.000}%
\special{ia 5695 2825 18 18 0.0000000 6.2831853}%
\special{pn 8}%
\special{ar 5695 2825 18 18 0.0000000 6.2831853}%
% CIRCLE 2 0 0 0 Black Black  
% 4 5955 1915 5968 1928 5968 1928 5968 1928
% 
\special{sh 1.000}%
\special{ia 5955 1915 18 18 0.0000000 6.2831853}%
\special{pn 8}%
\special{ar 5955 1915 18 18 0.0000000 6.2831853}%
% CIRCLE 2 0 0 0 Black Black  
% 4 6085 1785 6098 1798 6098 1798 6098 1798
% 
\special{sh 1.000}%
\special{ia 6085 1785 18 18 0.0000000 6.2831853}%
\special{pn 8}%
\special{ar 6085 1785 18 18 0.0000000 6.2831853}%
% CIRCLE 2 0 0 0 Black Black  
% 4 6215 1655 6228 1668 6228 1668 6228 1668
% 
\special{sh 1.000}%
\special{ia 6215 1655 18 18 0.0000000 6.2831853}%
\special{pn 8}%
\special{ar 6215 1655 18 18 0.0000000 6.2831853}%
% CIRCLE 2 0 0 0 Black Black  
% 4 6345 1525 6358 1538 6358 1538 6358 1538
% 
\special{sh 1.000}%
\special{ia 6345 1525 18 18 0.0000000 6.2831853}%
\special{pn 8}%
\special{ar 6345 1525 18 18 0.0000000 6.2831853}%
% CIRCLE 2 0 0 0 Black Black  
% 4 6475 1395 6488 1408 6488 1408 6488 1408
% 
\special{sh 1.000}%
\special{ia 6475 1395 18 18 0.0000000 6.2831853}%
\special{pn 8}%
\special{ar 6475 1395 18 18 0.0000000 6.2831853}%
% CIRCLE 2 0 0 0 Black Black  
% 4 6605 1265 6618 1278 6618 1278 6618 1278
% 
\special{sh 1.000}%
\special{ia 6605 1265 18 18 0.0000000 6.2831853}%
\special{pn 8}%
\special{ar 6605 1265 18 18 0.0000000 6.2831853}%
% LINE 2 2 3 0 Black White  
% 4 5435 3085 5435 2435 5435 2435 4395 2435
% 
\special{pn 8}%
\special{pa 5435 3085}%
\special{pa 5435 2435}%
\special{dt 0.045}%
\special{pa 5435 2435}%
\special{pa 4395 2435}%
\special{dt 0.045}%
% LINE 2 2 3 0 Black White  
% 4 5695 2825 5695 2695 5695 2695 4785 2695
% 
\special{pn 8}%
\special{pa 5695 2825}%
\special{pa 5695 2695}%
\special{dt 0.045}%
\special{pa 5695 2695}%
\special{pa 4785 2695}%
\special{dt 0.045}%
% LINE 0 0 3 0 Black White  
% 38 4200 1200 4200 2370 4200 2370 4460 2370 4460 2370 4460 2500 4460 2500 4590 2500 4590 2500 4590 2630 4590 2630 4850 2630 4850 2630 4850 2760 4850 2760 4980 2760 4980 2760 4980 3020 4980 3020 5110 3020 5110 3020 5110 3280 5110 3280 5240 3280 5240 3280 5240 3410 5240 3410 5370 3410 5370 3410 5500 3410 5500 3410 5500 3540 5500 3540 5760 3540 5760 3540 5760 3670 5760 3670 6670 3670
% 
\special{pn 20}%
\special{pa 4200 1200}%
\special{pa 4200 2370}%
\special{fp}%
\special{pa 4200 2370}%
\special{pa 4460 2370}%
\special{fp}%
\special{pa 4460 2370}%
\special{pa 4460 2500}%
\special{fp}%
\special{pa 4460 2500}%
\special{pa 4590 2500}%
\special{fp}%
\special{pa 4590 2500}%
\special{pa 4590 2630}%
\special{fp}%
\special{pa 4590 2630}%
\special{pa 4850 2630}%
\special{fp}%
\special{pa 4850 2630}%
\special{pa 4850 2760}%
\special{fp}%
\special{pa 4850 2760}%
\special{pa 4980 2760}%
\special{fp}%
\special{pa 4980 2760}%
\special{pa 4980 3020}%
\special{fp}%
\special{pa 4980 3020}%
\special{pa 5110 3020}%
\special{fp}%
\special{pa 5110 3020}%
\special{pa 5110 3280}%
\special{fp}%
\special{pa 5110 3280}%
\special{pa 5240 3280}%
\special{fp}%
\special{pa 5240 3280}%
\special{pa 5240 3410}%
\special{fp}%
\special{pa 5240 3410}%
\special{pa 5370 3410}%
\special{fp}%
\special{pa 5370 3410}%
\special{pa 5500 3410}%
\special{fp}%
\special{pa 5500 3410}%
\special{pa 5500 3540}%
\special{fp}%
\special{pa 5500 3540}%
\special{pa 5760 3540}%
\special{fp}%
\special{pa 5760 3540}%
\special{pa 5760 3670}%
\special{fp}%
\special{pa 5760 3670}%
\special{pa 6670 3670}%
\special{fp}%
% LINE 2 0 3 0 Black White  
% 2 6670 1200 4200 1200
% 
\special{pn 8}%
\special{pa 6670 1200}%
\special{pa 4200 1200}%
\special{fp}%
% LINE 2 0 3 0 Black White  
% 2 6670 1330 4200 1330
% 
\special{pn 8}%
\special{pa 6670 1330}%
\special{pa 4200 1330}%
\special{fp}%
% LINE 2 0 3 0 Black White  
% 2 6670 1460 4200 1460
% 
\special{pn 8}%
\special{pa 6670 1460}%
\special{pa 4200 1460}%
\special{fp}%
% LINE 2 0 3 0 Black White  
% 2 6670 1590 4200 1590
% 
\special{pn 8}%
\special{pa 6670 1590}%
\special{pa 4200 1590}%
\special{fp}%
% LINE 2 0 3 0 Black White  
% 2 6670 1720 4200 1720
% 
\special{pn 8}%
\special{pa 6670 1720}%
\special{pa 4200 1720}%
\special{fp}%
% LINE 2 0 3 0 Black White  
% 2 6670 1850 4200 1850
% 
\special{pn 8}%
\special{pa 6670 1850}%
\special{pa 4200 1850}%
\special{fp}%
% LINE 2 0 3 0 Black White  
% 2 6670 1980 4200 1980
% 
\special{pn 8}%
\special{pa 6670 1980}%
\special{pa 4200 1980}%
\special{fp}%
% LINE 2 0 3 0 Black White  
% 2 6670 2110 4200 2110
% 
\special{pn 8}%
\special{pa 6670 2110}%
\special{pa 4200 2110}%
\special{fp}%
% LINE 2 0 3 0 Black White  
% 2 6670 2240 4200 2240
% 
\special{pn 8}%
\special{pa 6670 2240}%
\special{pa 4200 2240}%
\special{fp}%
% LINE 2 0 3 0 Black White  
% 2 6670 2370 4200 2370
% 
\special{pn 8}%
\special{pa 6670 2370}%
\special{pa 4200 2370}%
\special{fp}%
% LINE 2 0 3 0 Black White  
% 2 6670 2500 4200 2500
% 
\special{pn 8}%
\special{pa 6670 2500}%
\special{pa 4200 2500}%
\special{fp}%
% LINE 2 0 3 0 Black White  
% 2 6670 2630 4200 2630
% 
\special{pn 8}%
\special{pa 6670 2630}%
\special{pa 4200 2630}%
\special{fp}%
% LINE 2 0 3 0 Black White  
% 2 6670 2760 4200 2760
% 
\special{pn 8}%
\special{pa 6670 2760}%
\special{pa 4200 2760}%
\special{fp}%
% LINE 2 0 3 0 Black White  
% 2 6670 2890 4200 2890
% 
\special{pn 8}%
\special{pa 6670 2890}%
\special{pa 4200 2890}%
\special{fp}%
% LINE 2 0 3 0 Black White  
% 2 6670 3020 4200 3020
% 
\special{pn 8}%
\special{pa 6670 3020}%
\special{pa 4200 3020}%
\special{fp}%
% LINE 2 0 3 0 Black White  
% 2 6670 3150 4200 3150
% 
\special{pn 8}%
\special{pa 6670 3150}%
\special{pa 4200 3150}%
\special{fp}%
% LINE 2 0 3 0 Black White  
% 2 6670 3280 4200 3280
% 
\special{pn 8}%
\special{pa 6670 3280}%
\special{pa 4200 3280}%
\special{fp}%
% LINE 2 0 3 0 Black White  
% 2 6670 3410 4200 3410
% 
\special{pn 8}%
\special{pa 6670 3410}%
\special{pa 4200 3410}%
\special{fp}%
% LINE 2 0 3 0 Black White  
% 2 6670 3540 4200 3540
% 
\special{pn 8}%
\special{pa 6670 3540}%
\special{pa 4200 3540}%
\special{fp}%
% LINE 2 0 3 0 Black White  
% 2 6670 3670 4200 3670
% 
\special{pn 8}%
\special{pa 6670 3670}%
\special{pa 4200 3670}%
\special{fp}%
% LINE 2 0 3 0 Black White  
% 2 6670 1200 6670 3670
% 
\special{pn 8}%
\special{pa 6670 1200}%
\special{pa 6670 3670}%
\special{fp}%
% STR 2 0 3 0 Black White  
% 4 4980 3880 4980 3980 2 0 0 0
% $J'_{7}$
\put(49.8000,-39.8000){\makebox(0,0)[lb]{$J'_{7}$}}%
% STR 2 0 3 0 Black White  
% 4 5480 3880 5480 3980 2 0 0 0
% $J'_{9}$
\put(54.8000,-39.8000){\makebox(0,0)[lb]{$J'_{9}$}}%
% VECTOR 2 0 3 0 Black White  
% 2 4850 3740 5230 3740
% 
\special{pn 8}%
\special{pa 4850 3740}%
\special{pa 5230 3740}%
\special{fp}%
\special{sh 1}%
\special{pa 5230 3740}%
\special{pa 5163 3720}%
\special{pa 5177 3740}%
\special{pa 5163 3760}%
\special{pa 5230 3740}%
\special{fp}%
% VECTOR 2 0 3 0 Black White  
% 2 5890 3740 5250 3740
% 
\special{pn 8}%
\special{pa 5890 3740}%
\special{pa 5250 3740}%
\special{fp}%
\special{sh 1}%
\special{pa 5250 3740}%
\special{pa 5317 3760}%
\special{pa 5303 3740}%
\special{pa 5317 3720}%
\special{pa 5250 3740}%
\special{fp}%
% VECTOR 2 0 3 0 Black White  
% 2 5250 3740 5890 3740
% 
\special{pn 8}%
\special{pa 5250 3740}%
\special{pa 5890 3740}%
\special{fp}%
\special{sh 1}%
\special{pa 5890 3740}%
\special{pa 5823 3720}%
\special{pa 5837 3740}%
\special{pa 5823 3760}%
\special{pa 5890 3740}%
\special{fp}%
% VECTOR 2 0 3 0 Black White  
% 2 5230 3740 4850 3740
% 
\special{pn 8}%
\special{pa 5230 3740}%
\special{pa 4850 3740}%
\special{fp}%
\special{sh 1}%
\special{pa 4850 3740}%
\special{pa 4917 3760}%
\special{pa 4903 3740}%
\special{pa 4917 3720}%
\special{pa 4850 3740}%
\special{fp}%
% VECTOR 2 0 3 0 Black White  
% 2 3210 3880 3600 3880
% 
\special{pn 8}%
\special{pa 3210 3880}%
\special{pa 3600 3880}%
\special{fp}%
\special{sh 1}%
\special{pa 3600 3880}%
\special{pa 3533 3860}%
\special{pa 3547 3880}%
\special{pa 3533 3900}%
\special{pa 3600 3880}%
\special{fp}%
% VECTOR 2 0 3 0 Black White  
% 2 3600 3880 3210 3880
% 
\special{pn 8}%
\special{pa 3600 3880}%
\special{pa 3210 3880}%
\special{fp}%
\special{sh 1}%
\special{pa 3210 3880}%
\special{pa 3277 3900}%
\special{pa 3263 3880}%
\special{pa 3277 3860}%
\special{pa 3210 3880}%
\special{fp}%
% STR 2 0 3 0 Black White  
% 4 3320 4000 3320 4100 2 0 0 0
% $J_{19}$
\put(33.2000,-41.0000){\makebox(0,0)[lb]{$J_{19}$}}%
% STR 2 0 3 0 Black White  
% 4 6390 3880 6390 3980 2 0 0 0
% $J'_{18}$
\put(63.9000,-39.8000){\makebox(0,0)[lb]{$J'_{18}$}}%
% VECTOR 2 0 3 0 Black White  
% 2 6280 3743 6670 3743
% 
\special{pn 8}%
\special{pa 6280 3743}%
\special{pa 6670 3743}%
\special{fp}%
\special{sh 1}%
\special{pa 6670 3743}%
\special{pa 6603 3723}%
\special{pa 6617 3743}%
\special{pa 6603 3763}%
\special{pa 6670 3743}%
\special{fp}%
% VECTOR 2 0 3 0 Black White  
% 2 6670 3740 6280 3740
% 
\special{pn 8}%
\special{pa 6670 3740}%
\special{pa 6280 3740}%
\special{fp}%
\special{sh 1}%
\special{pa 6280 3740}%
\special{pa 6347 3760}%
\special{pa 6333 3740}%
\special{pa 6347 3720}%
\special{pa 6280 3740}%
\special{fp}%
% STR 2 0 3 0 Black White  
% 4 5000 970 5000 1070 2 0 0 0
% $h':[19]\rightarrow[19]$
\put(50.0000,-10.7000){\makebox(0,0)[lb]{$h':[19]\rightarrow[19]$}}%
% BOX 2 5 3 0 Black White  
% 2 800 910 7000 4010
% 
\special{pn 8}%
\special{pa 800 910}%
\special{pa 7000 910}%
\special{pa 7000 4010}%
\special{pa 800 4010}%
\special{pa 800 910}%
\special{ip}%
% LINE 2 0 3 0 Black Black  
% 2 4200 1200 6670 3670
% 
\special{pn 8}%
\special{pa 4200 1200}%
\special{pa 6670 3670}%
\special{fp}%
\end{picture}}%
\caption{The boundary paths and pivots of $\w{h}$ and $\w{h'}$.}
\label{pic:example}
\end{figure}

The weights of anti-canonical bundles of $\Hess{S}{h}$ and  $\Hess{S}{h'}$ are given by
\begin{align*}
&\xi_h = \varpi_1+2\varpi_2+\varpi_3+\varpi_4+2\varpi_5+\varpi_6\hspace{20pt}\hspace{28.5pt}+2\varpi_{14}+\varpi_{15}+2\varpi_{16}+\varpi_{17}, \\
&\xi_{h'} = \hspace{26pt}2\varpi_1+\varpi_2+\varpi_3+2\varpi_4+\varpi_5\hspace{10pt}+\varpi_8\hspace{10pt}+2\varpi_{13}+\varpi_{14}+2\varpi_{15}+\varpi_{16},
\end{align*}
where $\varpi_8=\varpi_{h(1)-1}$ (cf.\ Lemma~\ref{l:calculation_of_xi_for_h_prime}).
In particular, this provides a pair of examples in Case 2-b since the coefficient of $\varpi_9=\varpi_{h(1)}$ in $\xi_h$ is $0$.
From this, we see that
\begin{align*}
J = \{J_9, J_{19}\} \ \ \text{with} \ \ J_9=\{7,8,\ldots,14\} , J_{19}=\{18,19,20\},
\end{align*}
and that
\begin{align*}
J' = \{J'_7, J'_9, J'_{18}\} \ \ \text{with} \ \ J'_7=\{6,7,8\} , J'_9=\{9,10,\ldots,13\} , J'_{18} = \{17,18,19\}.
\end{align*}
In one-line notation, we have 
\begin{align*}
\hspace{17pt}
\w{h'} = \ 
 \hspace{4pt}9 \hspace{12pt} 8 \hspace{8pt} 10 \hspace{6pt} 11 \hspace{9pt} 7 \hspace{8pt} 12 \hspace{6pt} 14 \hspace{6pt} 16 \hspace{6pt} 17 \hspace{6pt} 15 \hspace{6pt} 18 \hspace{6pt} 13 \hspace{6pt} 19 \hspace{9pt} 6 \hspace{12pt} 5 \hspace{12pt} 4 \hspace{12pt} 3 \hspace{12pt} 2 \hspace{12pt} 1 
\end{align*}
(see Figure \ref{pic:example}).
Take $u'\in\mathfrak{S}_{J'}$ as
\begin{align*}
\hspace{31pt}
u' = \ 
 \hspace{3pt} 1 \hspace{12pt} 2 \hspace{12pt} 3 \hspace{12pt} 4 \hspace{11pt} 5 \hspace{9pt} \fbox{6 \hspace{4pt} 7 \hspace{4pt} 8} \hspace{2pt} \fbox{13 \hspace{-2pt} 10\hspace{2pt} 12\hspace{5pt} 9\hspace{4pt} 11 \hspace{-5pt}} \hspace{4pt} 14 \hspace{6pt} 15 \hspace{6pt} 16 \hspace{6pt} \fbox{\hspace{-2pt}19 \hspace{-3pt} 18 \hspace{-3pt} 17 \hspace{-5pt}},
\end{align*}
where we emphasized the positions of $J'_7, J'_9, J'_{18}$ by enclosing the numbers of $\w{h'}$ on $J'$.
Since these two equalities imply that
\begin{align*}
\hspace{14pt}
u' \w{h'} = \ 
 13 \hspace{10pt} 8 \hspace{8pt} 10 \hspace{6pt} 12 \hspace{9pt} 7 \hspace{8pt} \fbox{9 \hspace{0pt} 14 \hspace{-2pt} 16} \hspace{2pt} \fbox{\hspace{-2pt}19 \hspace{-2pt} 15\hspace{2pt} 18\hspace{3pt} 11\hspace{2pt} 17 \hspace{-6pt}} \hspace{8pt} 6 \hspace{11pt} 5 \hspace{11pt} 4 \hspace{9pt} \fbox{\hspace{0pt}3 \hspace{4pt} 2 \hspace{4pt} 1 \hspace{-4pt}},
\end{align*}
it is now straightforward to verify that $u'$ satisfies conditions (i)-(iv) directly. 
We also have
\begin{align*}
\hspace{6pt}
\w{h} = \ 
 \hspace{4pt}9 \hspace{8pt} 10 \hspace{10pt} 8 \hspace{8pt} 11 \hspace{6pt} 12 \hspace{9pt} 7 \hspace{8pt} 13 \hspace{6pt} 15 \hspace{6pt} 17 \hspace{6pt} 18\hspace{6pt} 16\hspace{6pt} 19\hspace{6pt} 14 \hspace{6pt} 20 \hspace{9pt} 6 \hspace{12pt} 5 \hspace{11pt} 4 \hspace{12pt} 3 \hspace{12pt} 2 \hspace{12pt} 1,
\end{align*}
and hence it follows from \eqref{eq: two permutations} that
\begin{align*}
\bar{u}' &= \ 
 \hspace{3pt}1 \hspace{12pt} 2 \hspace{12pt} 3 \hspace{12pt} 4 \hspace{11pt} 5 \hspace{12pt} 6 \hspace{9pt} \fbox{7 \hspace{4pt} 8 \hspace{4pt} 9 \hspace{1pt} 14\hspace{2pt} 11\hspace{2pt} 13\hspace{2pt} 10 \hspace{-2pt} 12 \hspace{-5pt}} \hspace{4pt} 15 \hspace{5pt} 16 \hspace{6pt} 17 \hspace{5pt} \fbox{\hspace{-1pt}20 \hspace{-4pt} 19 \hspace{-3pt} 18 \hspace{-5pt}}, \\
\bar{u}' \w{h} &= \ 
 \hspace{3pt}9 \hspace{9pt} 14 \hspace{9pt} 8 \hspace{9pt} 11 \hspace{6pt} 13 \hspace{9pt} 7 \hspace{8pt} \fbox{\hspace{-6pt} 10 \hspace{-2pt} 15 \hspace{-2pt} 17 \hspace{-2pt} 20\hspace{2pt} 16\hspace{2pt} 19\hspace{2pt} 12 \hspace{-2pt} 18 \hspace{-5pt}} \hspace{7pt} 6 \hspace{11pt} 5 \hspace{11pt} 4 \hspace{9pt} \fbox{\hspace{0pt}3 \hspace{4pt} 2 \hspace{4pt} 1 \hspace{-4pt}},
\end{align*}
where we emphasized $J_9=J_{h(1)}$ and $J_{19}$ by enclosing the numbers of $\w{h}$ on $J$.
From this, we see that $\Mh=1<2=\lp{h}$, and that $\m{0}=2$, $\m{1}=1$, $\m{2}=0$ (see Figure~\ref{pic:example_mi}).
\begin{figure}[htbp]
%WinTpicVersion4.32a
{\unitlength 0.1in%
\begin{picture}(41.2000,26.6000)(13.9000,-39.6000)%
% CIRCLE 2 0 0 0 Black Black  
% 4 2553 1962 2569 1979 2569 1979 2569 1979
% 
\special{sh 1.000}%
\special{ia 2553 1962 23 23 0.0000000 6.2831853}%
\special{pn 8}%
\special{ar 2553 1962 23 23 0.0000000 6.2831853}%
% CIRCLE 2 0 0 0 Black Black  
% 4 2719 2793 2736 2810 2736 2810 2736 2810
% 
\special{sh 1.000}%
\special{ia 2719 2793 24 24 0.0000000 6.2831853}%
\special{pn 8}%
\special{ar 2719 2793 24 24 0.0000000 6.2831853}%
% CIRCLE 2 0 0 0 Black Black  
% 4 2886 3126 2902 3142 2902 3142 2902 3142
% 
\special{sh 1.000}%
\special{ia 2886 3126 23 23 0.0000000 6.2831853}%
\special{pn 8}%
\special{ar 2886 3126 23 23 0.0000000 6.2831853}%
% CIRCLE 2 0 0 0 Black Black  
% 4 3550 2295 3567 2311 3567 2311 3567 2311
% 
\special{sh 1.000}%
\special{ia 3550 2295 23 23 0.0000000 6.2831853}%
\special{pn 8}%
\special{ar 3550 2295 23 23 0.0000000 6.2831853}%
% CIRCLE 2 0 0 0 Black Black  
% 4 3218 2960 3234 2977 3234 2977 3234 2977
% 
\special{sh 1.000}%
\special{ia 3218 2960 23 23 0.0000000 6.2831853}%
\special{pn 8}%
\special{ar 3218 2960 23 23 0.0000000 6.2831853}%
% CIRCLE 2 0 0 0 Black Black  
% 4 3717 3292 3734 3308 3734 3308 3734 3308
% 
\special{sh 1.000}%
\special{ia 3717 3292 23 23 0.0000000 6.2831853}%
\special{pn 8}%
\special{ar 3717 3292 23 23 0.0000000 6.2831853}%
% CIRCLE 2 0 0 0 Black Black  
% 4 3384 3458 3401 3475 3401 3475 3401 3475
% 
\special{sh 1.000}%
\special{ia 3384 3458 24 24 0.0000000 6.2831853}%
\special{pn 8}%
\special{ar 3384 3458 24 24 0.0000000 6.2831853}%
% CIRCLE 2 0 0 0 Black Black  
% 4 3052 3625 3069 3641 3069 3641 3069 3641
% 
\special{sh 1.000}%
\special{ia 3052 3625 23 23 0.0000000 6.2831853}%
\special{pn 8}%
\special{ar 3052 3625 23 23 0.0000000 6.2831853}%
% VECTOR 2 0 3 0 Black Black  
% 4 2470 3791 2968 3791 2968 3791 2470 3791
% 
\special{pn 8}%
\special{pa 2470 3791}%
\special{pa 2968 3791}%
\special{fp}%
\special{sh 1}%
\special{pa 2968 3791}%
\special{pa 2901 3771}%
\special{pa 2915 3791}%
\special{pa 2901 3811}%
\special{pa 2968 3791}%
\special{fp}%
\special{pa 2968 3791}%
\special{pa 2470 3791}%
\special{fp}%
\special{sh 1}%
\special{pa 2470 3791}%
\special{pa 2537 3811}%
\special{pa 2523 3791}%
\special{pa 2537 3771}%
\special{pa 2470 3791}%
\special{fp}%
% VECTOR 2 0 3 0 Black Black  
% 4 2968 3791 3800 3791 3800 3791 2968 3791
% 
\special{pn 8}%
\special{pa 2968 3791}%
\special{pa 3800 3791}%
\special{fp}%
\special{sh 1}%
\special{pa 3800 3791}%
\special{pa 3733 3771}%
\special{pa 3747 3791}%
\special{pa 3733 3811}%
\special{pa 3800 3791}%
\special{fp}%
\special{pa 3800 3791}%
\special{pa 2968 3791}%
\special{fp}%
\special{sh 1}%
\special{pa 2968 3791}%
\special{pa 3035 3811}%
\special{pa 3021 3791}%
\special{pa 3035 3771}%
\special{pa 2968 3791}%
\special{fp}%
% STR 2 0 3 0 Black Black  
% 4 3135 1384 3135 1466 2 0 0 0
% $\bar{u}'\w{h}$
\put(31.3500,-14.6600){\makebox(0,0)[lb]{$\bar{u}'\w{h}$}}%
% LINE 2 0 3 0 Black Black  
% 2 2968 1300 2968 3960
% 
\special{pn 8}%
\special{pa 2968 1300}%
\special{pa 2968 3960}%
\special{fp}%
% STR 2 0 3 0 Black Black  
% 4 2610 3907 2610 3990 2 0 0 0
% $J_{9}^-$
\put(26.1000,-39.9000){\makebox(0,0)[lb]{$J_{9}^-$}}%
% STR 2 0 3 0 Black Black  
% 4 3275 3907 3275 3990 2 0 0 0
% $J_{9}^+$
\put(32.7500,-39.9000){\makebox(0,0)[lb]{$J_{9}^+$}}%
% LINE 2 2 3 0 Black Black  
% 2 2270 1970 3801 1970
% 
\special{pn 8}%
\special{pa 2270 1970}%
\special{pa 3801 1970}%
\special{dt 0.045}%
% LINE 2 2 3 0 Black Black  
% 2 2270 2800 4571 2800
% 
\special{pn 8}%
\special{pa 2270 2800}%
\special{pa 4571 2800}%
\special{dt 0.045}%
% LINE 2 2 3 0 Black Black  
% 2 2270 3130 5321 3130
% 
\special{pn 8}%
\special{pa 2270 3130}%
\special{pa 5321 3130}%
\special{dt 0.045}%
% STR 2 0 3 0 Black Black  
% 4 4780 2128 4780 2210 2 0 0 0
% $m_1=1$
\put(47.8000,-22.1000){\makebox(0,0)[lb]{$m_1=1$}}%
% STR 2 0 3 0 Black Black  
% 4 4040 1708 4040 1790 2 0 0 0
% $m_2=0$
\put(40.4000,-17.9000){\makebox(0,0)[lb]{$m_2=0$}}%
% STR 2 0 3 0 Black Black  
% 4 5510 2298 5510 2380 2 0 0 0
% $m_0=2$
\put(55.1000,-23.8000){\makebox(0,0)[lb]{$m_0=2$}}%
% STR 2 0 3 0 Black Black  
% 4 3890 1858 3890 1940 2 0 0 0
% \rotatebox{90}{$\underbrace{\hspace{30pt}}$}
\put(38.9000,-19.4000){\makebox(0,0)[lb]{\rotatebox{90}{$\underbrace{\hspace{30pt}}$}}}%
% STR 2 0 3 0 Black Black  
% 4 4640 2688 4640 2770 2 0 0 0
% \rotatebox{90}{$\underbrace{\hspace{90pt}}$}
\put(46.4000,-27.7000){\makebox(0,0)[lb]{\rotatebox{90}{$\underbrace{\hspace{90pt}}$}}}%
% STR 2 0 3 0 Black Black  
% 4 5370 3018 5370 3100 2 0 0 0
% \rotatebox{90}{$\underbrace{\hspace{114pt}}$}
\put(53.7000,-31.0000){\makebox(0,0)[lb]{\rotatebox{90}{$\underbrace{\hspace{114pt}}$}}}%
% VECTOR 2 0 3 0 Black Black  
% 2 2083 1300 2083 3960
% 
\special{pn 8}%
\special{pa 2083 1300}%
\special{pa 2083 3960}%
\special{fp}%
\special{sh 1}%
\special{pa 2083 3960}%
\special{pa 2103 3893}%
\special{pa 2083 3907}%
\special{pa 2063 3893}%
\special{pa 2083 3960}%
\special{fp}%
% LINE 2 0 3 0 Black Black  
% 2 2042 3627 2126 3627
% 
\special{pn 8}%
\special{pa 2042 3627}%
\special{pa 2126 3627}%
\special{fp}%
% STR 2 0 3 0 Black Black  
% 4 1855 3595 1855 3678 2 0 0 0
% 20
\put(18.5500,-36.7800){\makebox(0,0)[lb]{20}}%
% LINE 2 0 3 0 Black Black  
% 2 2042 1798 2126 1798
% 
\special{pn 8}%
\special{pa 2042 1798}%
\special{pa 2126 1798}%
\special{fp}%
% STR 2 0 3 0 Black Black  
% 4 1390 1803 1390 1885 2 0 0 0
% $h(1)=9$
\put(13.9000,-18.8500){\makebox(0,0)[lb]{$h(1)=9$}}%
% LINE 2 0 3 0 Black Black  
% 2 2040 1970 2124 1970
% 
\special{pn 8}%
\special{pa 2040 1970}%
\special{pa 2124 1970}%
\special{fp}%
% STR 2 0 3 0 Black Black  
% 4 1850 1947 1850 2030 2 0 0 0
% 10
\put(18.5000,-20.3000){\makebox(0,0)[lb]{10}}%
\end{picture}}%
\caption{The positions of $1$'s for $\bar{u}' \w{h}$.}
\label{pic:example_mi}
\end{figure}
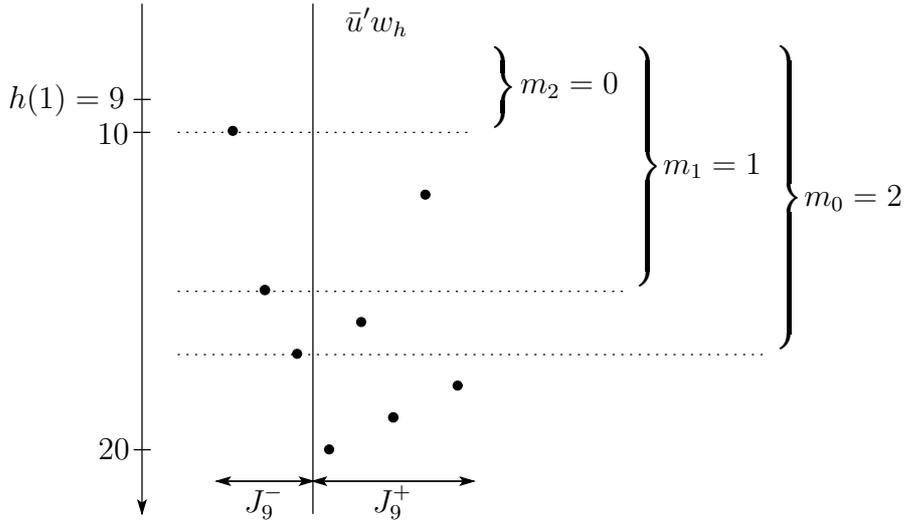
According to the definition of $u$ in Case 2-b, we let $u\coloneqq v_1v_0\bar{u}' =(s_{8})(s_{10}s_{9})\bar{u}' \in \mathfrak{S}_J$. It then follows that
\begin{align*}
\hspace{17pt}
 u = \ 
 \hspace{3pt}1 \hspace{12pt} 2 \hspace{12pt} 3 \hspace{12pt} 4 \hspace{11pt} 5 \hspace{12pt} 6 \hspace{9pt} \fbox{7 \hspace{4pt} 9 \hspace{1pt} 11 \hspace{-2pt} 14\hspace{2pt} 10\hspace{2pt} 13 \hspace{1pt} 8 \hspace{0pt} 12 \hspace{-5pt}} \hspace{4pt} 15 \hspace{5pt} 16 \hspace{6pt} 17 \hspace{6pt} \fbox{\hspace{-1pt}20 \hspace{-3pt} 19 \hspace{-3pt} 18 \hspace{-5pt}}.
\end{align*}
The modification from $\bar{u}' $ to $u=v_1v_0\bar{u}' $ makes
the order of the numbers of $u$ on $J_{h(1)}$ be the same as the order of the numbers of $\bar{u}' \w{h}$ on $J_{h(1)}$. We visualize this in Figure~\ref{pic:example_modification}.
\begin{figure}[htbp]
%WinTpicVersion4.32a
{\unitlength 0.1in%
\begin{picture}(58.2000,22.6000)(5.9000,-31.6000)%
% CIRCLE 2 0 0 0 Black Black  
% 4 1699 1466 1716 1483 1716 1483 1716 1483
% 
\special{sh 1.000}%
\special{ia 1699 1466 24 24 0.0000000 6.2831853}%
\special{pn 8}%
\special{ar 1699 1466 24 24 0.0000000 6.2831853}%
% CIRCLE 2 0 0 0 Black Black  
% 4 1865 1633 1882 1649 1882 1649 1882 1649
% 
\special{sh 1.000}%
\special{ia 1865 1633 23 23 0.0000000 6.2831853}%
\special{pn 8}%
\special{ar 1865 1633 23 23 0.0000000 6.2831853}%
% CIRCLE 2 0 0 0 Black Black  
% 4 2032 1798 2048 1815 2048 1815 2048 1815
% 
\special{sh 1.000}%
\special{ia 2032 1798 23 23 0.0000000 6.2831853}%
\special{pn 8}%
\special{ar 2032 1798 23 23 0.0000000 6.2831853}%
% CIRCLE 2 0 0 0 Black Black  
% 4 2697 1965 2713 1982 2713 1982 2713 1982
% 
\special{sh 1.000}%
\special{ia 2697 1965 23 23 0.0000000 6.2831853}%
\special{pn 8}%
\special{ar 2697 1965 23 23 0.0000000 6.2831853}%
% CIRCLE 2 0 0 0 Black Black  
% 4 2365 2131 2381 2148 2381 2148 2381 2148
% 
\special{sh 1.000}%
\special{ia 2365 2131 23 23 0.0000000 6.2831853}%
\special{pn 8}%
\special{ar 2365 2131 23 23 0.0000000 6.2831853}%
% CIRCLE 2 0 0 0 Black Black  
% 4 2863 2298 2880 2314 2880 2314 2880 2314
% 
\special{sh 1.000}%
\special{ia 2863 2298 23 23 0.0000000 6.2831853}%
\special{pn 8}%
\special{ar 2863 2298 23 23 0.0000000 6.2831853}%
% CIRCLE 2 0 0 0 Black Black  
% 4 2531 2464 2547 2481 2547 2481 2547 2481
% 
\special{sh 1.000}%
\special{ia 2531 2464 23 23 0.0000000 6.2831853}%
\special{pn 8}%
\special{ar 2531 2464 23 23 0.0000000 6.2831853}%
% CIRCLE 2 0 0 0 Black Black  
% 4 2198 2630 2215 2646 2215 2646 2215 2646
% 
\special{sh 1.000}%
\special{ia 2198 2630 23 23 0.0000000 6.2831853}%
\special{pn 8}%
\special{ar 2198 2630 23 23 0.0000000 6.2831853}%
% CIRCLE 2 0 0 0 Black Black  
% 4 4119 1465 4135 1481 4135 1481 4135 1481
% 
\special{sh 1.000}%
\special{ia 4119 1465 23 23 0.0000000 6.2831853}%
\special{pn 8}%
\special{ar 4119 1465 23 23 0.0000000 6.2831853}%
% CIRCLE 2 0 0 0 Black Black  
% 4 4285 1797 4302 1814 4302 1814 4302 1814
% 
\special{sh 1.000}%
\special{ia 4285 1797 24 24 0.0000000 6.2831853}%
\special{pn 8}%
\special{ar 4285 1797 24 24 0.0000000 6.2831853}%
% CIRCLE 2 0 0 0 Black Black  
% 4 4452 2130 4468 2147 4468 2147 4468 2147
% 
\special{sh 1.000}%
\special{ia 4452 2130 23 23 0.0000000 6.2831853}%
\special{pn 8}%
\special{ar 4452 2130 23 23 0.0000000 6.2831853}%
% CIRCLE 2 0 0 0 Black Black  
% 4 5116 1631 5133 1648 5133 1648 5133 1648
% 
\special{sh 1.000}%
\special{ia 5116 1631 24 24 0.0000000 6.2831853}%
\special{pn 8}%
\special{ar 5116 1631 24 24 0.0000000 6.2831853}%
% CIRCLE 2 0 0 0 Black Black  
% 4 4785 1964 4801 1981 4801 1981 4801 1981
% 
\special{sh 1.000}%
\special{ia 4785 1964 23 23 0.0000000 6.2831853}%
\special{pn 8}%
\special{ar 4785 1964 23 23 0.0000000 6.2831853}%
% CIRCLE 2 0 0 0 Black Black  
% 4 5283 2296 5300 2313 5300 2313 5300 2313
% 
\special{sh 1.000}%
\special{ia 5283 2296 24 24 0.0000000 6.2831853}%
\special{pn 8}%
\special{ar 5283 2296 24 24 0.0000000 6.2831853}%
% CIRCLE 2 0 0 0 Black Black  
% 4 4950 2462 4967 2479 4967 2479 4967 2479
% 
\special{sh 1.000}%
\special{ia 4950 2462 24 24 0.0000000 6.2831853}%
\special{pn 8}%
\special{ar 4950 2462 24 24 0.0000000 6.2831853}%
% CIRCLE 2 0 0 0 Black Black  
% 4 4618 2629 4635 2645 4635 2645 4635 2645
% 
\special{sh 1.000}%
\special{ia 4618 2629 23 23 0.0000000 6.2831853}%
\special{pn 8}%
\special{ar 4618 2629 23 23 0.0000000 6.2831853}%
% VECTOR 2 0 3 0 Black Black  
% 4 4030 2986 4528 2986 4528 2986 4030 2986
% 
\special{pn 8}%
\special{pa 4030 2986}%
\special{pa 4528 2986}%
\special{fp}%
\special{sh 1}%
\special{pa 4528 2986}%
\special{pa 4461 2966}%
\special{pa 4475 2986}%
\special{pa 4461 3006}%
\special{pa 4528 2986}%
\special{fp}%
\special{pa 4528 2986}%
\special{pa 4030 2986}%
\special{fp}%
\special{sh 1}%
\special{pa 4030 2986}%
\special{pa 4097 3006}%
\special{pa 4083 2986}%
\special{pa 4097 2966}%
\special{pa 4030 2986}%
\special{fp}%
% VECTOR 2 0 3 0 Black Black  
% 4 1610 2987 2109 2987 2109 2987 1610 2987
% 
\special{pn 8}%
\special{pa 1610 2987}%
\special{pa 2109 2987}%
\special{fp}%
\special{sh 1}%
\special{pa 2109 2987}%
\special{pa 2042 2967}%
\special{pa 2056 2987}%
\special{pa 2042 3007}%
\special{pa 2109 2987}%
\special{fp}%
\special{pa 2109 2987}%
\special{pa 1610 2987}%
\special{fp}%
\special{sh 1}%
\special{pa 1610 2987}%
\special{pa 1677 3007}%
\special{pa 1663 2987}%
\special{pa 1677 2967}%
\special{pa 1610 2987}%
\special{fp}%
% VECTOR 2 0 3 0 Black Black  
% 4 2109 2987 2941 2987 2941 2987 2109 2987
% 
\special{pn 8}%
\special{pa 2109 2987}%
\special{pa 2941 2987}%
\special{fp}%
\special{sh 1}%
\special{pa 2941 2987}%
\special{pa 2874 2967}%
\special{pa 2888 2987}%
\special{pa 2874 3007}%
\special{pa 2941 2987}%
\special{fp}%
\special{pa 2941 2987}%
\special{pa 2109 2987}%
\special{fp}%
\special{sh 1}%
\special{pa 2109 2987}%
\special{pa 2176 3007}%
\special{pa 2162 2987}%
\special{pa 2176 2967}%
\special{pa 2109 2987}%
\special{fp}%
% VECTOR 2 0 3 0 Black Black  
% 4 4528 2986 5360 2986 5360 2986 4528 2986
% 
\special{pn 8}%
\special{pa 4528 2986}%
\special{pa 5360 2986}%
\special{fp}%
\special{sh 1}%
\special{pa 5360 2986}%
\special{pa 5293 2966}%
\special{pa 5307 2986}%
\special{pa 5293 3006}%
\special{pa 5360 2986}%
\special{fp}%
\special{pa 5360 2986}%
\special{pa 4528 2986}%
\special{fp}%
\special{sh 1}%
\special{pa 4528 2986}%
\special{pa 4595 3006}%
\special{pa 4581 2986}%
\special{pa 4595 2966}%
\special{pa 4528 2986}%
\special{fp}%
% VECTOR 2 2 3 0 Black Black  
% 2 2032 1798 2032 2214
% 
\special{pn 8}%
\special{pa 2032 1798}%
\special{pa 2032 2214}%
\special{dt 0.045}%
\special{sh 1}%
\special{pa 2032 2214}%
\special{pa 2052 2147}%
\special{pa 2032 2161}%
\special{pa 2012 2147}%
\special{pa 2032 2214}%
\special{fp}%
% VECTOR 2 2 3 0 Black Black  
% 2 1865 1633 1865 2049
% 
\special{pn 8}%
\special{pa 1865 1633}%
\special{pa 1865 2049}%
\special{dt 0.045}%
\special{sh 1}%
\special{pa 1865 2049}%
\special{pa 1885 1982}%
\special{pa 1865 1996}%
\special{pa 1845 1982}%
\special{pa 1865 2049}%
\special{fp}%
% STR 2 0 3 0 Black Black  
% 4 2280 977 2280 1060 2 0 0 0
% $\bar{u}'$
\put(22.8000,-10.6000){\makebox(0,0)[lb]{$\bar{u}'$}}%
% STR 2 0 3 0 Black Black  
% 4 4700 1008 4700 1090 2 0 0 0
% $u=v_1v_0\bar{u}'$
\put(47.0000,-10.9000){\makebox(0,0)[lb]{$u=v_1v_0\bar{u}'$}}%
% LINE 2 0 3 0 Black Black  
% 2 2110 900 2110 3160
% 
\special{pn 8}%
\special{pa 2110 900}%
\special{pa 2110 3160}%
\special{fp}%
% LINE 2 0 3 0 Black Black  
% 2 4530 900 4530 3160
% 
\special{pn 8}%
\special{pa 4530 900}%
\special{pa 4530 3160}%
\special{fp}%
% VECTOR 2 0 3 0 Black Black  
% 2 1280 900 1280 3160
% 
\special{pn 8}%
\special{pa 1280 900}%
\special{pa 1280 3160}%
\special{fp}%
\special{sh 1}%
\special{pa 1280 3160}%
\special{pa 1300 3093}%
\special{pa 1280 3107}%
\special{pa 1260 3093}%
\special{pa 1280 3160}%
\special{fp}%
% LINE 2 0 3 0 Black Black  
% 2 1242 1466 1326 1466
% 
\special{pn 8}%
\special{pa 1242 1466}%
\special{pa 1326 1466}%
\special{fp}%
% STR 2 0 3 0 Black Black  
% 4 1112 1433 1112 1516 2 0 0 0
% 7
\put(11.1200,-15.1600){\makebox(0,0)[lb]{7}}%
% LINE 2 0 3 0 Black Black  
% 2 1240 2630 1324 2630
% 
\special{pn 8}%
\special{pa 1240 2630}%
\special{pa 1324 2630}%
\special{fp}%
% STR 2 0 3 0 Black Black  
% 4 1040 2587 1040 2670 2 0 0 0
% 14
\put(10.4000,-26.7000){\makebox(0,0)[lb]{14}}%
% LINE 2 0 3 0 Black Black  
% 2 1242 1798 1326 1798
% 
\special{pn 8}%
\special{pa 1242 1798}%
\special{pa 1326 1798}%
\special{fp}%
% STR 2 0 3 0 Black Black  
% 4 590 1803 590 1885 2 0 0 0
% $h(1)=9$
\put(5.9000,-18.8500){\makebox(0,0)[lb]{$h(1)=9$}}%
% STR 2 0 3 0 Black Black  
% 4 3481 2183 3481 2266 2 0 0 0
% $\leadsto$
\put(34.8100,-22.6600){\makebox(0,0)[lb]{$\leadsto$}}%
% STR 2 0 3 0 Black Black  
% 4 6080 1448 6080 1530 2 0 0 0
% $1$
\put(60.8000,-15.3000){\makebox(0,0)[lb]{$1$}}%
% STR 2 0 3 0 Black Black  
% 4 5740 1258 5740 1340 2 0 0 0
% $0$
\put(57.4000,-13.4000){\makebox(0,0)[lb]{$0$}}%
% STR 2 0 3 0 Black Black  
% 4 6410 1598 6410 1680 2 0 0 0
% $2$
\put(64.1000,-16.8000){\makebox(0,0)[lb]{$2$}}%
% STR 2 0 3 0 Black Black  
% 4 5590 1358 5590 1440 2 0 0 0
% \rotatebox{90}{$\underbrace{\hspace{10pt}}$}
\put(55.9000,-14.4000){\makebox(0,0)[lb]{\rotatebox{90}{$\underbrace{\hspace{10pt}}$}}}%
% STR 2 0 3 0 Black Black  
% 4 5940 1698 5940 1780 2 0 0 0
% \rotatebox{90}{$\underbrace{\hspace{46pt}}$}
\put(59.4000,-17.8000){\makebox(0,0)[lb]{\rotatebox{90}{$\underbrace{\hspace{46pt}}$}}}%
% STR 2 0 3 0 Black Black  
% 4 6270 2028 6270 2110 2 0 0 0
% \rotatebox{90}{$\underbrace{\hspace{70pt}}$}
\put(62.7000,-21.1000){\makebox(0,0)[lb]{\rotatebox{90}{$\underbrace{\hspace{70pt}}$}}}%
% LINE 2 2 3 0 Black Black  
% 2 4010 1800 5541 1800
% 
\special{pn 8}%
\special{pa 4010 1800}%
\special{pa 5541 1800}%
\special{dt 0.045}%
% LINE 2 2 3 0 Black Black  
% 2 4010 2130 5541 2130
% 
\special{pn 8}%
\special{pa 4010 2130}%
\special{pa 5541 2130}%
\special{dt 0.045}%
% LINE 2 2 3 0 Black Black  
% 2 4010 1460 5541 1460
% 
\special{pn 8}%
\special{pa 4010 1460}%
\special{pa 5541 1460}%
\special{dt 0.045}%
% STR 2 0 3 0 Black Black  
% 4 3410 1977 3410 2060 2 0 0 0
% $v_1 v_0$
\put(34.1000,-20.6000){\makebox(0,0)[lb]{$v_1 v_0$}}%
% LINE 2 5 3 0 Black Black  
% 2 1610 1630 2941 1630
% 
\special{pn 8}%
\special{pa 1610 1630}%
\special{pa 2941 1630}%
\special{ip}%
% LINE 2 5 3 0 Black Black  
% 2 1610 1800 2941 1800
% 
\special{pn 8}%
\special{pa 1610 1800}%
\special{pa 2941 1800}%
\special{ip}%
% LINE 2 2 3 0 Black Black  
% 2 1610 1470 2941 1470
% 
\special{pn 8}%
\special{pa 1610 1470}%
\special{pa 2941 1470}%
\special{dt 0.045}%
% LINE 2 2 3 0 Black Black  
% 2 1610 2050 2941 2050
% 
\special{pn 8}%
\special{pa 1610 2050}%
\special{pa 2941 2050}%
\special{dt 0.045}%
% LINE 2 2 3 0 Black Black  
% 2 1610 2230 2941 2230
% 
\special{pn 8}%
\special{pa 1610 2230}%
\special{pa 2941 2230}%
\special{dt 0.045}%
% LINE 2 0 3 0 Black Black  
% 2 1240 1630 1324 1630
% 
\special{pn 8}%
\special{pa 1240 1630}%
\special{pa 1324 1630}%
\special{fp}%
% STR 2 0 3 0 Black Black  
% 4 1112 1603 1112 1686 2 0 0 0
% 8
\put(11.1200,-16.8600){\makebox(0,0)[lb]{8}}%
% STR 2 0 3 0 Black Black  
% 4 1760 3107 1760 3190 2 0 0 0
% $J_{9}^-$
\put(17.6000,-31.9000){\makebox(0,0)[lb]{$J_{9}^-$}}%
% STR 2 0 3 0 Black Black  
% 4 2425 3107 2425 3190 2 0 0 0
% $J_{9}^+$
\put(24.2500,-31.9000){\makebox(0,0)[lb]{$J_{9}^+$}}%
% STR 2 0 3 0 Black Black  
% 4 4180 3107 4180 3190 2 0 0 0
% $J_{9}^-$
\put(41.8000,-31.9000){\makebox(0,0)[lb]{$J_{9}^-$}}%
% STR 2 0 3 0 Black Black  
% 4 4845 3107 4845 3190 2 0 0 0
% $J_{9}^+$
\put(48.4500,-31.9000){\makebox(0,0)[lb]{$J_{9}^+$}}%
\end{picture}}%
\caption{The positions of $1$'s for $\bar{u}' $ and $u$.}
\label{pic:example_modification}
\end{figure}
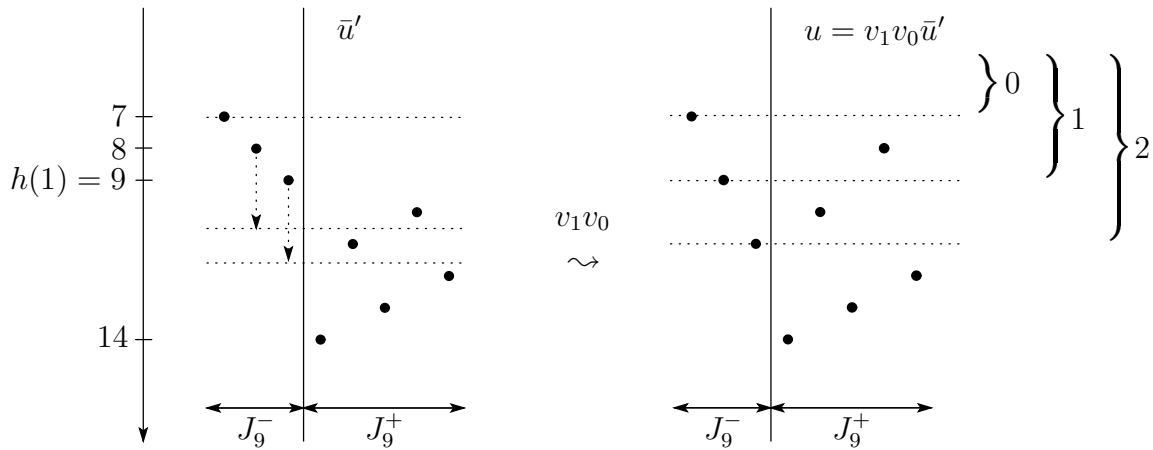

Now $u\w{h}$ is given by
\begin{align*}
\hspace{2pt}
u \w{h} = \ 
 \hspace{1pt}11 \hspace{6pt} 14 \hspace{9pt} 9 \hspace{8pt} 10 \hspace{6pt} 13 \hspace{9pt} 7 \hspace{8pt} \fbox{\hspace{-3pt} 8 \hspace{0pt} 15 \hspace{-2pt} 17 \hspace{-2pt} 20\hspace{2pt} 16\hspace{2pt} 19\hspace{2pt} 12 \hspace{-2pt} 18 \hspace{-5pt}} \hspace{7pt} 6 \hspace{11pt} 5 \hspace{11pt} 4 \hspace{9pt} \fbox{\hspace{0pt}3 \hspace{4pt} 2 \hspace{4pt} 1 \hspace{-4pt}},
\end{align*}
and one can verify that $u$ satisfies conditions (i)-(iv).

\end{document}